\RequirePackage{fix-cm}
\documentclass{svjour3}                     
\smartqed  
\usepackage[title]{appendix}
\usepackage{mathrsfs}
\usepackage{mathtools}
\usepackage{graphicx}
\usepackage{epstopdf}
\usepackage{stmaryrd}
 \usepackage{subfigure}
\usepackage{float}
\usepackage{amsmath}
    \usepackage{bm}
\usepackage{amsfonts,amssymb}
  \usepackage{dsfont}
\usepackage{amsfonts}
\usepackage{mathrsfs}
  \usepackage{pifont}
   \usepackage{amsmath}
 \usepackage{multirow}
 \usepackage{color}

 \numberwithin{equation}{section}

\begin{document}


\title{A new parameter free partially penalized immersed finite element and the optimal convergence analysis\thanks{
H. Ji is partially supported by the National Natural Science Foundation of China (Grants Nos. 11701291, 12101327 and 11801281) and the Natural Science Foundation of Jiangsu Province (Grant No. BK20200848); F. Wang is partially supported by the National Natural Science Foundation of China (Grant Nos. 12071227 and 11871281) and the Natural Science Foundation of the Jiangsu Higher Education Institutions of China (Grant No. 20KJA110001);  J. Chen is partially supported by the National Natural Science Foundation of China (Grant Nos. 11871281, 11731007 and 12071227); Z. Li is partially supported by  partially supported by a Simon grant (No. 633724).}
}

\titlerunning{Analysis of immersed finite element method}        

\author{Haifeng Ji         \and
        Feng Wang  \and
        Jinru Chen \and
        Zhilin Li
}


\institute{Haifeng Ji \at
              School of Science, Nanjing University of Posts and Telecommunications, Nanjing 210023, China\\
              \email{hfji@njupt.edu.cn}           
           \and
           Feng Wang \at
             Jiangsu Key Laboratory for NSLSCS, School of Mathematical Sciences, Nanjing Normal University, Nanjing 210023, China\\
              \email{fwang@njnu.edu.cn}
             \and
             Jinru Chen \at
             Jiangsu Key Laboratory for NSLSCS, School of Mathematical Sciences, Nanjing Normal University, Nanjing 210023, China \\
     School of Mathematics and Information Technology, Jiangsu Second Normal University, Nanjing 211200, China\\
              \email{jrchen@njnu.edu.cn}
             \and
             Zhilin Li \at
             Department of Mathematics, North Carolina State University, Raleigh, NC 27695, USA\\
              \email{zhilin@math.ncsu.edu}
              }

\date{Received: date / Accepted: date}

\maketitle
\begin{abstract}
This paper presents a new  parameter free partially penalized immersed finite element method and convergence analysis for solving second order elliptic interface problems. A lifting operator is introduced on interface edges to ensure the coercivity of the method without requiring an ad-hoc stabilization parameter.
The optimal approximation capabilities of the immersed finite element space is proved via a novel new approach that is much simpler than that in the literature.   A new trace inequality which is necessary to prove the optimal convergence of   immersed finite element methods is established on interface elements. Optimal error estimates are derived rigorously with the constant independent of the interface location relative to the mesh.
The new method and  analysis have also been extended to variable coefficients and three-dimensional problems. Numerical examples are also provided to confirm the theoretical analysis and efficiency of the new method.

\end{abstract}

\section{Introduction}

In this paper we consider immersed finite element (IFE) methods for solving the following second-order elliptic interface problem
\begin{align}
-\nabla\cdot(\beta(x)\nabla u(x))&=f(x)  \qquad\mbox{in } \Omega\backslash\Gamma,\label{p1.1}\\
[u]_{\Gamma}(x)&=0~~~~ \qquad\mbox{on } \Gamma,\label{p1.2}\\
[\beta\nabla u\cdot \textbf{n}]_{\Gamma}(x)&=0~~~~ \qquad\mbox{on } \Gamma,\label{p1.3}\\
u(x)&=0 ~~~~\qquad\mbox{on } \partial\Omega,\label{p1.4}
\end{align}
where $f\in L^2(\Omega)$, $\Omega\subset\mathbb{R}^d$, $d=2,3$ is a convex polygonal/polyhedral domain and $\Gamma$ is a compact curve/surface without boundary embedded in $\Omega$.  The interface $\Gamma$ divides $\Omega$ into two disjoint sub-domains $\Omega^+$ and $\Omega^-$. Without loss of generality, we assume that $\Omega^-$ lies strictly inside $\Omega$, see Figure~\ref{interfacepb} for an illustration.
The jump conditions on the interface $\Gamma$ are defined as
\begin{align}
[u]_{\Gamma}(x)&:=u^+(x)-u^-(x),\label{p1.6}\\
[\beta\nabla u\cdot\textbf{n}]_\Gamma(x)&:= \beta^+(x)\nabla u^+(x)\cdot\textbf{n}(x)-\beta^-(x)\nabla u^-(x)\cdot\textbf{n}(x),\label{p1.7}
\end{align}
where $u^\pm=u|_{\Omega^\pm}$ and $\textbf{n}(x)$ is the unit normal vector of the interface $\Gamma$ at $x\in\Gamma$ pointing toward $\Omega^+$.

We first consider the problem in two dimensions $(d=2)$ with  a piecewise constant coefficient, i.e.,
\begin{align}\label{p1.5}
\beta(x)=\beta^+>0 \mbox{ if } x\in\Omega^+ \mbox{ and } \beta(x)=\beta^->0 \mbox{ if } x\in\Omega^-.
\end{align}
The extensions to variable coefficients and 3D problems are presented in Section~\ref{sec_var} and Section~\ref{sec_3D}, respectively.

\begin{figure} [htbp]
\centering
\includegraphics[height=4cm,clip]{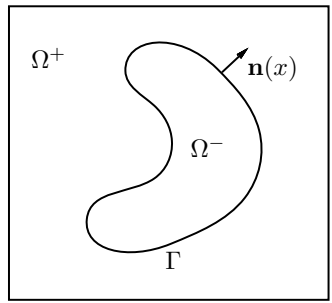}
 \centering \caption{A  diagram of the geometries of an interface problem.}\label{interfacepb} 
\end{figure}

 It is well known that the optimal convergence can be achieved by standard finite element methods if interface-fitted meshes are used, see for example \cite{bramble1996finite,chen1998finite,0xu}.
However, for a moving interface, it may be time consuming to obtain an interface-fitted mesh at different time levels.
IFE methods are designed to solve interface problems using unfitted meshes that are not necessarily aligned with
interfaces. An unfitted mesh is generated independent of the interface and allows the interface cut elements.
IFE methods are often coupled with structured meshes and can utilize
fast Poisson solvers  and other efficient software packages. Peskin's immersed boundary method \cite{Peskin1977Numerical} is one of successful examples using  unfitted meshes. We refer the readers to \cite{li2006immersed} for a brief review of various unfitted mesh methods for interface problems.


Traditional finite element methods using unfitted meshes only achieve sub-optimal convergence (i.e., $O(h^{1/2})$ in the $H^1$ norm and $O(h)$ in the $L^2$ norm) no matter how high degree of the polynomial is used, see  \cite{babuvska1970finite,2014A}.
The sub-optimal convergence is due to the interface condition~(\ref{p1.3}) that leads to discontinuous normal derivative across the interface in general  if $\beta^+\not=\beta^-$.
Thus, the regularity of the  solution is low on interface elements  which the interface cuts and smooth polynomials cannot approximate the  solution well enough on these elements. For finite element methods, roughly speaking, there are two approaches to recover the optimal convergence. The first one is to enrich the standard finite element space by augmenting extra degrees of freedom on interface elements and if necessary add some integral terms into the variational form to weakly enforce interface conditions, for example, the extended finite element method \cite{fries2010extended}, the unfitted Nitsche's method \cite{hansbo2002unfitted,2012An,2016High,2018Robust} and the enriched finite element method \cite{Wang2018A}. The other approach is to modify the traditional finite element space on interface elements according to interface conditions to achieve the optimal approximation, while keeping the degrees of freedom and the structure unchanged, for example, the multiscale finite method \cite{multi_CHU2010} and the IFE method \cite{li1998immersed,Li2003new,li2004immersed,he2012convergence,He2008,Guo2018CMA,Guo2019Improved} that we utilize in this paper.

The key idea of the original IFE in \cite{Li2003new} is to use a piecewise linear function as the basis function over an interface element  so that the continuity condition can be satisfied both for the function and the flux along the line segment approximating the interface.  It has been shown that  the modified finite element space (called the IFE space) exists  and has the optimal approximation capability. However,  one trade-off is that  the finite element basis  may be discontinuous across interface edges on interface elements.
In \cite{hou2004removing,hou2005numerical}, a Petrov-Galerkin finite element method  is proposed, in which the IFE space is used as the trial function space while the standard conforming linear finite element space is chosen as the test function space. However, the coefficient matrix of the resulting linear system of equations  is non-symmetric and the convergence proof is still an open challenge except for the one dimensional case \cite{jiInf18}.
Another approach is to add the contributions from the line integrals due to the discontinuities in the basis functions in deriving the weak form as in the  symmetric and consistent IFE method in \cite{ji2014sym}.  However, the coercivity was only verified by numerical examples in \cite{ji2014sym}. The partially penalized immersed finite element (PPIFE) method developed in \cite{taolin2015siam} includes some terms only on interface edges to penalize the discontinuity, which can guarantee the optimal convergence if the penalty parameter is larger enough and the solution is in the piecewise $H^3$ space.
The error estimate in the $L^2$ norm cannot be obtained by the standard duality argument due to the requirement of the higher regularity.
The PPIFE method is then extended to the second-order elliptic interface problem with  non-homogeneous jump conditions in \cite{jiESAJAM2018} under higher regularity assumptions, and other types of interface problems in \cite{Linpar2015,huang2021}.

 In this paper, we develop and analyze a parameter free PPIFE method based on a special designed lifting operator defined locally on interface elements. The new method avoids the limitation of the PPIFE in \cite{taolin2015siam} in choosing the stabilization parameter that may depends on $\beta^\pm$. 
 We show that the lifting operator can be expressed explicitly, and thus, is  easy to be computed.
 The idea of using liftings comes from discontinuous Galerkin methods. We refer the readers to the book \cite{di2011mathematical}, particularly Chapter 4.3, for the definition of liftings for discontinuous Galerkin methods.
However, different from the discontinuous Galerkin methods and the  original PPIFE method \cite{taolin2015siam}, we show that the penalty term involving the jump of IFE functions on interface edges does not need to be included since the IFE  functions are continuous at nodal points.
 %
%
We prove optimal error estimates of the new method in the $H^1$ and $L^2$ norms rigorously with usual piecewise $H^2$ regularity assumptions.  
In addition, we also show that our method and the analysis can be  extended to variable coefficients and three-dimensional problems.

 There are two major contributions in the theoretical analysis for IFE methods in this paper. First, we present a novel and simple way to prove the optimal interpolation error estimates in the $H^1$ and $L^2$ norms for the linear IFE space originally developed in \cite{Li2003new}. The first proof of this result was presented in \cite{li2004immersed} based on
the multipoint Taylor expansion and the piecewise $C^2$ assumption. Thus, the proof is long and tedious with the stronger than necessary regularity assumption.
 Recently, Guo and Lin \cite{GuoIMA2019}  presented a unified  multipoint Taylor expansion procedure for proving the optimal approximation capability of a group of IFE spaces
where the finite element function is a piecewise polynomial on subelements formed by the interface itself instead of its line approximation.
 For high-contrast interface problems, Guzm\'an et al. \cite{GuzmanJSC2017} proposed a finite element method where the shape functions on interface elements are also defined with subelements formed by the interface.
 %
 Note that the finite element functions in \cite{GuzmanJSC2017} are discontinuous on boundaries of interface elements. Therefore the finite element space has more degrees of freedom than that of IFE methods. Higher order methods are developed and analyzed  in \cite{201GUOSIAM}.  The key of the analysis technique developed in \cite{GuzmanJSC2017}  is to use a patch around the interface element to deal with possible small triangles cut by the interface. 
 Using the patch idea, Guo and Lin \cite{Guojcp2020} proposed a framework to analyze IFEs and proved the optimal approximation capability of an IFE space in three dimensions.
 In this paper, we have developed another analysis technique for the interpolation error for IFE spaces without using patches. The core ingredient of the analysis is to introduce some auxiliary functions on interface elements and then to carry out the analysis (see Lemma~\ref{lem_jiest} in Section~\ref{sec_anal}). The idea using the auxiliary functions is inspired from early  works on the augmented IFE method \cite{Ji2020JCP}.
 %

Note that some IFE methods use the exact interface information on interface triangles.  The downside of this approach is that the IFE shape functions often are discontinuous for curved interfaces. Other IFE methods use the line segments to approximate the interface so that the IFE shape functions are continuous in the interior of interface triangles. The first approach (the exact  interface) is advantageous for high order IFE methods  \cite{adjerid2018higher,adjerid2017high} and three-dimensional problems \cite{Guojcp2020}.
For two-dimensional problems, we use the second approach (line segments) and  present a rigorous 
proof of the fact that  the linear approximation of the interface is enough to ensure the optimal convergence for linear IFE methods. 
 %
 We note that the existing error analysis on the mismatch of the actual interface and the approximated interface for approximation capabilities of IFE spaces is based on  the argument in \cite{chen1998finite}.  Thus,  there will be a factor $|\log h|$ in those interpolation error estimates. In this paper, we use a technique from \cite{James1994A} so that we can actually  remove the $|\log h|$ factor in the optimal interpolation error estimates.

The second major contribution of this paper is  a new trace inequality (see Lemma~\ref{lema_trace} in Section~\ref{subsec_tra}) on interface elements, which is key in   proving the optimal convergence of IFE methods under a standard piecewise $H^2$ regularity assumption. The new developed trace inequality can be applied to improve the error analysis of the PPIFE method developed in \cite{taolin2015siam}. The proof of the new  trace inequality is based on the decomposition of functions along the normal and tangent directions of interfaces, and the fact that the IFE shape function and its flux are continuous across the approximated interface simultaneously.

The rest of the paper is organized as follows. In Section~\ref{sec_2}, we introduce notations,  and the linear IFE space. In Section~\ref{sec_3}, we define the local lifting operator and explain the new parameter free PPIFE method. The main theoretical results of this paper are presented in Section~\ref{sec_analysis}, where we give a new proof of the optimal interpolation error estimates for the linear IFE space; establish a new trace inequality for broken spaces;  and prove the optimal convergence of the new developed parameter free PPIFE method in the $H^1$ and $L^2$ norms under the standard piecewise $H^2$ regularity assumption. We extend the method and  analysis to variable coefficients and three dimensions in Section~\ref{sec_var} and Section~\ref{sec_3D}. Section~\ref{sec_num} presents some numerical examples to confirm the theoretical analysis.  We conclude in the last section.

\section{Notations and the IFE space}\label{sec_2}

Throughout the paper we adopt the standard notation $W^{k}_p(\Lambda)$ for Sobolev spaces on a domain $\Lambda$ with the norm $\|\cdot\|_{W^{k}_p(\Lambda)}$ and the seminorm $|\cdot|_{W^{k}_p(\Lambda)}$. Specially, we denote $W^{k}_2(\Lambda)$ by $H^{k}(\Lambda)$ with the norm $\|\cdot\|_{H^{k}(\Lambda)}$ and the semi-norm $|\cdot|_{H^{k}(\Lambda)}$. As usual $H_0^1(\Lambda)=\{v\in H^1(\Lambda) : v=0 \mbox{ on }\partial \Lambda\}$.
Furthermore, for a domain $\Lambda$, we define
$$\Lambda^+:=\Lambda\cap \Omega^+,\qquad \Lambda^-:=\Lambda\cap \Omega^-,$$
 and  a subspace of $H^1(\Lambda)$ by
\begin{equation}\label{def_H2}
\widetilde{H}^2(\Lambda):=\{v\in L^2(\Lambda) : v|_{\Lambda^\pm}\in H^2(\Lambda^\pm),~ [v]_{\Gamma\cap\Lambda}=0,~[\beta\nabla v\cdot \textbf{n}]_{\Gamma\cap\Lambda}=0\}
\end{equation}
equipped with the norm and semi-norm
$$
\|\cdot\|^2_{H^2(\Lambda^+\cup\Lambda^-)}=\|\cdot\|^2_{H^2(\Lambda^+)}+\|\cdot\|^2_{H^2(\Lambda^-)}, \quad|\cdot|^2_{H^2(\Lambda^+\cup\Lambda^-)}=|\cdot|^2_{H^2(\Lambda^+)}+|\cdot|^2_{H^2(\Lambda^-)}.
$$
%
We have the following regularity theorem for the interface problem, see \cite{huang2002some} and \cite{multi_CHU2010}.
\begin{theorem}\label{theo_reg}
If $\Omega$ is convex and polygonal, the interface $\Gamma$ is $C^2$,  and $f\in L^2(\Omega)$, then the interface problem (\ref{p1.1})-(\ref{p1.7}) has a unique
solution $u\in \widetilde{H}^2(\Omega)$  satisfying 
$$
\|u\|_{{H}^2(\Omega^+\cup\Omega^-)}\leq C\|f\|_{L^2(\Omega)},
$$
where $C$ is a  constant depending only on $\Omega$, $\Gamma$ and $\beta$.
\end{theorem}

Let $ \{\mathcal{T}_h\}_{h>0}$ be a family of triangulations of $\Omega$ such that no vertex of any element lies in the interior of an edge of another element. We use $h_T$ to denote the diameter of $T\in\mathcal{T}_h$ and define the mesh-size of the triangulation by $h=\max_{T\in\mathcal{T}_h}h_T$. We assume that $\mathcal{T}_h$ is  quasi-uniform, i.e., for every $T$, there exist positive constants $\kappa_1$ and $\kappa_2$ such that  $\kappa_1h\leq h_T\leq \kappa_2\rho_T$ where $\rho_T$ is the diameter of the largest circle inscribed in $T$. Let $\mathcal{E}_h$ be the set of edges and $\mathcal{N}_h$  be the set of vertices of the triangulation. We adopt the convention that elements $T\in\mathcal{T}_h$ and edges $e\in\mathcal{E}_h$ are open sets.  The sets of interface elements and interface edges are defined  as
\begin{equation*}
\mathcal{T}_h^\Gamma :=\{T\in\mathcal{T}_h :  T\cap \Gamma\not = \emptyset\} \quad\mbox{ and }\quad \mathcal{E}_h^\Gamma:=\{e\in \mathcal{E}_h : e \cap \Gamma\not = \emptyset\}.
\end{equation*}
The set of non-interface elements is defined by $\mathcal{T}^{non}_h=\mathcal{T}_h\backslash\mathcal{T}_h^{\Gamma}$.

\textbf{Assumption A.} The interface $\Gamma$ does not intersect the boundary of any interface element at more than two points. The interface $\Gamma$ does not intersect  $\overline{e}$ for any $e\in\mathcal{E}_h$  at more than one point.

We can always refine the mesh near the interface with large curvature until Assumption A is satisfied. As a common practice, we approximate the interface $\Gamma$ by $\Gamma_h$ that is composed of all the line segments connecting the intersection points of boundaries of interface elements $T\in\mathcal{T}^\Gamma_h$ and the interface $\Gamma$. The approximated $\Gamma_h$ divides $\Omega$ into two disjoint sub-domains $\Omega^+_h$ and $\Omega^-_h$. For convenience we approximate the coefficient $\beta(x)$  as
\begin{equation}\label{def_bth}
\beta_h(x)=\beta^+ \mbox{ if } x\in\Omega^+_h \quad\mbox{ and }\quad \beta_h(x)=\beta^- \mbox{ if } x\in\Omega^-_h.
\end{equation}
We will  show that the approximation of the interface by $\Gamma_h$ does not affect second order convergence when the interface $\Gamma$ is in $C^2$. Let $\textbf{n}_h(x)$ be the unit normal vector of $\Gamma_h$ pointing toward $\Omega^+_h$. The unit tangent vector $\textbf{t}_h(x)$ is obtained by a $90^\circ$ clockwise rotation of $\textbf{n}_h(x)$.  We note that $\textbf{n}_h(x)$ and $\textbf{t}_h(x)$ are viewed as  piecewise constant vector-valued functions defined on all interface elements.

%

\begin{figure} [htbp]
\centering
\includegraphics[height=5cm,clip]{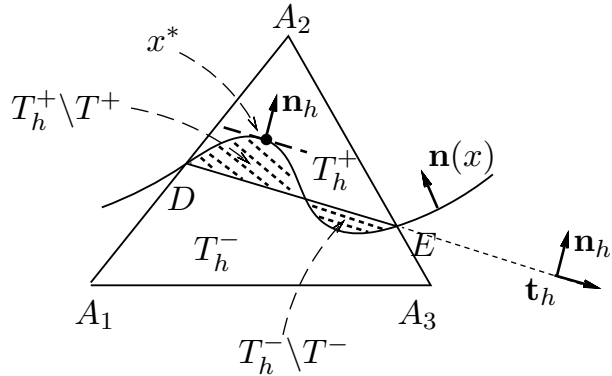}
 \caption{An interface element where the interface is approximated by the line segment.}\label{interface_ele} 
\end{figure}

\textbf{Linear IFE shape function space}. For an interface element $T\in\mathcal{T}_h^\Gamma$, we denote the intersection points of $\Gamma$ and $\partial T$ by $D$ and $E$.
%
The straight line $DE$ divides $T$ into $T_h^+= T\cap \Omega^+_h$ and $T_h^-=T\cap \Omega^-_h$, see Figure~\ref{interface_ele}. The linear IFE shape function on an interface element $T\in\mathcal{T}_h^\Gamma$ is defined as
\begin{equation}\label{ife_modi0}
\phi(x):=\left\{
\begin{aligned}
\phi^+=a^++b^+x_1+c^+x_2,\quad x=(x_1,x_2)\in T_h^+,\\
\phi^-=a^-+b^-x_1+c^-x_2,\quad x=(x_1,x_2)\in T_h^-,
\end{aligned}
\right.
\end{equation}
in which the coefficients are chosen so that  the following  conditions are satisfied
\begin{equation}\label{ifem_modi}
\phi^+(D)=\phi^-(D),~~ \phi^+(E)=\phi^-(E),~~\beta^+\nabla\phi^+\cdot \textbf{n}_{h}=\beta^-\nabla\phi^-\cdot \textbf{n}_{h}.
\end{equation}
The linear IFE shape function space  $S_h(T)$ is defined as  the set of functions in (\ref{ife_modi0}).
It is obvious that IFE shape function $\phi\in S_h(T)$ and its flux are continuous across $\Gamma_h\cap T$ simultaneously, i.e.,
\begin{equation}\label{conti_linear}
[\phi]_{\Gamma_{h}\cap T}=0 ~~\mbox{ and }~~  [\beta_h\nabla\phi\cdot \textbf{n}_{h}]_{\Gamma_{h}\cap T}=0\quad \mbox{ on }\Gamma_h\cap T.
\end{equation}
\begin{lemma}\label{lem_unique}
Let $A_i$, $i=1,2,3$ be vertices of an interface element $T\in\mathcal{T}_h^\Gamma$ and $\alpha_{max}$ be the maximum angle of the interface element $T$. If  $\alpha_{max}\leq \pi/2$, then the function $\phi\in S_h(T)$ defined in (\ref{ife_modi0})-(\ref{ifem_modi}) is uniquely determined by $\phi(A_i)$, $i=1,2,3$ regardless of the interface location.
\end{lemma}
\begin{proof}:
See  Appendix \ref{pro_lem_unique}. \qed
\end{proof}

\textbf{Linear IFE}. On an interface element $T\in\mathcal{T}_h^\Gamma$, we define the immersed finite element as $(T, S_h(T), \Pi_T)$, where
\begin{equation*}
\Pi_T=\{ N_1,N_2,N_3 \},~ N_i(\phi)=\phi(A_i), ~i=1,2,3.
\end{equation*}

On a non-interface element $T\in \mathcal{T}^{non}_h$, we denote the set of linear functions by $\mathbb{P}_1(T)$. Then, the IFE space $V_h^{{\rm IFE}}$ can be defined as the set of all functions satisfying
\begin{equation*}
\left\{
\begin{aligned}
&\phi|_T \in S_h(T), ~\forall  T\in\mathcal{T}_h^\Gamma,\\
&\phi|_T \in \mathbb{P}_1(T), ~\forall  T\in\mathcal{T}_h^{non},\\
&\phi \mbox{ is continuous at every vertices } x_i\in\mathcal{N}_h.
\end{aligned}
\right.
\end{equation*}
The IFE space $V_h^{{\rm IFE}}$ is a modification to the standard linear conforming finite element space to recover the optimal approximation capability. If $\beta^+=\beta^-$, the IFE space $V_h^{{\rm IFE}}$ becomes the standard linear conforming finite element space.  We also need the following space for homogeneous boundary condition 
$$V_{h,0}^{{\rm IFE}}=\{v\in V_h^{{\rm IFE}} : v(x_b)=0,~\forall x_b\in\mathcal{N}_h  \mbox{ and } x_b\in \partial \Omega\}.$$

\section{The parameter free PPIFE method}\label{sec_3}
To present the new  method, we first need a local lifting operator. On each interface element $T\in\mathcal{T}_h^\Gamma$, we define the space
\begin{equation*}
W_h(T):=\{\nabla v_h :   \forall v_h\in S_h(T) \}.
\end{equation*}
Let $e\in\mathcal{E}_h^\Gamma$ be an interface edge shared by two interface elements $T_1$ and $T_2$ such that $\overline{e}=\overline{T_1}\cap \overline{T_2}$.  We define a space associated with the edge $e$ as
\begin{equation}\label{lift2}
W_e:=\{w_h\in (L^2(\Omega))^2:~ w_h|_{T_1}\in W_h(T_1),~ w_h|_{T_2}\in W_h(T_2),~ w_h|_{\Omega\backslash(T_1\cup T_2)}=0\}.
\end{equation}
Given a scalar or vector function, the jump and average across the edge $e$ are denoted by
\begin{equation*}
[v]_e\textbf{n}_e:=(v|_{T_1}-v|_{T_2})\textbf{n}_e,~~\{v\}_e:=\frac{1}{2}(v|_{T_1}+v|_{T_2}),
\end{equation*}
where $\textbf{n}_e$ is the unit normal of $e$ pointing from $T_1$ to $T_2$.

We introduce a local  {\em  lifting operator}  $r_e : L^2(e)\rightarrow W_e$ for each $e\in\mathcal{E}_h^\Gamma$, which is defined as a functional $r_e(\varphi)\in W_e$ such that  for all $\varphi \in L^2(e)$,
\begin{equation}\label{def_lift}
\int_\Omega \beta_h(x) r_e(\varphi)\cdot w_hdx=\int_e\{\beta_h w_h\cdot \textbf{n}_e\}_e\, \varphi \,ds,\qquad \forall w_h\in W_e,
\end{equation}
where $\textbf{n}_e$ is the unit normal of the edge $e$. Since $w_h|_{\Omega\backslash(T_1\cup T_2)}=0$ for all $w_h\in W_e$,  we know that $r_e$ is a local lifting operator. Choosing $w_h=\nabla v_h$ in (\ref{def_lift}) and using $r_e(\varphi)|_{\Omega\backslash(T_1\cup T_2)}=0$, we find  
\begin{equation}\label{coneect}
\sum_{i=1,2}\int_{T_i} \beta_h(x) r_e(\varphi)\cdot \nabla v_hdx=\int_e\{\beta_h \nabla v_h\cdot \textbf{n}_e\}_e\varphi ds,\qquad \forall v_h\in V_h^{{\rm IFE}}.
\end{equation}

\textbf{The parameter free PPIFE method}: find $u_h\in V_{h,0}^{{\rm IFE}}$ such that
\begin{equation}\label{ph1}
A_h(u_h,v_h):=a_h(u_h,v_h)+s_h(u_h,v_h)=\int_{\Omega}fv_hdx, \qquad\forall v_h\in V_{h,0}^{{\rm IFE}},
\end{equation}
where
\begin{equation}\label{ph2}
\begin{aligned}
a_h(u_h,v_h):=\sum_{T\in\mathcal{T}_h}\int_T\beta_h(x)\nabla u_h\cdot\nabla v_hdx-\sum_{e\in\mathcal{E}_h^\Gamma}\int_e\left(\{\beta_h\nabla u_h\cdot\textbf{n}_e\}_e[v_h]_e+\{\beta_h\nabla v_h\cdot\textbf{n}_e\}_e[u_h]_e\right)ds
\end{aligned}
\end{equation}
and
\begin{equation}\label{sh}
s_h(u_h,v_h):=4\sum_{e\in\mathcal{E}_h^\Gamma}\int_{\Omega}\beta_h(x)r_e([u_h]_e)\cdot r_e([v_h]_e)dx.
\end{equation}

\begin{remark}\ \\
(1) The second term of the bilinear form $a_h(\cdot,\cdot)$ and the stability term $s_h(\cdot,\cdot)$ are added to
offset the errors from the discontinuities of IFE functions across interface edges. Those terms are zero
if $\beta^+=\beta^-$ and  the method becomes the standard  linear conforming finite element method.\ \\
(2) In practical implementation, these additional terms are only evaluated on interface edges and interface elements. Thus the extra computational cost in computing those terms is not significant in general.\ \\
(3)The idea of using liftings comes from the discontinuous Galerkin methods (see Chapter 4.3 in the book \cite{di2011mathematical}). The idea has also be applied to the cut finite element methods \cite{lehrenfeld2016removing,2016High}. However,  different from the discontinuous Galerkin methods and the original PPIFE method \cite{taolin2015siam}, we do not need to include the $\sum_{e\in\mathcal{E}_h^\Gamma} h^{-1}\int_{e}[u_h]_e[v_h]_eds$ \ term in the bilinear form since the functions in the IFE space are continuous at vertices of the triangulation.
\end{remark}

\begin{remark}
The local lifting operator $r_e$ needed in the parameter free penalty term $s_h(u_h,v_h)$ is easy to be computed. Let $T_1$ and $T_2$  be  two interface elements sharing the edge $e$. Given a function $\varphi\in L^2(e)$, from the definition (\ref{def_lift}), we know that the support of $r_e(\varphi)$ is $T_1\cup T_2$ and $r_e(\varphi)$ has the following form
\begin{equation}\label{lift_imp1}
r_e(\varphi)|_{T_i}=
\left\{
\begin{aligned}
&c_i\textbf{t}_{i,h}+\beta^-d_i\textbf{n}_{i,h}\qquad \mbox{ in }  T_{i,h}^+,\\
&c_i\textbf{t}_{i,h}+\beta^+d_i\textbf{n}_{i,h}\qquad \mbox{ in }  T_{i,h}^-,
\end{aligned}\right. \qquad i=1,2,
\end{equation}
where $T_{i,h}^\pm = T_i\cap\Omega_h^\pm$, $\textbf{n}_{i,h}=\textbf{n}_h|_{T_i}$ and $\textbf{t}_{i,h}=\textbf{t}_h|_{T_i}$, $i=1,2$.
We show that the coefficients $c_1, d_1, c_2, d_2$ can be expressed explicitly.
Choosing basis functions of $W_e$ as the test function $w_h$ in (\ref{def_lift}), for example,
\begin{equation}\label{lift_imp2}
\omega_1(x)=
\left\{
\begin{aligned}
&\textbf{t}_{1,h}~~\mbox{ if } x\in T_1,\\
&0\qquad   \mbox{ otherwise, }
\end{aligned}\right.
\qquad
\omega_2(x)=
\left\{
\begin{aligned}
&\beta^-\textbf{n}_{1,h}~~~~\mbox{ if } x\in T_{1,h}^+,\\
&\beta^+\textbf{n}_{1,h}~~~~\mbox{ if } x\in T_{1,h}^-,\\
&0\qquad\qquad   \mbox{ otherwise, }
\end{aligned}\right.
\end{equation}
we  obtain
\begin{equation}\label{lift_coet1}
c_1=\frac{\textbf{t}_{1,h}\cdot \textbf{n}_e \int_e\beta_h \varphi ds}{2(\beta^+|T_{1,h}^+|+\beta^-|T_{1,h}^-|)}, \qquad d_1=\frac{\textbf{n}_{1,h}\cdot \textbf{n}_e\int_e\varphi ds}{2(\beta^-|T_{1,h}^+|+\beta^+|T_{1,h}^-|)}.
\end{equation}
Similarly, we have
\begin{equation}\label{lift_coet2}
c_2=\frac{\textbf{t}_{2,h}\cdot \textbf{n}_e \int_e \beta_h  \varphi ds}{2(\beta^+|T_{2,h}^+|+\beta^-|T_{2,h}^-|)}, \qquad d_2=\frac{\textbf{n}_{2,h}\cdot \textbf{n}_e \int_e \varphi ds}{2(\beta^-|T_{2,h}^+|+\beta^+|T_{2,h}^-|)}.
\end{equation}
\end{remark}

\section{The error analysis}\label{sec_analysis}

In the analysis, we use $C$ to denote a generic error constant that is independent of $h$ and the interface location relative to the mesh but may depend on the coefficients $\beta^\pm$. The independence of the interface location relative to the mesh means that the constant $C$ is independent of how small $T\cap\Omega^+$ or $T\cap\Omega^-$ might be.


Denote  $\mbox{dist}(x,\Gamma)$ as  the distance between a point $x$ and the interface $\Gamma$, and  $N(\Gamma,\delta)=\{x\in\mathbb{R}^2: \mbox{dist}(x,\Gamma)< \delta\}$  as the neighborhood of $\Gamma$ of thickness $\delta$. Define a signed distance function near the interface as
\begin{equation*}
\rho(x)=\left\{
\begin{aligned}
&\mbox{dist}(x,\Gamma)\qquad&&\mbox{if } x\in\Omega^+\cap N(\Gamma,\delta_0)\\
&0&&\mbox{if } x\in\Gamma\\
&-\mbox{dist}(x,\Gamma)&&\mbox{if }  x\in\Omega^-\cap N(\Gamma,\delta_0).
\end{aligned}\right.
\end{equation*}
It is  known that there exists a constant $\delta_0>0$ such that the signed distance function $\rho(x)$ is well-defined in $N(\Gamma,\delta_0)$ and $\rho(x)\in C^2(N(\Gamma,\delta_0))$ since we assume that $\Gamma\in C^2$ (see \cite{foote1984regularity}). Now the unit normal and tangent vectors of the interface can be evaluated as 
$\textbf{n}(x)=\nabla \rho$ and $\textbf{t}(x) =\left(\frac{\partial\rho}{\partial x_2},-\frac{\partial\rho}{\partial x_1}\right)^T$, and 
these functions $\textbf{n}(x)$ and $\textbf{t}(x)$ are defined in $N(\Gamma,\delta_0)$.

\textbf{Assumption B.}  We assume that $h<\delta_0$ so that  $T\subset N(\Gamma,\delta_0)$ for all interface elements $T\in\mathcal{T}_h^\Gamma$.

Since $\rho(x)\in C^2(N(\Gamma,\delta_0))$, we have
\begin{equation}\label{nt_smooth}
\textbf{n}(x)\in \left(C^1(\overline{T})\right)^2\quad \mbox{ and }\quad\textbf{t}(x)\in \left(C^1(\overline{T})\right)^2,\quad \forall T\in\mathcal{T}_h^\Gamma.
\end{equation}
For any interface element $T\in\mathcal{T}_h^\Gamma$, by  Rolle's Theorem, there exists at least one point $x^*\in\Gamma\cap T$, see Figure~\ref{interface_ele},  such that
\begin{equation}\label{rolle}
\textbf{n}(x^*)=\textbf{n}_h(x^*) \quad \mbox{ and }\quad \textbf{t}(x^*)=\textbf{t}_h(x^*).
\end{equation}
By using  Taylor's expansion  at $x^*$, we further have
\begin{equation}\label{error_nt}
\|\textbf{n}-\textbf{n}_h\|_{L^\infty(T)}\leq Ch\quad \mbox{ and }\quad\|\textbf{t}-\textbf{t}_h\|_{L^\infty(T)}\leq Ch,\quad \forall T\in\mathcal{T}_h^\Gamma.
\end{equation}

In the following lemma, we present a $\delta$-strip argument that will be used for the error estimate in the region near the interface (see the third inequality in Lemma 2.1 in \cite{Li2010Optimal}).
\begin{lemma}\label{strip}
Let $\delta$ be sufficiently small. Then it holds for any $v\in H^1(\Omega)$ that 
\begin{equation*}
\|v\|_{L^2(N(\Gamma,\delta))}\leq C\sqrt{\delta}\|v\|_{H^1(\Omega)}.
\end{equation*}
Furthermore, if $v|_{\Gamma}=0$, then there holds
\begin{equation*}
\|v\|_{L^2(N(\Gamma,\delta))}\leq C\delta \|\nabla v\|_{L^2(N(\Gamma,\delta))}.
\end{equation*}
\end{lemma}
%
We need the following well-known extension result (see \cite{Gilbargbook}).
\begin{lemma}\label{lem_ext}
Assume that $u^\pm\in H^2(\Omega^\pm)$. Then there exist extensions $\mathbb{E}^\pm u^\pm \in H^2(\Omega)$ such that
\begin{equation*}
(\mathbb{E}^\pm u^\pm)|_{\Omega^\pm}=u^\pm~\mbox{ and }~\|\mathbb{E}^\pm u\|_{H^2(\Omega)}\leq C\|u^\pm\|_{H^2(\Omega^\pm)}
\end{equation*}
with a constant $C>0$ depending only on $\Omega^\pm$.
\end{lemma}

Recalling $T^\pm=T\cap \Omega^\pm,~ T_h^\pm=T\cap\Omega_h^\pm$ for all $T\in \mathcal{T}_h^\Gamma$, we define
\begin{equation}\label{tri}
T^\triangle:=(T_h^+\backslash T^+)\cup(T_h^-\backslash T^-).
\end{equation}
We shall need the following estimate on the region $T^\triangle$ (see Lemma 2 in \cite{James1994A}).
\begin{lemma}\label{lem_h3}
Assume  that $w\in H^1(T)$ and $T\in\mathcal{T}_h^\Gamma$. Then there is a constant $C$, independent of $h$ and $w$, such that
\begin{equation*}
\|w\|_{L^2(T^\triangle)}^2\leq C(h^2\|w\|^2_{L^2(\Gamma\cap T)}+h^4\|\nabla w\|^2_{L^2(T^\triangle)}).
\end{equation*}
\end{lemma}

\subsection{Approximation properties of the linear IFE space}\label{sec_anal}

We introduce an {\em interpolation operator} $I_h^{{\rm IFE}}: C^0(\overline{\Omega})\rightarrow V_h^{{\rm IFE}}$ such that
\begin{equation*}
(I_h^{{\rm IFE}}v)(x_i)=v(x_i),\quad \forall x_i\in\mathcal{N}_h, \quad \forall v\in C^0(\overline{\Omega}).
\end{equation*}
Let $V_h$ be the standard linear conforming finite element space associated with $\mathcal{T}_h$. Define the corresponding nodal interpolation operator $I_h: C^0(\overline{\Omega})\rightarrow V_h$ such that
\begin{equation*}
(I_hv)(x_i)=v(x_i),\quad\forall x_i\in\mathcal{N}_h,\quad\forall v\in C^0(\overline{\Omega}).
\end{equation*}
%
To simplify the notation, for a function $v\in \widetilde{H}^2(\Omega)$, we define
\begin{equation*}
v^i(x):=v|_{\Omega^i},\qquad \forall x\in\Omega^i,\quad i=+,-,
\end{equation*}
and for a function $v_h\in S_h(T)$ on an interface element $T\in\mathcal{T}_h^\Gamma$, 
\begin{equation*}
v_h^i(x):=(\mathbb{E}^{poly}v_h|_{T_h^i})(x),\qquad \forall x\in T,\quad i=+,-.
\end{equation*}
Here $\mathbb{E}^{poly}:\mathbb{P}_1(T_h^i)\rightarrow\mathbb{P}_1(T)$ is a polynomial  extension operator such that   $(\mathbb{E}^{poly}w)|_{T_h^i}=w$ for all $w\in\mathbb{P}_1(T_h^i)$, $i=+, -$, where  $\mathbb{P}_1(\Lambda)$ denotes the set of linear functions defined on a domain $\Lambda$.
%
We also define an operator $\mathbb{E}_h$ such that, for all $v\in\widetilde{H}^2(\Omega)$,
\begin{equation}\label{epm}
(\mathbb{E}_{h}v)(x)=
\left\{\begin{aligned}
&\mathbb{E}^+v^+\qquad \mbox{ if }x\in\Omega_h^+,\\
&\mathbb{E}^-v^-\qquad \mbox{ if }x\in\Omega_h^-.\\
\end{aligned}
\right.
\end{equation}
To approximate the broken function $\mathbb{E}_hv$ with $v\in\widetilde{H}^2(\Omega)$, we introduce a new  interpolation operator $I_h^{BK}$ on interface elements such that
\begin{equation}\label{ehpm}
(I_h^{BK}v)|_{T_h^i}=I_h\mathbb{E}^i v^i, \quad i=+,-, \quad\forall T\in\mathcal{T}_h^\Gamma, \quad \forall v\in\widetilde{H}^2(\Omega).
\end{equation}
On an interface element $T\in\mathcal{T}_h^\Gamma$,  for a function $v\in \widetilde{H}^2(\Omega)$, we define 
\begin{equation}\label{shuangfang1}
\begin{aligned}
&[\![v]\!](x):=\mathbb{E}^+v^+(x)-\mathbb{E}^-v^-(x),\qquad&\forall x\in T,\\
&[\![\beta \nabla v\cdot \textbf{n}]\!](x):=\beta^+ \nabla (\mathbb{E}^+v^+)\cdot\textbf{n}(x)-\beta^- \nabla (\mathbb{E}^+v^-)\cdot\textbf{n}(x),\qquad&\forall x\in T,\\
&[\![ I_h^{BK}v]\!](x):=(I_h\mathbb{E}^+v^+)(x)-(I_h\mathbb{E}^-v^-)(x),\qquad&\forall x\in T,\\
&[\![\beta \nabla (I_h^{BK}v)\cdot \textbf{n}]\!](x):=\beta^+ \nabla I_h(\mathbb{E}^+v^+)\cdot\textbf{n}(x)-\beta^- \nabla I_h(\mathbb{E}^-v^-)\cdot\textbf{n}(x),\qquad&\forall x\in T,
\end{aligned}
\end{equation}
and for a function $v_h\in V_h^{{\rm IFE}}$,
\begin{equation}\label{shuangfang2}
\begin{aligned}
&[\![v_h]\!](x):=v_h^+(x)-v_h^-(x),\qquad&\forall x\in T,\\
&[\![\beta \nabla v_h\cdot \textbf{n}]\!](x):=\beta^+ \nabla v^+_h\cdot\textbf{n}(x)-\beta^- \nabla v^-_h\cdot\textbf{n}(x),\qquad&\forall x\in T.
\end{aligned}
\end{equation}
 Note that the difference between $[\![\cdot]\!](x)$ and $[\cdot]_\Gamma(x)$ is the range of $x$.

We introduce {\em auxiliary functions}  on each interface element $T\in\mathcal{T}_h^\Gamma$.
%
Recalling that $D$ and $E$ are intersection points of $\Gamma$ and $\partial T$, we define auxiliary functions $\Upsilon(x)$, $\Psi_D(x)$ and $\Psi_E(x)$ as
\begin{equation}\label{jum_ji1}
\Upsilon(x):=\left\{
\begin{aligned}
\Upsilon^+=a^++b^+x_1+c^+x_2,\quad x=(x_1,x_2)\in T_h^+,\\
\Upsilon^-=a^-+b^-x_1+c^-x_2,\quad x=(x_1,x_2)\in T_h^-,
\end{aligned}
\right.
\end{equation}
such that
\begin{equation}\label{jum_ji1_cons}
\begin{aligned}
&\Upsilon(A_j)=0,~ j=1,2,3, \\
&\Upsilon^+(D)=\Upsilon^-(D),~~ \Upsilon^+(E)=\Upsilon^-(E),~~\beta^+\nabla\Upsilon^+\cdot \textbf{n}_h-\beta^-\nabla\Upsilon^-\cdot \textbf{n}_h=1,
\end{aligned}
\end{equation}
and
\begin{equation}\label{def_psi_linear1}
\Psi_i(x):=\left\{
\begin{aligned}
\Psi_i^+=a^++b^+x_1+c^+x_2,\quad x=(x_1,x_2)\in T_h^+,\\
\Psi_i^-=a^-+b^-x_1+c^-x_2,\quad x=(x_1,x_2)\in T_h^-,
\end{aligned}
\right. \quad i=D,E,
\end{equation}
such that
\begin{equation}\label{def_psi_linear2}
\begin{aligned}
&\Psi_i(A_j)=0,~ j=1,2,3,~~~i=D,E,\\
&\beta^+\nabla\Psi_i^+\cdot \textbf{n}_h=\beta^-\nabla\Psi_i^-\cdot \textbf{n}_h,\quad i=D,E,\\
&\Psi_D^+(D)-\Psi_D^-(D)=1,~~ \Psi_D^+(E)-\Psi_D^-(E)=0,\\
&\Psi_E^+(D)-\Psi_E^-(D)=0,~~ \Psi_E^+(E)-\Psi_E^-(E)=1.
\end{aligned}
\end{equation}

\begin{remark}
The functions $\Upsilon$, $\Psi_D$ and $\Psi_E$ defined above  exist and are unique. The justification is that the coefficient matrix is the same as that for determining the IFE shape functions in the space $S_h(T)$ if we write a linear system for the unknown coefficients $a^\pm$,$b^\pm$ and $c^\pm$.
\end{remark}

\begin{lemma}\label{lema_fenjie}
On each interface element $T\in \mathcal{T}_h^\Gamma$, let $[\![ I_h^{BK}v]\!]$ and $[\![\beta \nabla (I_h^{BK}v)\cdot \textbf{n}]\!]$ be define in (\ref{shuangfang1}).
Under the condition of Lemma~\ref{lem_unique}, the following identity  holds
\begin{equation}\label{lemma4}
I_h^{BK}v-I_h^{{\rm IFE}}v=[\![ I_h^{BK}v]\!](D)\Psi_D(x)+[\![I_h^{BK}v]\!](E)\Psi_E(x)+[\![\beta \nabla (I_h^{BK}v)\cdot \textbf{n}]\!](x^*)\Upsilon(x). 
\end{equation}
\end{lemma}
\begin{proof}:
Let  $w_h:=I_h^{BK}v-I_h^{{\rm IFE}}v$. It is easy to verify that $w_h(A_i)=0, i=1,2,3$ and $w_h|_{T_h^\pm}$ are linear functions. Define another piecewise linear function $v_h$ as 
\begin{equation*}
v_h(x):=[\![w_h]\!](D)\Psi_D(x)+[\![w_h]\!](E)\Psi_E(x)+[\![\beta \nabla w_h\cdot \textbf{n}_h]\!]\Upsilon(x).
\end{equation*}
Next, we prove $w_h=v_h$. From the definition (\ref{jum_ji1})-(\ref{def_psi_linear2}), we have
\begin{equation*}
[\![v_h]\!](D)=[\![w_h]\!](D),~ [\![v_h]\!](E)=[\![w_h]\!](E),~[\![\beta \nabla v_h\cdot \textbf{n}_h]\!]=[\![\beta \nabla w_h\cdot \textbf{n}_h]\!],~v_h(A_i)=0, i=1,2,3,
\end{equation*}
which implies $w_h(x)-v_h(x)\in S_h(T)$ and $(w_h-v_h)(A_i)=0, i=1,2,3$. From Lemma~\ref{lem_unique}, we know that the function $w_h-v_h$ is unique and $w_h-v_h=0$ through a simple verification.
Now, we get the decomposition  
\begin{equation}\label{adf}
w_h(x)=v_h(x)=[\![w_h]\!](D)\Psi_D(x)+[\![w_h]\!](E)\Psi_E(x)+[\![\beta \nabla w_h\cdot \textbf{n}_h]\!]\Upsilon(x).
\end{equation}
From (\ref{conti_linear}), (\ref{rolle}) and the definition (\ref{shuangfang2}), we find
\begin{equation*}
\begin{aligned}
&[\![w_h]\!](x_p)=[\![I_h^{BK}v]\!](x_p)-[\![I_h^{{\rm IFE}}v]\!](x_p)=[\![I_h^{BK}v]\!](x_p),\quad x_p=D,E,\\
&[\![\beta \nabla w_h\cdot \textbf{n}_h]\!]=[\![\beta (\nabla I_h^{BK}v)\cdot \textbf{n}_h]\!]-[\![\beta (\nabla I_h^{{\rm IFE}}v)\cdot \textbf{n}_h]\!]=[\![\beta (\nabla I_h^{BK}v)\cdot \textbf{n}]\!](x^*).
\end{aligned}
\end{equation*}
The above identities combined with (\ref{adf}) lead to (\ref{lemma4}).
\qed \end{proof}
\begin{lemma}\label{lem_jiest}
For each interface element $T\in\mathcal{T}^\Gamma_h$, let $\Psi_D(x)$, $\Psi_E(x)$  and $\Upsilon(x)$ be defined in (\ref{jum_ji1})-(\ref{def_psi_linear2}). Under the condition of Lemma~\ref{lem_unique}, there hold
\begin{equation*}
\begin{aligned}
&\|\Psi_i\|^2_{L^2(T)}\leq Ch^2,~|\Psi_i|^2_{H^1(T_h^+\cup T_h^-)}\leq C,\qquad i=D,E,\\
&\|\Upsilon\|^2_{L^2(T)}\leq Ch^4,~|\Upsilon|^2_{H^1(T)}\leq Ch^2,
\end{aligned}
\end{equation*}
where the constant $C$ only depends on $\beta^\pm$ and the shape-regular parameter $\kappa_2$.
\end{lemma}
\begin{proof}:
See  Appendix \ref{pro_lem_jiest}. \qed
\end{proof}
The following lemma provides a relation between  $\mathbb{E}_hv$ and $I_h^{{\rm IFE}}v$ for all $v\in  \widetilde{H}^2(\Omega)$.
\begin{lemma}\label{chazhi_error}
For any $v\in  \widetilde{H}^2(\Omega)$, under the condition of Lemma~\ref{lem_unique}, there exists a constant $C$ independent of $h$ and the interface location relative to the mesh such that
\begin{equation*}
\sum_{T\in\mathcal{T}_h^\Gamma}|\mathbb{E}_hv-I_h^{{\rm IFE}}v|^2_{H^m(T_h^+\cup T_h^-)}\leq Ch^{4-2m}\|v\|^2_{H^2(\Omega^+\cup\Omega^-)},\quad m=0,1.
\end{equation*}
\end{lemma}

\begin{proof}:
For each interface element $T\in\mathcal{T}_h^\Gamma$, by  the triangle inequality, we have
\begin{equation}\label{pro_main1}
|\mathbb{E}_hv-I_h^{{\rm IFE}}v|_{H^m(T_h^+\cup T_h^-)}\leq |\mathbb{E}_hv-I_h^{BK}v|_{H^m(T_h^+\cup T_h^-)}+|I_h^{BK}v-I_h^{{\rm IFE}}v|_{H^m(T_h^+\cup T_h^-)}.
\end{equation}
From (\ref{epm}) and (\ref{ehpm}), the estimate of the first term is the standard
\begin{equation}\label{main_prof12}
\begin{aligned}
|\mathbb{E}_hv-I_h^{BK}v|^2_{H^m(T_h^+\cup T_h^-)}&=(|\mathbb{E}^+v^+-I_h\mathbb{E}^+v^+|^2_{H^m(T_h^+)}+|\mathbb{E}^-v^--I_h\mathbb{E}^-v^-|^2_{H^m(T_h^-)})\\
&\leq Ch^{4-2m}(|\mathbb{E}^+v^+|^2_{H^2(T)}+|\mathbb{E}^-v^-|^2_{H^2(T)}).\\
\end{aligned}
\end{equation}
For the second term on the right-hand side of (\ref{pro_main1}),  from Lemma~\ref{lema_fenjie} and Lemma~\ref{lem_jiest},  we have
\begin{equation}\label{main_prof1}
\begin{aligned}
|I_h^{BK}v-I_h^{{\rm IFE}}&v|^2_{H^m(T_h^+\cup T_h^-)}\leq 3\sum_{x_p=D,E}[\![ I_h^{BK}v]\!]^2(x_p)|\Psi_{x_p}|^2_{H^m(T_h^+\cup T_h^-)}\\
&\qquad\qquad\qquad+3[\![\beta \nabla (I_h^{BK}v)\cdot \textbf{n}]\!]^2(x^*)|\Upsilon(x)|^2_{H^m(T)}\\
&\leq Ch^{2-2m}\sum_{x_p=D,E}[\![ I_h^{BK}v]\!]^2(x_p)+Ch^{4-2m}[\![\beta \nabla (I_h^{BK}v)\cdot \textbf{n}]\!]^2(x^*).
\end{aligned}
\end{equation}
Since $v\in \widetilde{H}^2(\Omega)$, we have  $[\![v]\!]^2(D)=[\![v]\!]^2(E)=0$, which leads to
\begin{equation}\label{pro_main3}
\begin{aligned}
\sum_{x_p=D,E}[\![ I_h^{BK}v]\!]^2(x_p)=&\sum_{x_p=D,E}[\![ I_h^{BK}v-v]\!]^2(x_p)\leq C\|[\![I_h^{BK}v-v]\!]\|_{L^\infty(T)}^2\\
&\leq C\|I_h\mathbb{E}^+v^+-\mathbb{E}^+v^+\|^2_{L^\infty(T)}+C\|I_h\mathbb{E}^-v^--\mathbb{E}^-v^-\|^2_{L^\infty(T)}\\
&\leq Ch^2(|\mathbb{E}^+v^+|^2_{H^2(T)}+|\mathbb{E}^-v^-|^2_{H^2(T)}),
\end{aligned}
\end{equation}
where we have used the standard interpolation error estimate in the last inequality, see Theorem 4.4.20 in \cite{brenner2008mathematical}.

The remaining term $[\![\beta \nabla (I_h^{BK}v)\cdot \textbf{n}]\!](x^*)$ in (\ref{main_prof1}) cannot be treated as (\ref{pro_main3})  because $\nabla u \cdot\textbf{n}$ is not well-defined at the point $x^*$ when $u\in \widetilde{H}^2(\Omega)$. Using the standard inverse inequality,  (\ref{error_nt}) and the relation $\textbf{n}(x^*)=\textbf{n}_h$, we can derive
\begin{equation}\label{pro_main4}
\begin{aligned}
|[\![\beta \nabla& (I_h^{BK}v)\cdot \textbf{n}]\!]^2(x^*)= \|[\![\beta \nabla (I_h^{BK}v)\cdot \textbf{n}_h]\!] \|_{L^\infty(T)}^2 \leq Ch^{-2} \|[\![\beta \nabla (I_h^{BK}v)\cdot \textbf{n}_h]\!] \|_{L^2(T)}^2\\
&\leq Ch^{-2}\left(\|[\![\beta \nabla (I_h^{BK}v)\cdot \textbf{n}_h-\beta \nabla v\cdot \textbf{n}_h]\!]\|_{L^2(T)}^2+\left\|[\![\beta \nabla v\cdot (\textbf{n}+\textbf{n}_h-\textbf{n})]\!]\right\|_{L^2(T)}^2\right)\\
&\leq C\sum_{i=\pm}|\mathbb{E}^iv^i|^2_{H^2(T)} +Ch^{-2}\left(\left\|[\![\beta \nabla v\cdot \textbf{n}]\!]\right\|_{L^2(T)}^2+\| \textbf{n}_h-\textbf{n}\|^2_{L^\infty(T)}\left\|[\![\beta \nabla v]\!]\right\|_{L^2(T)}^2\right)\\
&\leq C\sum_{i=\pm}\left(|\mathbb{E}^iv^i|^2_{H^2(T)}+|\mathbb{E}^iv^i|^2_{H^1(T)}\right) +Ch^{-2}\left\|[\![\beta \nabla v\cdot \textbf{n}]\!]\right\|_{L^2(T)}^2,
\end{aligned}
\end{equation}
where $[\![\beta \nabla v\cdot \textbf{n}]\!]$ represents the jump of the flux of $\mathbb{E}^\pm v^\pm$ (see (\ref{shuangfang1}) for the definition of the notation $[\![\cdot ]\!]$).
We combine (\ref{pro_main1})-(\ref{pro_main4}) to  obtain the error estimate on the interface element
\begin{equation*}
|\mathbb{E}_hv-I_h^{{\rm IFE}}v|^2_{H^m(T_h^+\cup T_h^-)}\leq Ch^{4-2m}\sum_{i=\pm}\|\mathbb{E}^iv^i\|^2_{H^2(T)}+Ch^{2-2m}\left\|[\![\beta \nabla v\cdot \textbf{n}]\!]\right\|_{L^2(T)}^2.
\end{equation*}
Summing up and using Lemma~\ref{lem_ext}, we get
\begin{equation}\label{pro_main_ls}
\begin{aligned}
\sum_{T\in\mathcal{T}^\Gamma_h}|\mathbb{E}_hv-I_h^{{\rm IFE}}v|^2_{W^m_2(T_h^+\cup T_h^-)}\leq Ch^{4-2m}\|v\|^2_{H^2(\Omega^+\cup\Omega^-)}+Ch^{2-2m}\sum_{T\in\mathcal{T}_h^\Gamma}\left\|[\![\beta \nabla v\cdot \textbf{n}]\!]\right\|_{L^2(T)}^2.
\end{aligned}
\end{equation}
Since $v\in \widetilde{H}^2(\Omega)$, from the definition (\ref{def_H2}) we know that $[\![\beta \nabla v\cdot \textbf{n}]\!]=0$ on $\Gamma$. Thus, by Lemma~\ref{strip} and the fact $\textbf{n}(x)\in \left(C^1\left(N(\Gamma,\delta_0)\right)\right)^2$, we have
\begin{equation}\label{pro_main_new}
\begin{aligned}
\sum_{T\in\mathcal{T}_h^\Gamma}\left\|[\![\beta \nabla v\cdot \textbf{n}]\!]\right\|_{L^2(T)}^2&\leq \left\|[\![\beta \nabla v\cdot \textbf{n}]\!]\right\|_{L^2(N(\Gamma,h))}^2\leq Ch^2\left|[\![\beta \nabla v\cdot \textbf{n}]\!]\right|_{H^1(N(\Gamma,h))}^2\\
&\leq Ch^2(\|\mathbb{E}^+v^+\|^2_{H^2(N(\Gamma,h))}+\|\mathbb{E}^-v^-\|^2_{H^2(N(\Gamma,h))})\leq Ch^2\|v\|^2_{H^2(\Omega^+\cup\Omega^-)}.
\end{aligned}
\end{equation}
Finally, substituting this into (\ref{pro_main_ls}) we complete the proof of the lemma.\qed
\end{proof}


Now we are ready to prove the optimal approximation properties of the linear IFE space.
\begin{theorem}\label{theorem_interpolation}
For any $v\in  \widetilde{H}^2(\Omega)$, under the condition of Lemma~\ref{lem_unique}, there exists a constant $C$  independent of $h$ and the interface location relative to the mesh such that
\begin{equation}\label{lem_main}
\|v-I_h^{{\rm IFE}}v\|_{L^2(\Omega)}+h\left(\sum_{T\in\mathcal{T}_h}|v-I_h^{{\rm IFE}}v|^2_{H^1(T)}\right)^{1/2}\leq Ch^2\|v\|_{H^2(\Omega^+\cup\Omega^-)}.
\end{equation}
\end{theorem}
\begin{proof}:
On each non-interface element $T\in\mathcal{T}_h^{non}$, we have the standard  estimate
\begin{equation}\label{pro_lemmain1}
\|v-I_h^{{\rm IFE}}v\|_{L^2(T)}+h|v-I_h^{{\rm IFE}}v|_{H^1(T)}\leq Ch^2\|v\|_{H^2(T)}.
\end{equation}

On each interface element $T\in\mathcal{T}_h^\Gamma$, by the triangle inequality, we have
\begin{equation}\label{pro_lemmain2}
|v-I_h^{{\rm IFE}}v|^2_{H^m(T)}\leq 2|\mathbb{E}_hv-I_h^{{\rm IFE}}v|^2_{H^m(T_h^+\cup T_h^-)}+2|v-\mathbb{E}_hv|^2_{H^m(T_h^+\cup T_h^-)}, ~~m=1,2.
\end{equation}
The first term on the right hand-side can be estimated by Lemma~\ref{chazhi_error}.

Next, we try to estimate the second term on the right-hand side of (\ref{pro_lemmain2}).
From (\ref{epm}), we know that $\mathbb{E}_hv=\mathbb{E}^+ v^+$ on $ T_h^+$, which together with the fact $ \mathbb{E}^+ v^+=v^+=v$ on $T_h^+\cap T^+$  implies
\begin{equation}\label{pro_lemmain3}
|\mathbb{E}_hv-v|^2_{H^m(T_h^+)}=|\mathbb{E}^+ v^+-v|^2_{H^m(T_h^+)}=|\mathbb{E}^+ v^+-v^-|^2_{H^m(T_h^+ \backslash T^+)},~m=1,2.
\end{equation}
Since $T_h^+ \backslash T^+ \subset T^\triangle$, $T_h^+ \backslash T^+ \subset \Omega^-$ and $\mathbb{E}^+ v^+=\mathbb{E}^-v^-$ on $\Gamma$, it follows from Lemma~\ref{lem_h3} that
\begin{equation}\label{pro_lemmain4}
\begin{aligned}
&\|\mathbb{E}^+ v^+-v^-\|^2_{L^2(T_h^+\backslash T^+)}\leq Ch^4|\mathbb{E}^+ v^+-\mathbb{E}^-v^-|^2_{H^1(T^\triangle)}\leq Ch^4\sum_{i=\pm}|\mathbb{E}^i v^i|^2_{H^1(T)},\\
&\|\nabla (\mathbb{E}^+ v-v^-)\|^2_{L^2(T_h^+ \backslash T^+ )}\leq C\left(h^2\|\nabla (v^+-v^-)\|^2_{L^2(\Gamma\cap T)}+h^4|\mathbb{E}^+ v-\mathbb{E}^-v^-|^2_{H^2(T^\triangle)}\right)\\
&~\quad\qquad\qquad\qquad\qquad\quad\quad\leq Ch^2\sum_{i=\pm}\|\nabla  v^i\|^2_{L^2(\Gamma\cap T)}+Ch^4\sum_{i=\pm}|\mathbb{E}^i v^i|^2_{H^2(T)}.
\end{aligned}
\end{equation}
From (\ref{pro_lemmain3}) and (\ref{pro_lemmain4}), we get, for $m=0,1$,
\begin{equation}\label{pro_lemmain5}
|v-\mathbb{E}_hv|^2_{H^m(T_h^+)}\leq Ch^{4-2m}\sum_{i=\pm}\|\nabla  v^i\|^2_{L^2(\Gamma\cap T)}+Ch^4\sum_{i=\pm}\|\mathbb{E}^i v^i\|^2_{H^2(T)}.
\end{equation}
Analogously, we have the following result on $T_h^-$,
\begin{equation}\label{pro_lemmain6}
|v-\mathbb{E}_hv|^2_{H^m(T_h^-)}\leq Ch^{4-2m}\sum_{i=\pm}\|\nabla  v^i\|^2_{L^2(\Gamma\cap T)}+Ch^4\sum_{i=\pm}\|\mathbb{E}^i v^i\|^2_{H^2(T)}.
\end{equation}
Combining (\ref{pro_lemmain1}), (\ref{pro_lemmain2}), (\ref{pro_lemmain5}), (\ref{pro_lemmain6}) and Lemma~\ref{chazhi_error}, we have
\begin{equation}\label{pro_lemmain7}
\begin{aligned}
\sum_{T\in\mathcal{T}_h}|v-I_h^{{\rm IFE}}v|^2_{H^m(T)}&\leq  Ch^{4-2m}\left(\|v\|^2_{H^2(\Omega^+\cup\Omega^-)}+\sum_{i=\pm} \|\nabla v^i\|^2_{L^2(\Gamma)}\right)+Ch^4\sum_{i=\pm}\|\mathbb{E}^i v^i\|^2_{H^2(\Omega)},
\end{aligned}
\end{equation}
which together with the global trace inequality on $\Omega^\pm$
\begin{equation}\label{pro_lemmain8}
\sum_{i=\pm} \|\nabla v^i\|^2_{L^2(\Gamma)}\leq C(\|v\|^2_{H^2(\Omega^-)}+\|v\|^2_{H^2(\Omega^+)})
\end{equation}
and the continuity of the extension (see Lemma~\ref{lem_ext})
\begin{equation}\label{pro_lemmain9}
\|\mathbb{E}^+ v^+\|^2_{H^2(\Omega)}+\|\mathbb{E}^- v^-\|^2_{H^2(\Omega)}\leq C(\|v\|^2_{H^2(\Omega^+)}+\|v\|^2_{H^2(\Omega^-)})
\end{equation}
implies  the estimate (\ref{lem_main}).
\qed \end{proof}

\subsection{The trace inequality for the space $\widetilde{H}^2(T)+ S_h(T)$}\label{subsec_tra}
 Assume $v\in \widetilde{H}^2(T)$ and $w_h\in S_h(T)$, the standard trace inequality cannot be applied to $\nabla (v-w_h)$ because $v-w_h\not \in H^2(T)$.
To establish the trace inequality for the broken space, we first need the following trace lemma for the space $\widetilde{H}^2(T)$.
\begin{lemma}\label{lem_tra2}
Assume  $v\in \widetilde{H}^2(T)$, there exists a constant $C$ independent of $h$ and the interface location relative to the mesh such that
\begin{equation*}
\|\nabla v\|_{L^2(\partial T)}\leq C(h^{-\frac{1}{2}}\|\nabla v\|_{L^2(T)}+h^{\frac{1}{2}}|v|_{H^2(T^+\cup T^-)}).
\end{equation*}
\end{lemma}
\begin{proof}:
Since $v\in \widetilde H^2(T)$, we have
$\nabla v^\pm\in (H^1(T^\pm))^2$ with $T^\pm=\Omega^\pm\cap T.$
By (\ref{nt_smooth}), we obtain
\begin{equation*}
\nabla v^\pm\cdot \textbf{t}(x)\in H^1(T^\pm)~\mbox{ and }~\beta^\pm\nabla v^\pm\cdot\textbf{n}(x)\in H^1(T^\pm).
\end{equation*}
Using the condition $[v]_{\Gamma\cap T}=[\beta\nabla v\cdot\textbf{n}]_{\Gamma\cap T}=0$ in the definition  (\ref{def_H2}), we have
\begin{equation}\label{lem1_p1}
\nabla v\cdot \textbf{t}(x)\in H^1(T)~\mbox{ and }~\beta(x)\nabla v\cdot\textbf{n}(x)\in H^1(T).
\end{equation}
By the standard trace inequality, we get
\begin{equation*}
\begin{aligned}
&\|\nabla v\cdot \textbf{t}\|_{L^2(\partial T)}\leq C(h^{-\frac{1}{2}}\|\nabla v\cdot \textbf{t}\|_{L^2(T)}+h^{\frac{1}{2}}|\nabla v\cdot \textbf{t}|_{H^1(T)}),\\
&\|\beta\nabla v\cdot \textbf{n}\|_{L^2(\partial T)}\leq C(h^{-\frac{1}{2}}\|\beta\nabla v\cdot \textbf{n}\|_{L^2(T)}+h^{\frac{1}{2}}|\beta\nabla v\cdot \textbf{n}|_{H^1(T)}).
\end{aligned}
\end{equation*}
Hence, we finally conclude
\begin{equation*}
\begin{aligned}
\|\nabla v&\|_{L^2(\partial T)}\leq C(\|\nabla v\cdot \textbf{t}\|_{L^2(\partial T)}+ \|\beta\nabla v\cdot \textbf{n}\|_{L^2(\partial T)})\\
&\leq Ch^{-\frac{1}{2}}(\|\nabla v\cdot \textbf{t}\|_{L^2(T)}+\|\beta\nabla v\cdot \textbf{n}\|_{L^2(T)})+Ch^{\frac{1}{2}}(|\nabla v\cdot \textbf{t}|_{H^1(T)}+|\beta\nabla v\cdot \textbf{n}|_{H^1(T)})\\
&\leq C(h^{-\frac{1}{2}}\|\nabla v\|_{L^2(T)}+h^{\frac{1}{2}}|v|_{H^2(T^+\cup T^-)}).
\end{aligned}
\end{equation*}
\qed \end{proof}

The following trace inequality for the spaces $\widetilde{H}^2(T)+ S_h(T)$ is important in the convergence proof.
\begin{lemma}\label{lema_trace}
Let $T\in\mathcal{T}_h^\Gamma$ be an interface element and $S_h(T)$ be  the linear IFE shape function space.  For any $v\in \widetilde{H}^2(T)$ and any $w_h\in S_h(T)$,  there exists a constant $C$ independent of $h$ and the interface location relative to the mesh such that
\begin{equation}\label{le1}
\|\nabla (v-w_h)\|_{L^2(\partial T)}\leq Ch^{-\frac{1}{2}}\|\nabla(v-w_h)\|_{L^2(T)}+Ch^{\frac{1}{2}}(|v|_{H^2(T^+\cup T^-)}+|v|_{H^1(T)}).
\end{equation}
\end{lemma}
\begin{proof}:
First we split the left-hand side of the inequality as
\begin{equation}\label{lem_tra}
\begin{aligned}
\|\nabla(v-w_h)\|_{L^2(\partial T)}&\leq C \|\nabla (v-w_h)\cdot \textbf{t}_h\|_{L^2(\partial T)}+C\|\beta \nabla(v-w_h)\cdot\textbf{n}_h\|_{L^2(\partial T)}\\
&\leq C\left(\|\nabla v\cdot \textbf{t}-\nabla w_h\cdot \textbf{t}_h\|_{L^2(\partial T)}+\|\nabla v\cdot(\textbf{t}-\textbf{t}_h)\|_{L^2(\partial T)}\right.\\
&\quad + \left.\|\beta\nabla v\cdot \textbf{n}-\beta_h\nabla w_h\cdot \textbf{n}_h\|_{L^2(\partial T)}+\|\beta\nabla v\cdot(\textbf{n}-\textbf{n}_h)\|_{L^2(\partial T)}\right),
\end{aligned}
\end{equation}
where we have used the fact $\beta=\beta_h$ on the boundary $\partial T$.

Next, we estimate the first and the third terms on the right-hand side of (\ref{lem_tra}). From (\ref{conti_linear}), we have
\begin{equation}\label{lem1_p2}
\nabla w_h\cdot \textbf{t}_h(x)\in H^1(T),
\end{equation}
which together with (\ref{lem1_p1}) implies
\begin{equation*}
\nabla v\cdot \textbf{t}-\nabla w_h\cdot \textbf{t}_h\in H^1(T).
\end{equation*}
Thus, by the standard trace inequality, the first term on the right-hand side of (\ref{lem_tra}) can be estimated as
\begin{equation}\label{pro_trace1}
\begin{aligned}
\|\nabla v\cdot \textbf{t}&-\nabla w_h\cdot \textbf{t}_h\|_{L^2(\partial T)}\leq C\left(h^{-\frac{1}{2}}\|\nabla v\cdot \textbf{t}-\nabla w_h\cdot \textbf{t}_h\|_{L^2(T)}+h^{\frac{1}{2}}|\nabla v\cdot \textbf{t}-\nabla w_h\cdot \textbf{t}_h|_{H^1(T)}\right)\\
&\leq Ch^{-\frac{1}{2}}\left(\|\nabla v\cdot \textbf{t}_h-\nabla w_h\cdot \textbf{t}_h\|_{L^2(T)}+\|\nabla v\cdot (\textbf{t}- \textbf{t}_h)\|_{L^2(T)}\right) +Ch^{\frac{1}{2}}|v|_{H^2(T^+\cup T^-)}\\
&\leq Ch^{-\frac{1}{2}}\left(\|\nabla (v- w_h)\|_{L^2(T)}+\|\nabla v\cdot (\textbf{t}- \textbf{t}_h)\|_{L^2(T)}\right) +Ch^{\frac{1}{2}}|v|_{H^2(T^+\cup T^-)}.
\end{aligned}
\end{equation}
For the third term on the right-hand side of (\ref{lem_tra}), from (\ref{conti_linear}), we also we have
\begin{equation}\label{lem1_p2_2}
\beta_h(x)\nabla w_h\cdot\textbf{n}_h\in H^1(T).
\end{equation}
Thus, it follows from (\ref{lem1_p1}) that
\begin{equation}\label{trace_he}
\beta\nabla v\cdot \textbf{n}-\beta_h\nabla w_h\cdot \textbf{n}_h\in H^1(T).
\end{equation}
Similar to (\ref{pro_trace1}), by the standard trace inequality, we obtain
\begin{equation}\label{pro_trac_1}
\begin{aligned}
\|\beta\nabla v\cdot \textbf{n}-\beta_h\nabla w_h\cdot \textbf{n}_h\|_{L^2(\partial T)}&\leq Ch^{-\frac{1}{2}}\left(\|\nabla (v- w_h)\|_{L^2(T)}+\|\nabla v\cdot (\textbf{n}- \textbf{n}_h)\|_{L^2(T)}\right)\\
 &\quad+Ch^{\frac{1}{2}}|v|_{H^2(T^+\cup T^-)}.
 \end{aligned}
\end{equation}
Combining (\ref{lem_tra}), (\ref{pro_trace1}) and (\ref{pro_trac_1}), we obtain
\begin{equation*}
\begin{aligned}
\|\nabla(v-w_h)\|_{L^2(\partial T)}&\leq Ch^{-\frac{1}{2}}\|\nabla (v- w_h)\|_{L^2(T)}+Ch^{\frac{1}{2}}|v|_{H^2(T^+\cup T^-)}\\
&\quad+Ch^{-\frac{1}{2}}(\|\nabla v\cdot (\textbf{n}- \textbf{n}_h)\|_{L^2(T)}+\|\nabla v\cdot (\textbf{t}- \textbf{t}_h)\|_{L^2(T)})\\
&\quad+C(\|\nabla v\cdot(\textbf{t}-\textbf{t}_h)\|_{L^2(\partial T)}+\|\nabla v\cdot(\textbf{n}-\textbf{n}_h)\|_{L^2(\partial T)}).
\end{aligned}
\end{equation*}
Using (\ref{error_nt}), we further get
\begin{equation*}
\begin{aligned}
\|\nabla(v-w_h)\|_{L^2(\partial T)}&\leq Ch^{-\frac{1}{2}}\|\nabla (v- w_h)\|_{L^2(T)}+Ch^{\frac{1}{2}}|v|_{H^2(T^+\cup T^-)}\\
&\quad+Ch^{\frac{1}{2}}\|\nabla v\|_{L^2(T)}+Ch\|\nabla v\|_{L^2(\partial T)},
\end{aligned}
\end{equation*}
which, together with Lemma~\ref{lem_tra2}, completes the proof.

\qed \end{proof}

\subsection{The stability analysis of the local lifting operator}

\begin{lemma}\label{lem_stab_lift}
There exists a  constant $C$ independent of $h$ and the interface location relative to the mesh such that
\begin{equation*}
\|r_e(\varphi)\|_{L^2(\Omega)}\leq Ch^{-\frac{1}{2}}\|\varphi\|_{L^2(e)},\quad \forall \varphi\in L^2(e),\quad \forall e\in\mathcal{E}_h^\Gamma.
\end{equation*}
\end{lemma}
\begin{proof}:
Let $T_1$ and $T_2$ be the interface elements sharing  the edge $e$, i.e., $\overline{T_1}\cap \overline{T_2}=\overline{e}$ and $T_1, T_2\in\mathcal{T}_h^\Gamma$.  Then the support of $r_e(\varphi)$ is $T_1\cup T_2$.
Taking $w_h=r_e(\varphi)$ in (\ref{def_lift}), we have
\begin{equation}\label{pro_lem_lif1}
\begin{aligned}
\|r_e(\varphi)\|^2_{L^2(\Omega)}&\leq C\|\beta_h^{1/2}r_e(\varphi)\|^2_{L^2(T_1\cup T_2)}=C\int_e\{\beta_h r_e(\varphi)\cdot\textbf{n}_e\}_e\varphi ds\\
&\leq C\|\{\beta_hr_e(\varphi)\}_e\|_{L^2(e)}\|\varphi\|_{L^2(e)}\leq C\|\varphi\|_{L^2(e)}\sum_{i=1,2}\|r_e(\varphi)|_{T_i}\|_{L^2(e)}.
\end{aligned}
\end{equation}
By definition, there exists a function $v_h\in S_h(T_1)$ such that $\nabla v_h=r_e(\varphi)|_{T_1}$. 
Choosing $v=0$ and $w_h=v_h$ in Lemma~\ref{lema_trace}, we have
\begin{equation}\label{pro_lem_lif2}
  \|r_e(\varphi)|_{T_1} \|_{L^2(e)}=\|\nabla v_h\|_{L^2(e)}\leq Ch^{-\frac{1}{2}}\|\nabla v_h\|_{L^2(T_1)}=Ch^{-\frac{1}{2}}\|r_e(\varphi)\|_{L^2(T_1)}.
\end{equation}
Similarly, on the element $T_2$, we also have the  estimate
\begin{equation}\label{pro_lem_lif3}
  \|r_e(\varphi)|_{T_2} \|_{L^2(e)}\leq Ch^{-\frac{1}{2}}\|r_e(\varphi)\|_{L^2(T_2)}.
\end{equation}
The lemma follows from (\ref{pro_lem_lif1})-(\ref{pro_lem_lif3}).
\qed \end{proof}

\subsection{The optimal convergence analysis  of the parameter free PPIFE method}\label{sub_sec_ana_IFE}

For all $v\in \left(H_0^1(\Omega)\cap\widetilde{H}^2(\Omega)\right)+ V_h^{{\rm IFE}}$, we define the following mesh-dependent norms
\begin{equation*}
\|v\|^2_h=\sum_{T\in\mathcal{T}_h}\|\sqrt{\beta_h}\nabla v\|^2_{L^2(T)}
\end{equation*}
and
\begin{equation}\label{def_nor2}
\interleave v \interleave^2_h=\|v\|_h^2+\sum_{e\in\mathcal{E}_h^\Gamma}h\|\{\beta_h\nabla v\}_e\|^2_{L^2(e)}+\sum_{e\in\mathcal{E}_h^\Gamma}h^{-1}\| [v]_e\|^2_{L^2(e)}+s_h(v,v).
\end{equation}
Note that $\|\cdot\|_h$ is indeed a norm because $\|v\|_h=0$ implies $v$ is a piecewise constant, which together with the zero boundary condition and the continuity at $\Gamma_h$ and nodal points implies $v=0$.

It follows from the Cauchy-Schwarz inequality that
\begin{equation}\label{conti}
|A_h(w,v)|\leq \interleave w \interleave_h \interleave v\interleave_h,\qquad\forall w,v\in \left(H_0^1(\Omega)\cap\widetilde{H}^2(\Omega)\right)+ V_h^{{\rm IFE}}.
\end{equation}
The following lemma shows that the parameter free PPIFE method is coercive with respect to $\|\cdot\|_h$.

\begin{lemma}[Coercivity]
The parameter free PPIFE method has the following coercive relation
\begin{equation}\label{Coercivity}
A_h(v_h,v_h)\geq \frac{1}{2}\|v_h\|_h^2,\qquad \forall v_h\in V_h^{{\rm IFE}}.
\end{equation}
\end{lemma}
\begin{proof}:
For any $e\in\mathcal{E}_h^\Gamma$, we denote by $\mathcal{P}_e$ the set of two triangles in $\mathcal{T}_h^\Gamma$ sharing  the edge $e$. From (\ref{def_lift}) and (\ref{coneect}), we know that the support of $r_e([v_h]_e)$ is $\cup_{T\in\mathcal{P}_e} T$ and
\begin{equation*}
2\sum_{e\in\mathcal{E}_h^\Gamma}\int_e\{\beta_h\nabla v_h\cdot\textbf{n}_e\}_e[v_h]_eds=2\sum_{e\in\mathcal{E}_h^\Gamma} \sum_{T\in \mathcal{P}_e}\int_T \beta_h(x) r_e([v_h]_e)\cdot \nabla v_hdx.
\end{equation*}
From the Cauchy-Schwarz inequality, we further get
\begin{equation}\label{pro_coer1}
\begin{aligned}
2\sum_{e\in\mathcal{E}_h^\Gamma}&\int_e\{\beta_h\nabla v_h\cdot\textbf{n}_e\}_e[v_h]_eds\\
&\leq 2\sum_{e\in\mathcal{E}_h^\Gamma}\left( \sum_{T\in \mathcal{P}_e }\int_T \beta_h r_e([v_h]_e)\cdot r_e([v_h]_e)dx\right)^{\frac{1}{2}}\left( \sum_{T\in \mathcal{P}_e }\int_T \beta_h \nabla v_h \cdot \nabla v_h dx\right)^{\frac{1}{2}}\\
&\leq 2\left(\sum_{e\in\mathcal{E}_h^\Gamma} \sum_{T\in \mathcal{P}_e }\int_T \beta_h r_e([v_h]_e)\cdot r_e([v_h]_e)dx\right)^{\frac{1}{2}}\left(\sum_{e\in\mathcal{E}_h^\Gamma} \sum_{T\in \mathcal{P}_e }\int_T \beta_h \nabla v_h \cdot \nabla v_h dx\right)^{\frac{1}{2}}.
\end{aligned}
\end{equation}
From Assumption A, we know that each interface element has at most two interface edges. Thus, each interface element is calculated at most twice, i.e.,
\begin{equation}\label{pro_coer2}
\sum_{e\in\mathcal{E}_h^\Gamma} \sum_{T\in \mathcal{P}_e }\int_T \beta_h \nabla v_h \cdot \nabla v_h dx\leq 2\sum_{T\in\mathcal{T}_h}\int_T \beta_h(x) \nabla v_h \cdot \nabla v_h dx.
\end{equation}
Substituting (\ref{sh}) and (\ref{pro_coer2}) into (\ref{pro_coer1}), we find
\begin{equation*}
\begin{aligned}
2\sum_{e\in\mathcal{E}_h^\Gamma}\int_e\{\beta_h\nabla v_h\cdot\textbf{n}_e\}_e[v_h]_eds&\leq \left(s_h(v_h,v_h)\right)^{\frac{1}{2}}\left(2\sum_{T\in\mathcal{T}_h}\int_T \beta_h(x) \nabla v_h \cdot \nabla v_h dx\right)^{\frac{1}{2}}\\
&\leq \frac{1}{2\epsilon}s_h(v_h,v_h)+\epsilon \sum_{T\in\mathcal{T}_h}\int_T \beta_h(x) \nabla v_h \cdot \nabla v_h dx.
\end{aligned}
\end{equation*}
 With $\epsilon=\frac{1}{2}$, the inequality above becomes,
\begin{equation*}
2\sum_{e\in\mathcal{E}_h^\Gamma}\int_e\{\beta_h\nabla v_h\cdot\textbf{n}_e\}_e[v_h]_eds\leq s_h(v_h,v_h)+\frac{1}{2} \sum_{T\in\mathcal{T}_h}\int_T \beta_h(x) \nabla v_h \cdot \nabla v_h dx.
\end{equation*}
Therefore, from (\ref{ph2}) and (\ref{sh}) we arrive at
\begin{equation*}
a_h(v_h,v_h)+s_h(v_h,v_h)\geq \frac{1}{2} \sum_{T\in\mathcal{T}_h}\int_T \beta_h(x) \nabla v_h \cdot \nabla v_h dx = \frac{1}{2}\|v_h\|_h^2.
\end{equation*}
This completes the proof of this lemma.
\qed \end{proof}

Next, we try to prove the equivalence of the two norms $\|\cdot\|_h$ and $ \interleave\cdot\interleave_h$ on the IFE space $V_h^{{\rm IFE}}$. We need the following lemma, see Lemma 3.4 in \cite{ji2014sym} and (4.15) in \cite{guo2021SIAM} for the 2D cases, and  Appendix \ref{pro_lem_tiaoyue} for the 3D cases.
\begin{lemma}\label{lem_tiaoyue}
Under the condition of Lemma~\ref{lem_unique}, there exists a constant $C$ independent of $h$ and the interface location relative to the mesh such that
\begin{equation}\label{tiaoyue}
\|[\phi]_e\|^2_{L^2(e)}\leq Ch\left(\|\nabla \phi\|^2_{L^2(T_1)}+\|\nabla \phi\|^2_{L^2(T_2)}\right),\quad \forall e\in\mathcal{E}_h^\Gamma,~~ \forall \phi \in V_h^{{\rm IFE}},
\end{equation}
where $\overline{T_1}\cap \overline{T_2}=\overline{e}$ and $T_1, T_2\in\mathcal{T}_h^\Gamma$.
\end{lemma}

The equivalence of these two norms is shown  in the following lemma.
\begin{lemma}\label{lema_equ}
Under the condition of Lemma~\ref{lem_unique}, there exists a constant $C$ independent of $h$ and the interface location relative to the mesh such that
\begin{equation}\label{equali}
\|v_h\|_h \leq \interleave v_h \interleave_h\leq C\|v_h\|_h,\qquad \forall v_h\in V_h^{{\rm IFE}}.
\end{equation}
\end{lemma}
\begin{proof}:
The first inequality is obvious. Thus, we just need to prove the second inequality. For any $e\in\mathcal{E}_h^\Gamma$, we denote by $\mathcal{P}_e$ the set of two triangles in $\mathcal{T}_h^\Gamma$ sharing the edge $e$. Setting $v=0$ in Lemma~\ref{lema_trace}, we have
\begin{equation}\label{pro_equ1}
\begin{aligned}
\sum_{e\in\mathcal{E}_h^\Gamma}h&\|\{\beta_h\nabla v_h\}_e\|^2_{L^2(e)}\leq C\sum_{e\in\mathcal{E}_h^\Gamma} \sum_{T\in\mathcal{P}_e}h\|\nabla v_h\|^2_{L^2(\partial T)}\leq C\sum_{e\in\mathcal{E}_h^\Gamma} \sum_{T\in\mathcal{P}_e}\|\nabla v_h\|^2_{L^2(T)}\leq C\|v_h\|^2_h.
\end{aligned}
\end{equation}
From Lemma~\ref{lem_tiaoyue}, we obtain
\begin{equation}\label{pro_equ2}
\sum_{e\in\mathcal{E}_h^\Gamma}h^{-1}\| [v_h]_e\|^2_{L^2(e)}\leq C\sum_{e\in\mathcal{E}_h^\Gamma} \sum_{T\in\mathcal{P}_e}\|\nabla v_h\|^2_{L^2(T)}\leq C\|v_h\|^2_h.
\end{equation}
From Lemma~\ref{lem_stab_lift} for the local lifting operator and (\ref{pro_equ2}), we  arrive at
\begin{equation*}
s_h(v_h,v_h)\leq C\sum_{e\in\mathcal{E}_h^\Gamma}\left\|r_e([v_h]_e)\right\|^2_{L^2(\Omega)}\leq C\sum_{e\in\mathcal{E}_h^\Gamma}h^{-1} \|[v_h]_e\|^2_{L^2(e)}\leq C\|v_h\|^2_h,
\end{equation*}
which together with (\ref{def_nor2}), (\ref{pro_equ1}) and (\ref{pro_equ2}) yields the second inequality in (\ref{equali}).
\qed \end{proof}

The following lemma provides an optimal estimate for the interpolation error in terms of the norm $\interleave\cdot \interleave_h$.
\begin{lemma}\label{ener_app}
Suppose $v\in \widetilde{H}^2(\Omega)$ and the condition of Lemma~\ref{lem_unique} holds, then there exists a constant $C$ independent of $h$ and the interface location relative to the mesh such that
\begin{equation*}
\interleave v-I_h^{{\rm IFE}}v \interleave_h\leq Ch\|v\|_{H^2(\Omega^+\cup\Omega^-)}.
\end{equation*}
\end{lemma}
\begin{proof}:
 For any $e\in\mathcal{E}_h^\Gamma$, let $T_1$ and $T_2$ be two elements sharing  the edge $e$. Since $(v-I_h^{{\rm IFE}}v)|_{T}\in H^1(T)$ for all $T\in\mathcal{T}_h^\Gamma$, by the standard trace inequality, we have
\begin{equation}\label{pro_int1}
h^{-1}\| [v-I_h^{{\rm IFE}}v]_e\|^2_{L^2(e)}\leq Ch^{-2}\sum_{i=1,2}\|v-I_h^{{\rm IFE}}v\|^2_{L^2(T_i)}+C\sum_{i=1,2}|v-I_h^{{\rm IFE}}v|^2_{H^1(T_i)}.
\end{equation}
On the other hand, since $(v-I_h^{{\rm IFE}}v)|_{T}\in \widetilde{H}^2(T)+ S_h(T)$ for all $T\in\mathcal{T}_h^\Gamma$, by Lemma~\ref{lema_trace}, we have
\begin{equation}\label{pro_int2}
\begin{aligned}
h\|\{\beta_h\nabla &(v-I_h^{{\rm IFE}}v)\}_e\|^2_{L^2(e)}\leq Ch\sum_{ T\in\mathcal{P}_e   }\|\nabla (v-I_h^{{\rm IFE}}v) \|^2_{L^2(\partial T)}\\
&\leq C\sum_{i=1,2}\|\nabla (v-I_h^{{\rm IFE}}v)\|^2_{L^2(T_i)}+Ch^2\sum_{ i=1,2 }(|v|^2_{H^2(T_i^+\cup T_i^-)}+|v|^2_{H^1(T_i)}).
\end{aligned}
\end{equation}
From Lemma~\ref{lem_stab_lift} for the local lifting operator, we find
\begin{equation}\label{pro_int3}
\begin{aligned}
s_h(v-I_h^{{\rm IFE}}v,v-I_h^{{\rm IFE}}v)&\leq C\sum_{e\in\mathcal{E}_h^\Gamma}\left\|r_e([v-I_h^{{\rm IFE}}v]_e)\right\|^2_{L^2(\Omega)}\leq C\sum_{e\in\mathcal{E}_h^\Gamma}h^{-1}\|[v-I_h^{{\rm IFE}}v]_e\|^2_{L^2(e)}.
\end{aligned}
\end{equation}
Combining (\ref{def_nor2}), (\ref{pro_int1}), (\ref{pro_int2}) and (\ref{pro_int3}), we obtain
\begin{equation*}
\begin{aligned}
\interleave v-I_h^{{\rm IFE}}v \interleave _h^2&\leq Ch^{-2}\|v-I_h^{{\rm IFE}}v\|^2_{L^2(\Omega)}+C\sum_{T\in\mathcal{T}_h}|v-I_h^{{\rm IFE}}v|^2_{H^1(T)}+ Ch^2\|v\|_{H^2(\Omega^+\cup\Omega^-)},
\end{aligned}
\end{equation*}
which together with Theorem~\ref{theorem_interpolation} implies this lemma.
\qed \end{proof}

The following lemma concerns the consistent error caused by replacing $\beta(x)$ by $\beta_h(x)$.
\begin{lemma}\label{lem_consis}
Let $u$ and $u_h$ be the solutions to (\ref{p1.1})-(\ref{p1.7}) and (\ref{ph1})-(\ref{sh}), respectively.  Then it holds that
\begin{equation*}
A_h(u-u_h,v_h)= \sum_{T\in\mathcal{T}_h^\Gamma}\int_{T^\triangle}(\beta_h(x)-\beta(x))\nabla u\cdot\nabla v_hdx,  \qquad \forall v_h\in V_h^{{\rm IFE}},
\end{equation*}
where $T^\triangle$ is define in (\ref{tri}).

\end{lemma}
\begin{proof}:
Integrating by parts and summing up over all triangles in $\mathcal{T}_h$, we have, for any $v_h\in V_h^{IEF}$,
\begin{equation}\label{consis1}
\begin{aligned}
&\int_{\Omega}fv_hdx=\sum_{T\in\mathcal{T}_h}\int_T\beta(x)\nabla u\cdot\nabla v_hdx-\sum_{e\in\mathcal{E}^\Gamma_h}\int_{e}\beta(x)\nabla u\cdot\textbf{n}_e[v_h]ds\\
&=\sum_{T\in\mathcal{T}_h}\int_T\beta(x)\nabla u\cdot\nabla v_hdx-\sum_{e\in\mathcal{E}_h^\Gamma}\int_e\{\beta\nabla u\cdot\textbf{n}_e\}_e[v_h]_e+\{\beta\nabla v_h\cdot\textbf{n}_e\}_e[u]_eds,
\end{aligned}
\end{equation}
where we have used the facts that $u|_{\partial\Omega}=0$, $v_h|_{\partial\Omega}=0$, the function $u$ and its flux are continuous on all edges $\mathcal{E}_h$ since $u\in \widetilde{H}^2(\Omega)$, and $v_h\in V_h^{{\rm IFE}}$ is only discontinuous on interface edges $\mathcal{E}_h^\Gamma$.

Since $[u]_e=0$, for any $e\in\mathcal{E}_h^\Gamma$, we have $r_e([u]_e)=0$ and
\begin{equation}\label{cosis2}
s_h(u,v_h)=4\sum_{e\in\mathcal{E}^\Gamma_h}\int_{\Omega}\beta_h(x) r_e([u]_e)\cdot r_e([v_h]_e)dx=0,\qquad\forall v_h\in V_h^{{\rm IFE}}.
\end{equation}
Combining (\ref{ph1})-(\ref{sh}), (\ref{consis1}) and (\ref{cosis2}), and using the fact that $\beta_h(x)=\beta(x)$ on interface edges $\mathcal{E}_h^\Gamma$, we arrive at the desired identity,
\begin{equation*}
a_h(u-u_h,v_h)+s_h(u-u_h,v_h)=\sum_{T\in\mathcal{T}_h^\Gamma}\int_{T^\triangle}(\beta_h(x)-\beta(x))\nabla u\cdot\nabla v_hdx.
\end{equation*}
\qed \end{proof}

\vspace{-0.45cm}
We now provide the $H^1$ error estimate for the parameter free PPIFE method in the following theorem.
\begin{theorem}\label{theo_mainH1}
Let $u$ and $u_h$ be the solutions to (\ref{p1.1})-(\ref{p1.7}) and (\ref{ph1})-(\ref{sh}), respectively. Under the condition of Lemma~\ref{lem_unique}, there exists a constant $C$ independent of $h$ and the interface location relative to the mesh such that
\begin{equation}\label{h1_error}
\interleave u-u_h \interleave_h\leq Ch\|u\|_{H^2(\Omega^+\cup\Omega^-)}.
\end{equation}
\end{theorem}
\begin{proof}:
By using Lemma~\ref{lema_equ} for the equivalence of two norms, the coercivity (\ref{Coercivity})  and the continuity (\ref{conti}) of the bilinear form $A_h(\cdot,\cdot)$, we have
\begin{equation}\label{pro_error2}
\begin{aligned}
\interleave u_h&-I_h^{{\rm IFE}}u_h \interleave^2_h\leq C\|u_h-I_h^{{\rm IFE}}u_h\|^2_h\leq CA_h(u_h-I_h^{{\rm IFE}}u_h,u_h-I_h^{{\rm IFE}}u_h)\\
&=CA_h(u-I_h^{{\rm IFE}}u_h,u_h-I_h^{{\rm IFE}}u_h)+CA_h(u_h-u,u_h-I_h^{{\rm IFE}}u_h)\\
&\leq C\interleave u-I_h^{{\rm IFE}}u_h\interleave_h \interleave u_h-I_h^{{\rm IFE}}u_h \interleave_h+CA_h(u_h-u,u_h-I_h^{{\rm IFE}}u_h)
\end{aligned}
\end{equation}
From Lemma~\ref{lem_consis}, we  get
\begin{equation}\label{pro_error3}
\begin{aligned}
&\left|A_h(u_h-u,u_h-I_h^{{\rm IFE}}u_h)\right|\leq\sum_{T\in\mathcal{T}_h^\Gamma}\int_{T^\triangle}\left|(\beta_h-\beta)\nabla u\cdot\nabla (u_h-I_h^{{\rm IFE}}u_h)\right|ds\\
&\leq C\sum_{T\in\mathcal{T}_h^\Gamma}\int_{T^\triangle}\left|\nabla u\cdot\nabla (u_h-I_h^{{\rm IFE}}u_h)\right|ds\leq C\interleave u_h-I_h^{{\rm IFE}}u_h\interleave_h\left(\sum_{T\in\mathcal{T}_h^\Gamma}\|\nabla u\|^2_{L^2(T^\triangle)}\right)^{1/2}.
\end{aligned}
\end{equation}
From Lemma~\ref{lem_h3}, the global trace inequality on $\Omega^\pm$ and Lemma~\ref{lem_ext}, we  continue to derive the following,
\begin{equation}\label{pro_error4}
\begin{aligned}
\sum_{T\in\mathcal{T}_h^\Gamma} \|\nabla u\|^2_{L^2(T^\triangle)}&=\sum_{T\in\mathcal{T}_h^\Gamma}\sum_{i=\pm}\|\nabla u^i\|^2_{L^2(T^\triangle\cap T^i)}\leq\sum_{T\in\mathcal{T}_h^\Gamma}\sum_{i=\pm}\|\nabla \mathbb{E}^iu^i\|^2_{L^2(T^\triangle) }\\
&\leq C\sum_{T\in\mathcal{T}_h^\Gamma}\sum_{i=\pm} \left(h^2\|\nabla u^i\|^2_{L^2(T\cap \Gamma)}+h^4|\mathbb{E}^iu^i|^2_{H^2(T^\triangle)}\right)\\
&\leq Ch^2\sum_{i=\pm}\| \nabla u^i \|^2_{L^2(\Gamma)}+Ch^4\sum_{i=\pm}| \mathbb{E}^i u^i|^2_{H^2(\Omega)}\\
&\leq Ch^2\sum_{i=\pm} \|u \|^2_{H^2(\Omega^i)}=Ch^2\|u\|^2_{H^2(\Omega^+\cup \Omega^-)}.
\end{aligned}
\end{equation}
Substituting (\ref{pro_error3}) and (\ref{pro_error4}) into (\ref{pro_error2}), we obtain
\begin{equation*}
\interleave u_h-I_h^{{\rm IFE}}u_h \interleave_h\leq C\interleave u-I_h^{{\rm IFE}}u_h\interleave_h +Ch\|u\|_{H^2(\Omega^+\cup \Omega^-)}.
\end{equation*}
Thus, by the triangle inequality and Lemma~\ref{ener_app}, we  arrive at
\begin{equation*}
\begin{aligned}
\interleave u- u_h \interleave_h&\leq  \interleave u-I_h^{{\rm IFE}}u_h \interleave_h+\interleave u_h-I_h^{{\rm IFE}}u_h \interleave_h\\
&\leq  C\interleave u-I_h^{{\rm IFE}}u_h \interleave_h+Ch\|u\|_{H^2(\Omega^+\cup \Omega^-)}\leq Ch\|u\|_{H^2(\Omega^+\cup \Omega^-)},
\end{aligned}
\end{equation*}
which completes the proof of the theorem.
\qed \end{proof}

Finally, we show the optimal $L^2$ error estimate for the  parameter free PPIFE method using the standard duality argument.
\begin{theorem}\label{theo_mainL2}
Let $u$ and $u_h$ be the solutions to (\ref{p1.1})-(\ref{p1.7}) and (\ref{ph1})-(\ref{sh}), respectively. Under the condition of Lemma~\ref{lem_unique}, there exists a constant $C$ independent of $h$ and the interface location relative to the mesh such that
\begin{equation*}
\|u-u_h\|_{L^2(\Omega)}\leq Ch^2\|u\|_{H^2(\Omega^+\cup \Omega^-)}.
\end{equation*}
\end{theorem}
\begin{proof}:
Let $z$ be the solution of the following auxiliary problem
\begin{equation}\label{auxi}
\begin{aligned}
&-\nabla\cdot(\beta(x)\nabla z)=u-u_h\qquad \mbox{ in }\Omega\backslash\Gamma,\\
&[z]_\Gamma=0,~[\beta\nabla z\cdot \textbf{n}]_\Gamma=0\qquad ~\mbox{ on }\Gamma,\\
&z=0\qquad \qquad\qquad\qquad\qquad~~\mbox{ on }\partial \Omega.
\end{aligned}
\end{equation}
Since $u-u_h\in L^2(\Omega)$, it follows from Theorem~\ref{theo_reg} that
\begin{equation}\label{reg_L2}
z\in \widetilde{H}^2(\Omega)~~\mbox{ and }~~\|z\|_{H^2(\Omega^+\cup\Omega^-)}\leq C\|u-u_h\|_{L^2(\Omega)}.
\end{equation}
Multiplying (\ref{auxi}) by $u-u_h$ and integrating by parts, we find
\begin{equation*}
\begin{aligned}
\|u-u_h\|_{L^2(\Omega)}^2&=\sum_{T\in\mathcal{T}_h}\int_T-\nabla\cdot(\beta(x)\nabla z)(u-u_h)dx\\
&=\sum_{T\in\mathcal{T}_h}\int_T\beta\nabla z\cdot\nabla(u-u_h)dx-\sum_{e\in\mathcal{E}_h^\Gamma}\int_e\{\beta\nabla z\cdot\textbf{n}_e\}_e[u-u_h]_eds.
\end{aligned}
\end{equation*}
Using the facts that $[z]_e=0$  and $s_h(z,u-u_h)=0$, we have
\begin{equation}\label{pro_l2_1}
\begin{aligned}
\|u-u_h&\|_{L^2(\Omega)}^2=\sum_{T\in\mathcal{T}_h}\int_T\beta\nabla z\cdot\nabla(u-u_h)dx+s_h(z,u-u_h)\\
&\qquad\qquad-\sum_{e\in\mathcal{E}_h^\Gamma}\int_e\left(\{\beta\nabla z\cdot\textbf{n}_e\}_e[u-u_h]_e+\{\beta\nabla (u-u_h)\cdot\textbf{n}_e\}_e[z]_e\right)ds\\
&=A_h(z,u-u_h)+\sum_{T\in\mathcal{T}_h}\int_T(\beta-\beta_h)\nabla z\cdot \nabla(u-u_h)dx\\
&=A_h(z-I_h^{{\rm IFE}}z,u-u_h)+A_h(I_h^{{\rm IFE}}z,u-u_h)+\sum_{T\in\mathcal{T}_h}\int_T(\beta-\beta_h)\nabla z\cdot \nabla(u-u_h)dx.
\end{aligned}
\end{equation}
The first term of (\ref{pro_l2_1})  is bounded as shown below,
\begin{equation}\label{pro_l2_2}
\begin{aligned}
\left|A_h(z-I_h^{{\rm IFE}}z,u-u_h)\right|&\leq C \interleave z-I_h^{{\rm IFE}}z\interleave_h\interleave u-u_h \interleave_h\leq Ch^2\|z\|_{H^2(\Omega^+\cup\Omega^-)}\|u\|_{H^2(\Omega^+\cup\Omega^-)},
\end{aligned}
\end{equation}
where we have used (\ref{conti}) in the first inequality, and Lemma~\ref{ener_app} and Theorem~\ref{theo_mainH1} in the  second inequality.
From Lemma~\ref{lem_consis} and the symmetry of the bilinear form $A_h(\cdot,\cdot)$, we estimate the second term of (\ref{pro_l2_1})  below,
\begin{equation*}
\begin{aligned}
&\left|A_h(I_h^{{\rm IFE}}z,u-u_h)\right|\leq \sum_{T\in\mathcal{T}_h^\Gamma}\int_{T^\triangle}\left|(\beta_h-\beta)\nabla u\cdot\nabla I_h^{{\rm IFE}}z\right|dx\\
&\leq C\sum_{T\in\mathcal{T}_h^\Gamma}\|\nabla u\|_{L^2(T^\triangle)}\|\nabla I_h^{{\rm IFE}}z-z\|_{L^2(T^\triangle)}+C\sum_{T\in\mathcal{T}_h^\Gamma}\|\nabla u\|_{L^2(T^\triangle)}\|\nabla z\|_{L^2(T^\triangle)}\\
&\leq C\left(\interleave I_h^{{\rm IFE}}z-z \interleave_h+\left(\sum_{T\in\mathcal{T}_h^\Gamma}\|\nabla z\|^2_{L^2(T^\triangle)} \right)^{1/2} \right)\left(\sum_{T\in\mathcal{T}_h^\Gamma}\|\nabla u\|^2_{L^2(T^\triangle)} \right)^{1/2}\\
&\leq Ch^2\|z\|_{H^2(\Omega^+\cup\Omega^-)}\|u\|_{H^2(\Omega^+\cup\Omega^-)}+Ch\|u\|_{H^2(\Omega^+\cup\Omega^-)}\left(\sum_{T\in\mathcal{T}_h^\Gamma}\|\nabla z\|^2_{L^2(T^\triangle)} \right)^{1/2},
\end{aligned}
\end{equation*}
where we have used (\ref{pro_error4}) and Lemma~\ref{ener_app} in the last inequality.
Similar to (\ref{pro_error4}), we also  have
\begin{equation}\label{wwww}
\sum_{T\in\mathcal{T}_h^\Gamma}\|\nabla z\|^2_{L^2(T^\triangle)} \leq Ch^2\|z\|^2_{H^2(\Omega^+\cup\Omega^-)},
\end{equation}
which leads to
\begin{equation}\label{pro_l2_3}
\left|A_h(I_h^{{\rm IFE}}z,u-u_h)\right|\leq Ch^2\|z\|_{H^2(\Omega^+\cup\Omega^-)}\|u\|_{H^2(\Omega^+\cup\Omega^-)}.
\end{equation}
Next, the third term of (\ref{pro_l2_1}) can be estimated below
\begin{equation}\label{pro_l2_4}
\begin{aligned}
&\left|\sum_{T\in\mathcal{T}_h}\int_{T}(\beta-\beta_h)\nabla z\cdot \nabla(u-u_h)dx\right|\leq C\sum_{T\in\mathcal{T}_h^\Gamma}\|\nabla z\|_{L^2(T^\triangle)}\|\nabla (u-u_h)\|_{L^2(T^\triangle)}\\
&\leq  \interleave u-u_h\interleave_h \left(\sum_{T\in\mathcal{T}_h^\Gamma}\|\nabla z\|^2_{L^2(T^\triangle)}\right)^{1/2}\leq Ch^2\|z\|_{H^2(\Omega^+\cup\Omega^-)}\|u\|_{H^2(\Omega^+\cup\Omega^-)},
\end{aligned}
\end{equation}
where we have used (\ref{wwww}) and Lemma~\ref{ener_app} in the last inequality.
Substituting (\ref{pro_l2_2}), (\ref{pro_l2_3}) and (\ref{pro_l2_4}) into (\ref{pro_l2_1})  and using the regularity result (\ref{reg_L2}), we  finally arrive at
\begin{equation*}
\|u-u_h\|_{L^2(\Omega)}\leq Ch^2\|u\|_{H^2(\Omega^+\cup\Omega^-)},
\end{equation*}
which completes the proof of the theorem.
\qed \end{proof}

\section{Extension to the interface problem with variable coefficients}\label{sec_var}

In this section, we consider the interface problem with variable coefficients, i.e.,
\begin{align*}
\beta(x)=\beta^+(x) ~ \mbox{ if }~ x\in\Omega^+ ~~~\mbox{ and } ~~~\beta(x)=\beta^-(x) ~ \mbox{ if } ~x\in\Omega^-,
\end{align*}
where $\beta^\pm(x)$ are defined in slight larger domains $\Omega^\pm_e:=\Omega^\pm\cup N(\Gamma,\delta_0)$. We assume 
$\beta^i\in C^1(\overline{\Omega_e^i})$, $i=+,-$. Thus, there exist positive constants $\beta_{min}$, $\beta_{max}$ and $C_\beta$ such that 
\begin{equation}\label{new_beta_inequality}
\beta_{min}\leq\beta^\pm(x)\leq\beta_{max}, ~\forall x\in\Omega_e^\pm \quad \mbox{ and }\quad \|\nabla \beta^\pm\|_{\Omega^\pm_e}\leq C_\beta.
\end{equation}

On an interface element $T\in\mathcal{T}_h^\Gamma$, we use some sort of  averages of the coefficients  $\overline{\beta}^+_h$ and $\overline{\beta}^-_h$  over the sub-region such that
\begin{equation}\label{vari_yaoqiu}
\|\overline{\beta}^+_h-\beta^-(x)\|_{L^\infty(T)}\leq Ch ~~\mbox{ and } ~~\|\overline{\beta}^-_h-\beta^-(x)\|_{L^\infty(T)}\leq Ch,
\end{equation}
where the constant $C$ may depend on $\| \beta^+\|_{W^1_\infty(T)}$ and $\| \beta^-\|_{W^1_\infty(T)}$.  For example, we can choose $\overline{\beta}^\pm_h=\beta^\pm(x_m)$ with an arbitrary point $x_m\in T$.


\textbf{Modifications to the parameter free PPIFE method for variable coefficients.}
The coefficients $\beta_h(x)$ in the local lifting operator (\ref{def_lift}) and the method (\ref{ph1})-(\ref{sh}) are replaced by
\begin{equation*}
\beta_h(x)=\beta^+(x) \mbox{ if } x\in \Omega^+_h ~~\mbox{ and }~~ \beta_h(x)=\beta^-(x) \mbox{ if } x\in \Omega^-_h.
\end{equation*}
In the construction of the IFE space, we replace the third equation in (\ref{ifem_modi}) by
\begin{equation}\label{modi_flux}
\overline{\beta}^+_h\nabla \phi^+\cdot \textbf{n}_h=\overline{\beta}^-_h\nabla \phi^-\cdot \textbf{n}_h.
\end{equation}
The constants $\beta^\pm$ in $W_h(T)$ are also replaced by $\overline{\beta}_h^\pm$.
The coefficients (\ref{lift_coet1}) and (\ref{lift_coet2}) for the local lifting operator now are computed from the following,
\begin{equation*}
c_1=\frac{\textbf{t}_{1,h}\cdot \textbf{n}_e \int_e\beta_h \varphi ds}{2\int_{T_1}\beta_hdx},~d_1=\frac{\textbf{n}_{1,h}\cdot \textbf{n}_e\left(\overline{\beta}_h^-\int_{e^+}\beta_h\varphi ds+\overline{\beta}_h^+\int_{e^-}\beta_h\varphi ds\right)}{2(\overline{\beta}_h^-)^2\int_{T_{1,h}^+}\beta_hdx+2(\overline{\beta}_h^+)^2\int_{T_{1,h}^-}\beta_hdx},
\end{equation*}
\begin{equation*}
c_2=\frac{\textbf{t}_{1,h}\cdot \textbf{n}_e \int_e\beta_h \varphi ds}{2\int_{T_2}\beta_hdx},~d_2=\frac{\textbf{n}_{1,h}\cdot \textbf{n}_e\left(\overline{\beta}_h^-\int_{e^+}\beta_h\varphi ds+\overline{\beta}_h^+\int_{e^-}\beta_h\varphi ds\right)}{2(\overline{\beta}_h^-)^2\int_{T_{2,h}^+}\beta_hdx+2(\overline{\beta}_h^+)^2\int_{T_{2,h}^-}\beta_hdx},
\end{equation*}
where $e^\pm=e\cap \Omega_h^\pm$.

\textbf{Modifications to the analysis}.
On  an  interface element $T\in\mathcal{T}_h^\Gamma$, we define a function $\overline{\beta}_h(x)$ such that
\begin{equation*}
\overline{\beta}_h(x)=\overline{\beta}_h^+~~\mbox{ if }~~ x\in T^+_h ~~~~\mbox{ and }~~~~ \overline{\beta}_h(x)=\overline{\beta}_h^- ~\mbox{ if }~ x\in T^-_h .
\end{equation*}

First we consider the modification  to the proof of  Lemma~\ref{lema_trace} for the trace inequality.
From (\ref{modi_flux}), the second equality in (\ref{conti_linear}) becomes
\begin{equation*}
[\overline{\beta}_h\nabla\phi\cdot \textbf{n}_{h}]_{\Gamma_{h}\cap T}=0\quad \mbox{ on }\Gamma_h\cap T.
\end{equation*}
Thus, we only need to replace (\ref{lem1_p2_2}) and (\ref{trace_he}) by
\begin{equation*}
\overline{\beta}_h(x)\nabla w_h\cdot\textbf{n}_h\in H^1(T)\mbox{ and } \beta\nabla v\cdot \textbf{n}-\overline{\beta}_h\nabla w_h\cdot \textbf{n}_h\in H^1(T),
\end{equation*}
and the remaining proof process is the same.

Next we consider modifications  to the proof of Lemma~\ref{chazhi_error} for the interpolation error estimation.
In the construction of  the auxiliary functions, we  change the corresponding identities in (\ref{jum_ji1_cons}) and (\ref{def_psi_linear2}) to
\begin{equation*}
\overline{\beta}^+_h\nabla\Upsilon^+\cdot \textbf{n}_h-\overline{\beta}^-_h\nabla\Upsilon^-\cdot \textbf{n}_h=1,\qquad\overline{\beta}^+_h\nabla\Psi_i^+\cdot \textbf{n}_h=\overline{\beta}^-_h\nabla\Psi_i^-\cdot \textbf{n}_h,\quad i=D,E.
\end{equation*}
Now the result (\ref{lemma4}) in Lemma~\ref{lema_fenjie} and the inequality (\ref{main_prof1}) in Lemma~\ref{chazhi_error} become
\begin{equation*}
\begin{aligned}
I_h^{BK}v-I_h^{{\rm IFE}}v&=[\![ I_h^{BK}v]\!](D)\Psi_D(x)+[\![I_h^{BK}v]\!](E)\Psi_E(x)+[\![\overline{\beta}_h \nabla (I_h^{BK}v)\cdot \textbf{n}]\!](x^*)\Upsilon(x)
\end{aligned}
\end{equation*}
 and 
\begin{equation}\label{vari_1}
|I_h^{BK}v-I_h^{{\rm IFE}}v|^2_{H^m(T_h^+\cup T_h^-)}\leq Ch^{2-2m}\sum_{x_p=D,E}[\![ I_h^{BK}v]\!]^2(x_p)+Ch^{4-2m}[\![\overline{\beta}_h \nabla (I_h^{BK}v)\cdot \textbf{n}]\!]^2(x^*).
\end{equation}
Thus, the estimate (\ref{pro_main4}) in the proof of Lemma~\ref{chazhi_error} needs to be changed as
\begin{equation}\label{vari_2}
\begin{aligned}
|[\![\overline{\beta}_h \nabla& (I_h^{BK}v)\cdot \textbf{n}]\!]^2(x^*)\leq \|[\![\overline{\beta}_h\nabla (I_h^{BK}v)\cdot \textbf{n}_h]\!] \|_{L^\infty(T)}^2 \leq Ch^{-2} \|[\![\overline{\beta}_h \nabla (I_h^{BK}v)\cdot \textbf{n}_h]\!] \|_{L^2(T)}^2\\
&\leq Ch^{-2}  \|[\![\beta \nabla (I_h^{BK}v)\cdot \textbf{n}_h]\!] \|_{L^2(T)}^2+ Ch^{-2} \|[\![(\beta-\overline{\beta}_h) \nabla (I_h^{BK}v)\cdot \textbf{n}_h]\!] \|_{L^2(T)}^2\\
\end{aligned}
\end{equation}
The estimation of the first term on the right-hand side is the same as that in (\ref{pro_main4}). For the second term,  using (\ref{vari_yaoqiu}), we have
\begin{equation}\label{vari_3}
\begin{aligned}
 &h^{-2}\|[\![(\beta-\overline{\beta}_h) \nabla (I_h^{BK}v)\cdot \textbf{n}_h]\!] \|_{L^2(T)}^2\\
&\leq  Ch^{-2}\sum_{i=\pm}\|\beta^i-\overline{\beta}^i_h\|^2_{L^\infty(T)} \| \nabla I_h \mathbb{E}^iv^i\|_{L^2(T)}^2\leq C\sum_{i=\pm}\| \nabla I_h \mathbb{E}^iv^i\|_{L^2(T)}^2\\
&\leq C\sum_{i=\pm}\left(\|\nabla(\mathbb{E}^iv^i-I_h \mathbb{E}^iv^i)\|_{L^2(T)}^2+\|\nabla\mathbb{E}^iv^i\|_{L^2(T)}^2\right)\leq C\sum_{i=\pm}|\mathbb{E}^iv^i|_{H^1(T)}^2.
\end{aligned}
\end{equation}
Combining (\ref{vari_1})-(\ref{vari_3}) and (\ref{pro_main4}), we can see that (\ref{pro_main_ls}) also holds. Using (\ref{new_beta_inequality}),  we find that  (\ref{pro_main_new}) is also correct.  Hence, the optimal interpolation error estimates are not affected and  Lemma~\ref{chazhi_error} is  still  valid.

The other lemmas and theorems can be  adapted to the variable case easily. Note that for Theorem~\ref{theo_mainL2} about  the $L^2$ error estimate, we need  the regularity result  (\ref{reg_L2})  for the problem with variable coefficients which was proved in \cite{2012Uniform}.

\section{Extension to three dimensions}\label{sec_3D}
In this section, we extend our method and analysis to three dimensions. For simplicity, we also assume that $\beta(x)$  is a piecewise constant. 
Let now $\mathcal{T}_h$ be a simplicial triangulation of $\Omega$ and $\mathcal{E}_h$ be all the faces of the mesh $\mathcal{T}_h$. We assume that the triangulation is quasi-uniform.  Different from the 2D cases, the  points of intersection of the interface and the edges of a tetrahedron are usually not coplanar.  The linear approximation of the interface determined by three of these intersection points on each interface element may not be continuous across interface faces (see \cite{kafafy2005three, Guojcp2020}). Thus, we should construct the IFE space according to the exact interface.  However, a discrete interface is also needed for the purpose of analysis.
On an interface element $T\in\mathcal{T}_h^\Gamma$, let $x^*$ be a fixed point on $\Gamma\cap \overline{T}$ and $\Gamma_{h,T}^{ext}$ be the plane which is tangent to $\Gamma$ at $x^*$. Then we  have $\mathbf{n}_h=\mathbf{n}(x^*)$ on this interface element. In practical implementations, we can choose $x^*$ as one of the  intersection points of the interface and the edges of the interface element.
 The discrete interface now is  defined as  $\Gamma_h=\bigcup_{T\in\mathcal{T}_h^\Gamma}\Gamma_{h,T}$ with $\Gamma_{h,T}=\Gamma_{h,T}^{ext}\cap T$. 
As $\Gamma \in C^2$,  similar to the 2D cases, if $h\leq h_0$,  it holds
\begin{equation}\label{rela_3d_nnh2}
\|{\rm dist}(x,\Gamma_{h,T}^{ext})\|_{L^\infty(\Gamma\cap T)}\leq Ch^2,~~\|\mathbf{n}-\mathbf{n}_h\|_{L^\infty(T)}\leq Ch,~~~ \forall T\in \mathcal{T}_h^\Gamma.
\end{equation}

\textbf{Modifications to the parameter free PPIFE method in 3D.}
The local IFE space on an interface element $T$ is then defined by
\begin{equation}\label{def_sh3d}
\begin{aligned}
S_h(T):=\{&\phi\in L^2(T) : \phi|_{T^\pm}=\phi^\pm|_{T^\pm},~ \forall\phi^\pm\in\mathbb{P}_1(T) \mbox{ satisfying }[\![\phi]\!]|_{\Gamma_{h,\Gamma}^{ext}}=0,~~[\![\beta\nabla\phi\cdot\mathbf{n}_h]\!]=0\}.
\end{aligned}
\end{equation}
We note that the condition $[\![\phi]\!]|_{\Gamma_{h,\Gamma}^{ext}}=0$ in the above definition is equivalent to 
\begin{equation}\label{equa3Dth}
[\![\phi]\!](x^*)=0\mbox{ and }[\![\nabla \phi\cdot \mathbf{t}_{i,h}]\!]=0, ~i=1,2,
\end{equation}
where $\mathbf{t}_{1,h}$ and $\mathbf{t}_{2,h}$ are standard basis vectors in the plane $\Gamma_{h,\Gamma}^{ext}$.

\begin{lemma}\label{lem_unique3D}
Let $A_i$, $i=1,2,3,4$ be vertices of an interface element $T\in\mathcal{T}_h^\Gamma$,  $\alpha_{max}$ be the maximum angle of  all faces of  the tetrahedron, and  $\gamma_{max}$ be the maximum dihedral angle of the tetrahedron. If $\alpha_{max}\leq\pi/2$ and $\gamma_{max}\leq\pi/2$, then the function $\phi\in S_h(T)$ defined in (\ref{def_sh3d}) is uniquely determined by $\phi(A_i)$, $i=1,2,3,4$, regardless of the interface location.
\end{lemma}
\begin{proof}:
See  Appendix \ref{pro_lem_unique3D}. \qed
\end{proof}

We emphasize that, different from the  2D cases, the discrete interface $\Gamma_h$ now is not continuous.  
If the IFE space is defined according to the discrete interface $\Gamma_h$, then the IFE functions are not well defined on interface faces due to the discontinuity of $\Gamma_h$.  The same issue exists for $\beta_h(x)$ defined in (\ref{def_bth}).

We replace the coefficient $\beta_h(x)$ in the definition of the local lifting operator (\ref{def_lift}) and (\ref{ph1})-(\ref{sh}) of the algorithm by the exact coefficient $\beta(x)$. We also replace the constant $4$ in the bilinear form $s_h(\cdot,\cdot)$ (see (\ref{sh})) by $8$ to ensure the the coercivity (\ref{Coercivity}) since each interface element is now calculated at most four times when estimating the integrals in the left-hand side of  (\ref{pro_coer2}). Although the IFE function $v_h$ is discontinuous across the interface, we also use the notation $\nabla v_h$ for the piecewise gradient for simplicity.

\textbf{Modifications to the analysis in 3D}.
On an interface element with vertices $A_i$, $i=1,2,3,4$, we define the auxiliary functions $\Psi(x)$, $\Upsilon(x)$ and $\Theta_i$, $i=1,2$ as
\begin{equation}
\Psi|_{T^\pm}=\Psi^\pm|_{T^\pm},~ \Psi^\pm\in \mathbb{P}_1(T),~~\Upsilon|_{T^\pm}=\Upsilon^\pm|_{T^\pm},~ \Upsilon^\pm\in \mathbb{P}_1(T),~~\Theta_i|_{T^\pm}=\Theta_i^\pm|_{T^\pm},~ \Theta_i^\pm\in \mathbb{P}_1(T),
\end{equation}
such that
$$
\Psi(A_j)=\Upsilon(A_j)=\Theta_i(A_j)=0,~j=1,2,3,4,
$$
and
\begin{equation}
\begin{aligned}
&[\![\Psi]\!](x^*)=1,~~&&[\![\beta\nabla\Psi\cdot \textbf{n}_h]\!]=0,~~&&[\![\nabla\Psi\cdot \textbf{t}_{j,h}]\!]=0,~&& j=1,2,\\
&[\![\Upsilon]\!](x^*)=0,~~&&[\![\beta\nabla\Upsilon\cdot \textbf{n}_h]\!]=1,~~&&[\![\nabla\Upsilon\cdot \textbf{t}_{j,h}]\!]=0,~ &&j=1,2,\\
&[\![\Theta_i]\!](x^*)=0,~~&&[\![\beta\nabla\Theta_i\cdot \textbf{n}_h]\!]=0,~~&&[\![\nabla\Theta_i\cdot \textbf{t}_{j,h}]\!]=\delta_{ij},~ &&j=1,2.
\end{aligned}
\end{equation}

\begin{lemma}\label{lem_jiest3D}
Under the conditions of Lemma~\ref{lem_unique3D}, the auxiliary functions defined above satisfy the following estimates:
\begin{equation*}
|\Psi|^2_{H^m(T^+\cup T^-)}\leq Ch^{3-2m}, ~|\Upsilon|^2_{H^m(T^+\cup T^-)}\leq Ch^{5-2m}, ~|\Theta_i|^2_{H^m(T^+\cup T^-)}\leq Ch^{5-2m}, ~i=1,2,~m=0,1,
\end{equation*}
where the constant $C$ is independent of $h$ and the interface location relative to the mesh.
\end{lemma}
\begin{proof}:
See  Appendix \ref{pro_lem_jiest3D}. \qed
\end{proof}

Since the IFE space and the auxiliary functions are defined according to the exact interface, we replace $T_h^+$ and $T_h^-$ by $T^+$ and $T^-$ in the definition of the operator $I_h^{BK}$ in (\ref{ehpm}), i.e.,
\begin{equation}
(I_h^{BK}v)|_{T^i}=I_h\mathbb{E}^i v^i, \quad i=+,-, \quad\forall T\in\mathcal{T}_h^\Gamma, \quad \forall v\in\widetilde{H}^2(\Omega).
\end{equation}
 Similar to Lemma~\ref{lema_fenjie}, we have
 the following identity on interface elements,
\begin{equation}\label{decomp_3D}
I_h^{BK}v-I_h^{{\rm IFE}}v=[\![ I_h^{BK}v]\!](x^*)\Psi(x)+[\![\beta \nabla (I_h^{BK}v)\cdot \textbf{n}_h]\!]\Upsilon(x)+\sum_{i=1,2}[\![\beta \nabla (I_h^{BK}v)\cdot \textbf{t}_{i,h}]\!]\Theta_i(x).
\end{equation}
We emphasize that the relation (\ref{equa3Dth}) has been utilized to prove the above decomposition. The optimal interpolation error estimates are proved in the following theorem.
\begin{theorem}\label{chazhi_error3D}
For any $v\in  \widetilde{H}^2(\Omega)$, under the conditions of Lemma~\ref{lem_unique3D}, there exists a constant $C$ independent of $h$ and the interface location relative to the mesh such that
\begin{equation*}
\sum_{T\in\mathcal{T}_h}|v-I_h^{{\rm IFE}}v|^2_{H^m(T^+\cup T^-)}\leq Ch^{4-2m}\|v\|^2_{H^2(\Omega^+\cup\Omega^-)},\quad m=0,1.
\end{equation*}
\end{theorem}
\begin{proof}:
The proof is along the same lines as that of Lemma~\ref{chazhi_error}. It suffices to consider an interface element $T\in\mathcal{T}_h^\Gamma$.
The triangle inequality leads to 
\begin{equation*}
|v-I_h^{{\rm IFE}}v|_{H^m(T^+\cup T^-)}\leq |v-I_h^{BK}v|_{H^m(T^+\cup T^-)}+|I_h^{BK}v-I_h^{{\rm IFE}}v|_{H^m(T^+\cup T^-)}.
\end{equation*}
The estimate of the first term is similar to (\ref{main_prof12}). For the second term, using (\ref{decomp_3D}) and  Lemma~\ref{lem_jiest3D} we have
\begin{equation}\label{inter_pro1_3D}
\begin{aligned}
&|I_h^{BK}v-I_h^{{\rm IFE}}v|^2_{H^m(T^+\cup T^-)}\\
&\quad\leq Ch^{3-2m}[\![ I_h^{BK}v]\!]^2(x^*)+Ch^{5-2m}[\![\beta \nabla (I_h^{BK}v)\cdot \textbf{n}_h]\!]^2+Ch^{5-2m}\sum_{i=1,2}[\![\nabla (I_h^{BK}v)\cdot \textbf{t}_{i,h}]\!]^2.
\end{aligned}
\end{equation}
Since  $[\![v]\!]^2(x^*)=0$, it holds 
\begin{equation}\label{inter_pro2_3D}
\begin{aligned}
[\![ I_h^{BK}v]\!]^2(x^*)=&[\![ I_h^{BK}v-v]\!]^2(x^*)\leq C\|[\![I_h^{BK}v-v]\!]\|_{L^\infty(T)}^2\\
&\leq C\sum_{i=\pm}\|I_h\mathbb{E}^iv^i-\mathbb{E}^iv^i\|^2_{L^\infty(T)}\leq Ch\sum_{i=\pm}|\mathbb{E}^iv^i|^2_{H^2(T)}.
\end{aligned}
\end{equation}
Similar to (\ref{pro_main4}), using the second inequality in  (\ref{rela_3d_nnh2}) we have
\begin{equation}\label{inter_pro3_3D}
[\![\beta \nabla (I_h^{BK}v)\cdot \textbf{n}_h]\!]^2\leq Ch^{-1}\sum_{i=\pm}\left(|\mathbb{E}^iv^i|^2_{H^2(T)}+|\mathbb{E}^iv^i|^2_{H^1(T)}\right) +Ch^{-3}\left\|[\![\beta \nabla v\cdot \textbf{n}]\!]\right\|_{L^2(T)}^2.
\end{equation}
To estimate the third term on the right-hand side of (\ref{inter_pro1_3D}), we need tangential gradients $\nabla_\Gamma$ and $\nabla_{\Gamma_h}$ which are defined by
\begin{equation*}
(\nabla_\Gamma z)(x):=\nabla z-(\mathbf{n}\cdot\nabla z)\mathbf{n}, \quad(\nabla_{\Gamma_h} z)(x):=\nabla z-(\mathbf{n}_h\cdot\nabla z)\mathbf{n}_h,  ~ \forall x\in T, ~ \forall z\in H^1(T), ~\forall T\in\mathcal{T}_h^\Gamma.
\end{equation*}
By definition, it holds 
\begin{equation}\label{rela_grad_gamma}
\begin{aligned}
&|\nabla z\cdot\mathbf{t}_{i,h}|\leq |\nabla_{\Gamma_h} z|,~~~|\nabla_{\Gamma_h}z|\leq |\nabla z|,\\
&\|\nabla_{\Gamma_h}z-\nabla_{\Gamma}z\|_{L^2(T)} =\|(\mathbf{n}\cdot \nabla z)\mathbf{n}-(\mathbf{n}_h\cdot \nabla z)\mathbf{n}_h\|_{L^2(T)} \\
&~~~\qquad\qquad\qquad\qquad=\|(\mathbf{n}\cdot \nabla z)(\mathbf{n}-\mathbf{n}_h)-((\mathbf{n}-\mathbf{n}_h)\cdot \nabla z)\mathbf{n}_h\|_{L^2(T)} \\
&~~~\qquad\qquad\qquad\qquad\leq C\|\mathbf{n}-\mathbf{n}_h\|_{L^\infty(T)}\|\nabla z\|_{L^2(T)}.
\end{aligned}
\end{equation}
Using the above inequalities and the second inequality in  (\ref{rela_3d_nnh2}) we  derive 
\begin{equation}\label{inter_pro4_3D}
\begin{aligned}
[\![\nabla (I_h^{BK}v)&\cdot \textbf{t}_{i,h}]\!]^2\leq [\![ \nabla_{\Gamma_h} (I_h^{BK}v)]\!]^2\leq Ch^{-3}\|[\![ \nabla_{\Gamma_h} (I_h^{BK} v )]\!]\|^2_{L^2(T)}\\
&\leq Ch^{-3}\|[\![ \nabla_{\Gamma_h} (I_h^{BK} v -v)+(\nabla_{\Gamma_h}-\nabla_{\Gamma}+\nabla_{\Gamma})v^\pm]\!]\|^2_{L^2(T)}\\
&\leq Ch^{-3}\left( \|[\![ \nabla_{\Gamma_h} (I_h^{BK} v-v)]\!]\|^2_{L^2(T)}+\|[\![\nabla_{\Gamma_h}v-\nabla_{\Gamma}v]\!]\|^2_{L^2(T)}+\|[\![\nabla_{\Gamma}v]\!]\|^2_{L^2(T)}\right)\\
&\leq Ch^{-3}\left( \|[\![ \nabla (I_h^{BK} v-v)]\!]\|^2_{L^2(T)}+\|\mathbf{n}-\mathbf{n}_h\|^2_{L^\infty(T)}\|[\![\nabla v]\!]\|^2_{L^2(T)}+\|[\![\nabla_{\Gamma}v]\!]\|^2_{L^2(T)}\right)\\
&\leq Ch^{-1}\sum_{i=\pm}\left(|\mathbb{E}^iv^i|^2_{H^2(T)}+|\mathbb{E}^iv^i|^2_{H^1(T)}\right) +Ch^{-3}\left\|[\![ \nabla_{\Gamma} v]\!]\right\|_{L^2(T)}^2.
\end{aligned}
\end{equation}
Substituting (\ref{inter_pro2_3D}), (\ref{inter_pro3_3D}) and (\ref{inter_pro4_3D}) into (\ref{inter_pro1_3D}) yields  
\begin{equation*}
|I_h^{BK}v-I_h^{{\rm IFE}}v|^2_{H^m(T^+\cup T^-)}\leq Ch^{4-2m}\sum_{i=\pm}\|\mathbb{E}^iv^i\|^2_{H^2(T)} +Ch^{2-2m}\left(\left\|[\![\beta \nabla v\cdot \textbf{n}]\!]\right\|_{L^2(T)}^2+\left\|[\![\nabla_\Gamma v]\!]\right\|_{L^2(T)}^2\right).
\end{equation*}
Using the facts $[\![\nabla_\Gamma v]\!]\in \left(H^1(N(\Gamma,\delta_0))\right)^3$ and  $[\![\nabla_\Gamma v]\!]|_{\Gamma}=\mathbf{0}$, the remaining proof is the same as that of Lemma~\ref{chazhi_error} since Lemma~\ref{strip} also holds for the 3D cases.\qed
\end{proof}

The trace inequality (\ref{le1}) in Lemma~\ref{lema_trace} also holds.  Note that for any $z\in H^1(T)$, we have 
\begin{equation}\label{decomp_nhth}
\nabla z=(\nabla_{\Gamma_h} z)+(\mathbf{n}_h\cdot\nabla z)\mathbf{n}_h.
\end{equation}
Then,  the split (\ref{lem_tra}) in the proof is changed to
\begin{equation}
\begin{aligned}
\|\nabla(v-w_h)\|_{L^2(\partial T)}&\leq C \|\nabla_{\Gamma_h} (v-w_h)\|_{L^2(\partial T)}+C\|\beta \nabla(v-w_h)\cdot\textbf{n}_h\|_{L^2(\partial T)}\\
&\leq C\left(\|\nabla_\Gamma v-\nabla_{\Gamma_h} w_h\|_{L^2(\partial T)}+\|\nabla_\Gamma v-\nabla_{\Gamma_h} v\|_{L^2(\partial T)}\right.\\
&\quad + \left.\|\beta\nabla v\cdot \textbf{n}-\beta\nabla w_h\cdot \textbf{n}_h\|_{L^2(\partial T)}+\|\beta\nabla v\cdot(\textbf{n}-\textbf{n}_h)\|_{L^2(\partial T)}\right).
\end{aligned}
\end{equation}
The  remaining proof of Lemma~\ref{lema_trace} can be easily adapted to the 3D cases  using (\ref{rela_grad_gamma}) and the fact that $\beta\nabla w_h\cdot \mathbf{n}_h$ and $\nabla_{\Gamma_h} w_h$ are constants.

Since the IFE space is defined according to the exact interface and the exact coefficient $\beta(x)$ is used in the IFE method,  the analysis for the IFE method in subsection~\ref{sub_sec_ana_IFE} can be adapted to the 3D cases if  we replace $T_h^\pm$ and $\beta_h$ by $T^\pm$ and $\beta^\pm$, respectively.
Note that, we only need to take into consideration of the discontinuity of IFE functions across the interface $\Gamma$. 

The discontinuity of IFE functions on $\Gamma$ causes two issues. The first one is that the trace inequality (\ref{pro_int1}) in the proof of Lemma~\ref{ener_app} does not hold since $(v-I_h^{{\rm IFE}}v)|_{T}$  does not belong to $H^1(T)$ now.
To overcome this issue,  we define an operator $\widehat{I_{h}^{{\rm IFE}}}$ by
\begin{equation}\label{def_hatIFE}
(\widehat{I_{h}^{{\rm IFE}}}v)|_T=\left\{
\begin{aligned}
(I_{h}^{{\rm IFE}}v)^+\quad \mbox{ in } T_h^+,\\
(I_{h}^{{\rm IFE}}v)^-\quad \mbox{ in } T_h^-,
\end{aligned}
\right.
\end{equation}
where $T_h^+$ and $T_h^-$ are subdomains of $T$ divided by the plane $\Gamma_{h}^{ext}$.  It is obvious that $(\widehat{I_{h}^{{\rm IFE}}}v)|_T\in C^0(T)$.
We have the following lemma whose proof is given  in Appendix~\ref{pro_lem_mish3D}
\begin{lemma}\label{lem_mish3D}
Let $T\in\mathcal{T}_h^\Gamma$ be an interface element and $e\in\mathcal{E}_h^\Gamma$ be one of its interface faces. For any $v\in  \widetilde{H}^2(\Omega)$, there exists a constant $C$ independent of $h$ and the interface location relative to the mesh  such that 
\begin{equation}\label{dis}
\begin{aligned}
&|\widehat{I_{h}^{{\rm IFE}}}v-I_{h}^{{\rm IFE}}v|_{H^m(T)}\leq Ch^{5/2-2m}| I_{h}^{{\rm IFE}}v|_{H^1(T)},\\
&\|(\widehat{I_{h}^{{\rm IFE}}}v-I_{h}^{{\rm IFE}}v)|_T\|_{L^2(e)}\leq  Ch^{3/2}| I_{h}^{{\rm IFE}}v|_{H^1(T)}.
\end{aligned}
\end{equation}
\end{lemma}
Based on the above lemma, the inequality (\ref{pro_int1}) in the proof of Lemma~\ref{ener_app}  is changed to
\begin{equation*}
\begin{aligned}
h^{-1}&\| [v-I_h^{{\rm IFE}}v]_e\|^2_{L^2(e)}\leq Ch^{-1}\| [v-\widehat{I_h^{{\rm IFE}}}v]_e\|^2_{L^2(e)}+Ch^{-1}\| [\widehat{I_h^{{\rm IFE}}}v-I_h^{{\rm IFE}}v]_e\|^2_{L^2(e)}\\
&\leq C\sum_{i=1,2}(h^{-2}\|v-\widehat{I_h^{{\rm IFE}}}v\|^2_{L^2(T_i)}+|v-\widehat{I_h^{{\rm IFE}}}v|^2_{H^1(T_i)})+Ch^{-1}\| [\widehat{I_h^{{\rm IFE}}}v-I_h^{{\rm IFE}}v]_e\|^2_{L^2(e)}\\
&\leq C\sum_{i=1,2}\left(h^{-2}\|v-I_h^{{\rm IFE}}v\|^2_{L^2(T_i)}+|v-I_h^{{\rm IFE}}v|^2_{H^1(T_i)}+h^{-2}\|\widehat{I_h^{{\rm IFE}}}v-I_h^{{\rm IFE}}v\|^2_{L^2(T_i)}\right.\\
&\qquad\qquad\qquad\left.+|\widehat{I_h^{{\rm IFE}}}v-I_h^{{\rm IFE}}v|^2_{H^1(T_i)}+h^{-1}\| (\widehat{I_h^{{\rm IFE}}}v-I_h^{{\rm IFE}}v)|_{T_i}\|^2_{L^2(e)}\right)\\
&\leq C\sum_{i=1,2}\left(h^{-2}\|v-I_h^{{\rm IFE}}v\|^2_{L^2(T_i)}+|v-I_h^{{\rm IFE}}v|^2_{H^1(T_i)}+h|I_h^{{\rm IFE}}v|^2_{H^1(T_i)}\right).
\end{aligned}
\end{equation*}
Summing over all interface faces and using Theorem~\ref{chazhi_error3D}, we have
\begin{equation*}
\begin{aligned}
\sum_{e\in\mathcal{E}_h^\Gamma}h^{-1}\|& [v-I_h^{{\rm IFE}}v]_e\|^2_{L^2(e)}\leq Ch^2\|v\|^2_{H^2(\Omega^+\cup\Omega^-)}+Ch\sum_{T\in\mathcal{T}_h^\Gamma}\left(|v|^2_{H^1(T)}+|v-I_h^{{\rm IFE}}v|^2_{H^1(T)}\right)\\
&\leq Ch^2\|v\|^2_{H^2(\Omega^+\cup\Omega^-)}+Ch\sum_{i=\pm}|\mathbb{E}^iv^i|^2_{H^1(N(\Gamma,h))}\\
&\leq Ch^2\|v\|^2_{H^2(\Omega^+\cup\Omega^-)}+Ch^2\sum_{i=\pm}\|\mathbb{E}^iv^i\|^2_{H^2(\Omega)}\leq Ch^2\|v\|^2_{H^2(\Omega^+\cup\Omega^-)},
\end{aligned}
\end{equation*}
where we have used Lemma~\ref{strip} in the third inequality. Therefore, the result of Lemma~\ref{ener_app}  also holds.

The other issue caused by the discontinuity of IFE functions on $\Gamma$ is the consistent error of the proposed IFE method. The identity in Lemma~\ref{lem_consis} now becomes 
\begin{equation}\label{cons3D}
A_h(u-u_h,v_h)= \int_{\Gamma}\beta^- \nabla u^-\cdot\mathbf{n} [v_h]_\Gamma ds,  \qquad \forall v_h\in V_h^{{\rm IFE}}.
\end{equation}
The estimate of the consistent error is shown in the following lemma whose proof is given in Appendix \ref{pro_consis_3D}
\begin{lemma}\label{lem_consis_3D}
For any $u\in \widetilde{H}^2(\Omega)$ and $v_h\in V_h^{{\rm IFE}}$, it holds 
\begin{equation*}
\left|\int_{\Gamma}\beta^- \nabla u^-\cdot\mathbf{n} [v_h]_\Gamma ds\right|\leq Ch^{3/2}\|u\|_{H^2(\Omega^-)}\left(\sum_{T\in\mathcal{T}_h^\Gamma}\|\nabla v_h\|^2_{L^2(T^+\cup T^-)}\right)^{1/2}.
\end{equation*}
\end{lemma}

With the help of (\ref{cons3D}) and the above lemma,  the proof of Theorem~\ref{theo_mainH1} can be easily modified and the result (\ref{h1_error}) also holds under the conditions of Lemma~\ref{lem_unique3D}.

For the proof of Theorem~\ref{theo_mainL2}, the equation (\ref{pro_l2_1}) now becomes 
\begin{equation}\label{ulti}
\begin{aligned}
\|u-u_h&\|_{L^2(\Omega)}^2=A_h(w-I_h^{{\rm IFE}}z,u-u_h)+\int_{\Gamma}\beta^- \nabla u^-\cdot\mathbf{n} [I_h^{{\rm IFE}}z]_\Gamma ds+\int_\Gamma \beta^-\nabla z^-\cdot\mathbf{n}[u_h]_\Gamma ds.
\end{aligned}
\end{equation}
By Lemma~\ref{lem_consis_3D}, the second term can be estimated as
\begin{equation*}
\begin{aligned}
&\left|\int_{\Gamma}\beta^- \nabla u^-\cdot\mathbf{n} [I_h^{{\rm IFE}}z]_\Gamma ds\right|\leq Ch^{3/2}\|u\|_{H^2(\Omega^-)}\left(\sum_{T\in\mathcal{T}_h^\Gamma}\|\nabla I_h^{{\rm IFE}}z\|^2_{L^2(T^+\cup T^-)}\right)^{1/2}\\
&\qquad\qquad\leq Ch^{3/2}\|u\|_{H^2(\Omega^-)}\left(\left(\sum_{T\in\mathcal{T}_h^\Gamma}\|\nabla (I_h^{{\rm IFE}}z-z)\|^2_{L^2(T^+\cup T^-)}\right)^{1/2}+\|\nabla z\|_{L^2(N(\Gamma,h))}\right)\\
&\qquad\qquad\leq Ch^{3/2}\|u\|_{H^2(\Omega^-)}\left(h\|z\|_{H^2(\Omega^+\cup\Omega^-)}+\sum_{i=\pm}\|\nabla \mathbb{E}^iz^i\|_{L^2(N(\Gamma,h))}\right)\\
&\qquad\qquad\leq Ch^{3/2}\|u\|_{H^2(\Omega^-)}\left(h\|z\|_{H^2(\Omega^+\cup\Omega^-)}+\sum_{i=\pm}h^{1/2}\|\mathbb{E}^iz^i\|_{H^2(\Omega)}\right)\\
&\qquad\qquad\leq Ch^{2}\|u\|_{H^2(\Omega^+\cup \Omega^-)}\|z\|_{H^2(\Omega^+\cup \Omega^-)},
\end{aligned}
\end{equation*}
where in the third inequality we have used Theorem~\ref{chazhi_error3D}  and in the fourth inequality we have used Lemma~\ref{strip}. The estimate of the third term on the right-hand of (\ref{ulti}) is analogous. The remaining proof  is standard. The $L^2$ error estimate in Theorem~\ref{theo_mainL2}  also holds under the conditions of Lemma~\ref{lem_unique3D}.

\section{Numerical examples}\label{sec_num}
In this section, we present some numerical examples for the parameter free PPIFE method to validate the theoretical analysis. Let the domain $\Omega$ be the unit square $(-1,1)\times(-1,1)$, and the interface $\Gamma$ be the zero level set of a
function $\varphi(x)$
so that $\Omega^+=\{x\in\Omega : \varphi(x)>0\}$ and $\Omega^-=\{x\in\Omega : \varphi(x)<0\}$.
We use a non-homogeneous boundary condition $u|_{\partial \Omega}=g$.

For simplicity, we use Cartesian meshes that are formed by first partitioning $\Omega$ into $N\times N$ congruent squares and then cutting these squares along one of diagonals in the same direction.  
We examine the convergence rate of the parameter free PPIFE method using the following norms:
\begin{equation}
|e_h|_{H^1}:=\|u-u_h\|_h \qquad \mbox{ and }\qquad \|e_h\|_{L^2}:=\|u-u_h\|_{L^2(\Omega)}.
\end{equation}

\textbf{Example 1} (from \cite{Li2003new}).
The level set function is $\varphi(x)=\sqrt{x_1^2+x_2^2}-r_0$ with $r_0=0.5$.
The exact solution to the interface problem is chosen as
\begin{equation}
u(x)=\left\{
\begin{aligned}
&\frac{r^3}{\beta^-}\qquad &\mbox{ in }\Omega^-,\\
&\frac{r^3}{\beta^+}+\left(\frac{1}{\beta^-}-\frac{1}{\beta^+}\right)r_0^3 &\mbox{ in }\Omega^+,
\end{aligned}
\right.
\end{equation}
where $r=\sqrt{x_1^2+x_2^2}$. From the PDE, we find the source term is $f(x)=-9\sqrt{x_1^2+x_2^2}$.

 We test our method for two cases:   Case 1: $\beta^-=1$, $\beta^+=2$, $10$, $1000$ and $100000$;   Case 2: $\beta^+=1$ and $\beta^-=2$, $10$, $1000$ and $100000$. The numerical results reported in Tables~\ref{ex1_1L2}--\ref{ex1_2H1} show optimal orders of convergence:
$$
|e_h|_{H^1}\approx O(h) \qquad \mbox{ and }\qquad \|e_h\|_{L^2}\approx O(h^2),
$$
which are in  agreement with Theorems~\ref{theo_mainH1} and \ref{theo_mainL2}.

\begin{table}[H]
\caption{The  $\|e_h\|_{L^2}$  errors and convergence rates for  Case 1 of Example 1, $\beta^-=1$, $\beta^+=2$, $10$, $1000$ and $100000$.\label{ex1_1L2}}
\begin{center}
\begin{tabular}{|c|c c|c c|c c|c c|}
  \hline
       $N$  &  $\beta^+=2$   &  rate   &  $\beta^+=10$  &  rate  &  $\beta^+=1000$  &  rate  & $\beta^+=100000$  &  rate    \\ \hline
     8 &  4.029E-02  &       &   1.363E-02  &         &   1.210E-02  &       & 1.211E-02  &             \\ \hline
    16 &  1.018E-02  & 1.99  &   3.734E-03  &  1.87   &   4.353E-03  & 1.47  & 4.505E-03  & 1.43  \\ \hline
    32 &  2.560E-03  & 1.99  &   9.981E-04  &  1.90   &   1.312E-03  & 1.73  & 1.842E-03  & 1.29  \\ \hline
    64 &  6.403E-04  & 2.00  &   2.480E-04  &  2.01   &   4.034E-04  & 1.70  & 8.009E-04  & 1.20  \\ \hline
   128 &  1.605E-04  & 2.00  &   6.344E-05  &  1.97   &   7.446E-05  & 2.44  & 2.445E-04  & 1.71  \\ \hline
   256 &  4.013E-05  & 2.00  &   1.580E-05  &  2.01   &   1.674E-05  & 2.15  & 6.692E-05  & 1.87  \\ \hline
   512 &  1.004E-05  & 2.00  &   3.953E-06  &  2.00   &   3.953E-06  & 2.08  & 1.794E-05  & 1.90  \\ \hline
  1024 &  2.509E-06  & 2.00  &   9.851E-07  &  2.00   &   9.485E-07  & 2.06  & 3.887E-06  & 2.21   \\ \hline
\end{tabular}
\end{center}
\end{table}

\begin{table}[H]
\caption{ The  $|e_h|_{H^1}$ errors and convergence rates for  Case 1 of  Example 1, $\beta^-=1$, $\beta^+=2$, $10$, $1000$ and $100000$.\label{ex1_1H1}}
\begin{center}
\begin{tabular}{|c|c c|c c|c c|c c|}
  \hline
       $N$  &  $\beta^+=2$   &  rate   &  $\beta^+=10$  &  rate  &  $\beta^+=1000$  &  rate  & $\beta^+=100000$  &  rate    \\ \hline
         8  &  5.823E-01   &       & 2.851E-01 &      &  1.313E-01  &          & 1.288E-01   &        \\ \hline
        16  &  2.929E-01   & 0.99  & 1.466E-01 & 0.96 &  8.084E-02  &    0.70  & 8.127E-02   &   0.66\\ \hline
        32  &  1.467E-01   & 1.00  & 7.402E-02 & 0.99 &  4.323E-02  &    0.90  & 4.950E-02   &   0.72\\ \hline
        64  &  7.337E-02   & 1.00  & 3.709E-02 & 1.00 &  2.206E-02  &    0.97  & 2.903E-02   &   0.77\\ \hline
       128  &  3.669E-02   & 1.00  & 1.856E-02 & 1.00 &  1.029E-02  &    1.10  & 1.459E-02   &   0.99\\ \hline
       256  &  1.835E-02   & 1.00  & 9.282E-03 & 1.00 &  5.039E-03  &    1.03  & 7.163E-03   &   1.03\\ \hline
       512  &  9.173E-03   & 1.00  & 4.642E-03 & 1.00 &  2.498E-03  &    1.01  & 3.335E-03   &   1.10\\ \hline
      1024  &  4.587E-03   & 1.00  & 2.321E-03 & 1.00 &  1.240E-03  &    1.01  & 1.485E-03   &   1.17  \\ \hline
\end{tabular}
\end{center}
\end{table}

\begin{table}[H]
\caption{ The  $\|e_h\|_{L^2}$ errors and convergence rates for  Case 2 of Example 1, $\beta^+=1$, $\beta^-=2$, $10$, $1000$ and $100000$.\label{ex1_2L2}}
\begin{center}
\begin{tabular}{|c|c c|c c|c c|c c|}
  \hline
       $N$  &  $\beta^-=2$   &  rate   &  $\beta^-=10$  &  rate  &  $\beta^-=1000$  &  rate  & $\beta^-=100000$  &  rate    \\ \hline
   8 & 7.770E-02 &       &  7.758E-02 &       &  7.776E-02 &       &  7.777E-02  &        \\ \hline
  16 & 1.957E-02 & 1.99  &  1.953E-02 &  1.99 &  1.959E-02 &  1.99 &  1.959E-02  & 1.99  \\ \hline
  32 & 4.908E-03 & 2.00  &  4.904E-03 &  1.99 &  4.937E-03 &  1.99 &  5.005E-03  & 1.97  \\ \hline
  64 & 1.229E-03 & 2.00  &  1.229E-03 &  2.00 &  1.237E-03 &  2.00 &  1.281E-03  & 1.97  \\ \hline
 128 & 3.074E-04 & 2.00  &  3.078E-04 &  2.00 &  3.078E-04 &  2.01 &  3.698E-04  & 1.79  \\ \hline
 256 & 7.687E-05 & 2.00  &  7.701E-05 &  2.00 &  7.692E-05 &  2.00 &  9.528E-05  & 1.96  \\ \hline
 512 & 1.922E-05 & 2.00  &  1.926E-05 &  2.00 &  1.925E-05 &  2.00 &  2.564E-05  & 1.89  \\ \hline
1024 & 4.805E-06 & 2.00  &  4.817E-06 &  2.00 &  4.820E-06 &  2.00 &  5.997E-06  & 2.10   \\ \hline
\end{tabular}
\end{center}
\end{table}

\begin{table}[H]
\caption{ The  $|e_h|_{H^1}$ errors and convergence rates for  Case 2 in Example 1, $\beta^+=1$, $\beta^-=2$, $10$, $1000$ and $100000$.\label{ex1_2H1}}
\begin{center}
\begin{tabular}{|c|c c|c c|c c|c c|}
  \hline
       $N$  &  $\beta^-=2$   &  rate   &  $\beta^-=10$  &  rate  &  $\beta^-=1000$  &  rate  & $\beta^-=100000$  &  rate    \\ \hline
    8  & 8.046E-01 &       & 7.998E-01  &       &7.949E-01  &        &7.948E-01 &          \\ \hline
   16  & 4.036E-01 & 1.00  & 4.008E-01  & 1.00  &3.996E-01  &  0.99  &3.996E-01 &  0.99  \\ \hline
   32  & 2.020E-01 & 1.00  & 2.005E-01  & 1.00  &2.008E-01  &  0.99  &2.020E-01 &  0.98  \\ \hline
   64  & 1.010E-01 & 1.00  & 1.003E-01  & 1.00  &1.005E-01  &  1.00  &1.011E-01 &  1.00  \\ \hline
  128  & 5.051E-02 & 1.00  & 5.013E-02  & 1.00  &5.011E-02  &  1.00  &5.106E-02 &  0.99  \\ \hline
  256  & 2.526E-02 & 1.00  & 2.507E-02  & 1.00  &2.504E-02  &  1.00  &2.545E-02 &  1.00  \\ \hline
  512  & 1.263E-02 & 1.00  & 1.253E-02  & 1.00  &1.252E-02  &  1.00  &1.271E-02 &  1.00  \\ \hline
 1024  & 6.314E-03 & 1.00  & 6.267E-03  & 1.00  &6.257E-03  &  1.00  &6.308E-03 &  1.01   \\ \hline
\end{tabular}
\end{center}
\end{table}

\textbf{Example 2} (an interface problem with a variable coefficient and a non-convex interface).   The interface is the zero level set of the function,
 $$\varphi(x)=(3(x_1^2+x_2^2)-x_1)^2-x_1^2-x_2^2+0.02.$$
 The exact solution is chosen as $u(x)=\varphi(x)/\beta(x)$, where
 \begin{equation}
 \beta(x)=\left\{
 \begin{aligned}
 &\beta^+(x)=300(2+\sin(6x_1+6x_2))& \mbox{ if } \varphi(x)>0,\\
 &\beta^-(x)=2+\cos(6x_1+6x_2) &\mbox{ if } \varphi(x)<0.
 \end{aligned}\right.
 \end{equation}
 It is easy to verify that the jump conditions (\ref{p1.2})-(\ref{p1.3}) are satisfied. The interface in this example is non-convex and more general, see Figure~\ref{ex2_fig}.

For this variable coefficient interface problem, the average of the coefficient  on an  interface element $T\in\mathcal{T}_h^\Gamma$ is  chosen as $\overline{\beta}_h^\pm=\beta^\pm(x_m)$, where
$x_m$ is the midpoint of $\Gamma_h\cap T$. The numerical results are reported in Table~\ref{ex2_biao}, which confirm  the optimal convergence.

 \begin{figure} [htbp]
\centering
\subfigure{ 
\includegraphics[width=0.45\textwidth]{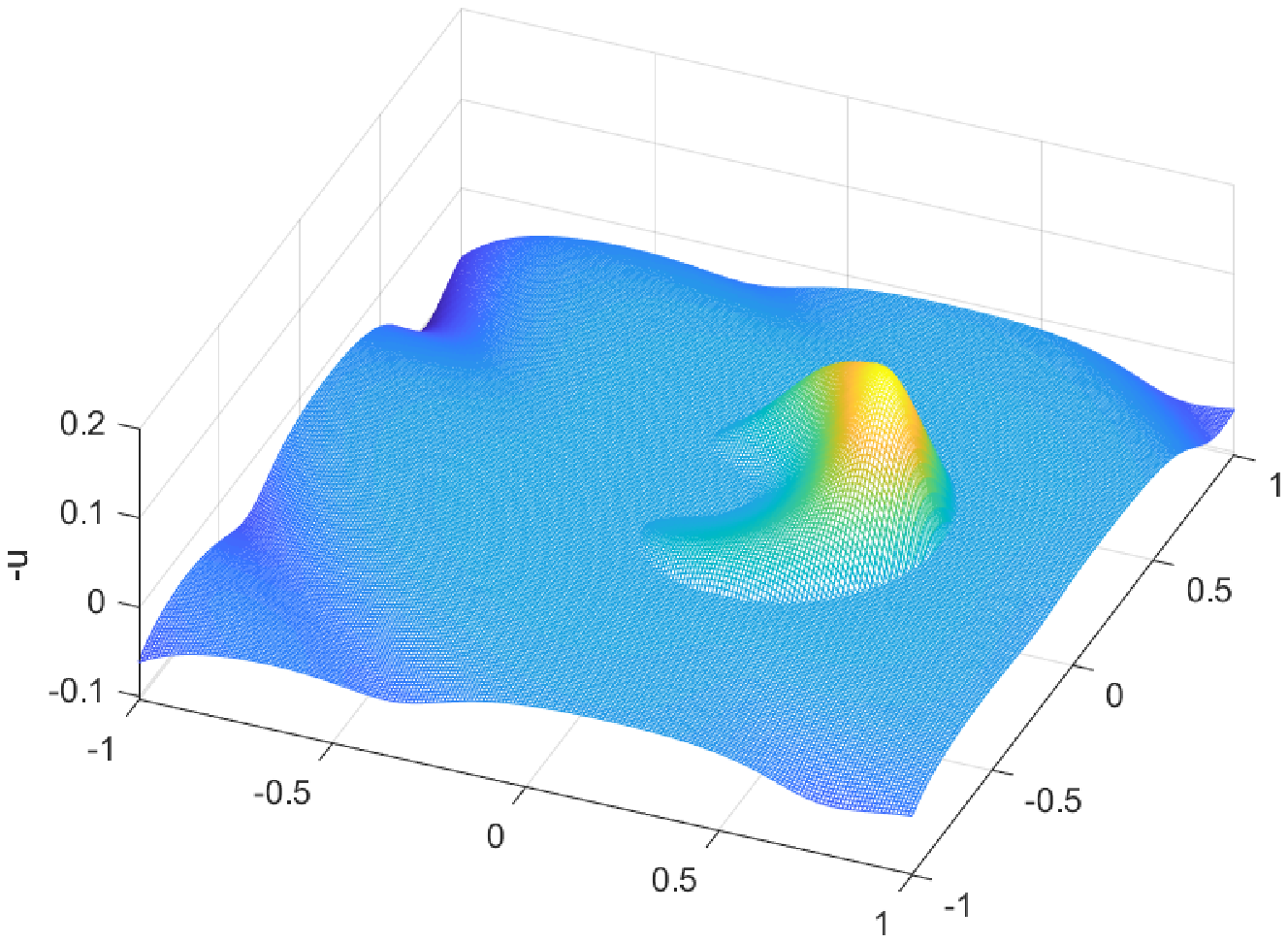}}
\subfigure{
\includegraphics[width=0.45\textwidth]{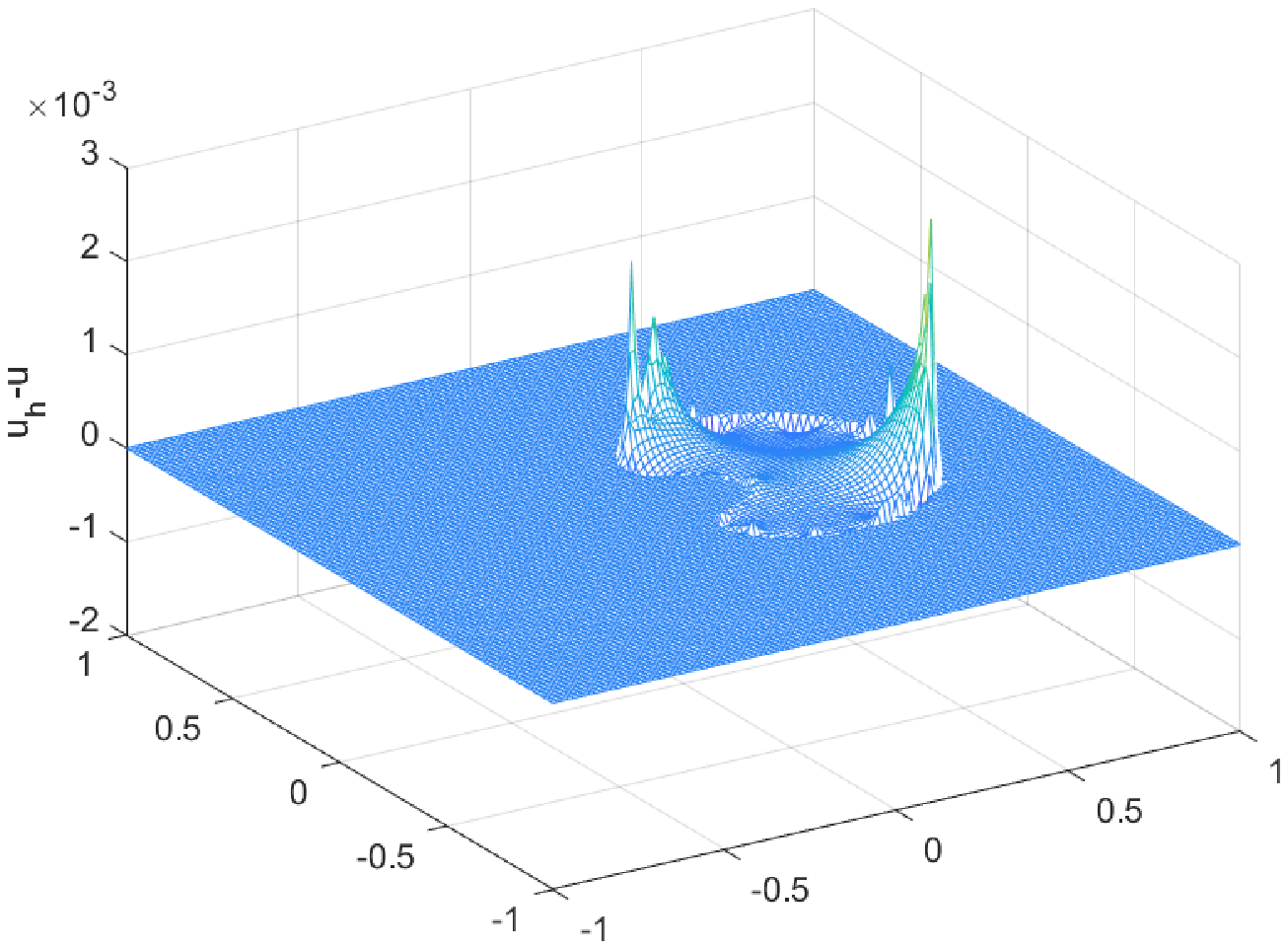}}
 \caption{The solution and the error distribution obtained with $N=128$ for Example 2. Left: the plot of $-u(x)$; Right: the plot of $u_h(x)-u(x)$.\label{ex2_fig}} 
\end{figure}

\begin{table}[H]
\caption{The  $\|e_h\|_{L^2}$  and $|e_h|_{H^1}$ errors and convergence rates for   Example 2. \label{ex2_biao}}
\begin{center}
\begin{tabular}{|c|c c|c c|}
\hline
        $N$  &  $\|e_h\|_{L^2}$  &  rate &   $|e_h|_{H^1}$  &  rate    \\ \hline
         8   &   3.146E-02  &           &    1.673E+00 &                       \\ \hline
        16   &   1.210E-02  &    1.38   &   8.997E-01  &    0.89      \\ \hline
        32   &   4.882E-03  &    1.31   &   4.661E-01  &    0.95     \\ \hline
        64   &   1.456E-03  &    1.75   &   2.339E-01  &    0.99     \\ \hline
       128   &   2.603E-04  &    2.48   &   1.156E-01  &    1.02    \\ \hline
       256   &   5.795E-05  &    2.17   &   5.763E-02  &    1.00   \\ \hline
       512   &   1.376E-05  &    2.07   &   2.878E-02  &    1.00    \\ \hline
      1024   &   2.676E-06  &    2.36   &   1.435E-02  &    1.00    \\ \hline
\end{tabular}
\end{center}
\end{table}

\section{Conclusions and outlook}

In this paper, a new parameter free partially penalized immersed finite element method with unfitted meshes for solving elliptic interface problems has been developed and analyzed. The degrees of freedom of the new method are the same as that of the standard linear conforming finite element method.  Furthermore, if the coefficient is continuous, then the new method becomes the standard linear conforming finite element method.  Not only the optimal approximation capabilities of the immersed finite element  space but also the optimal convergence of the new PPIFE method are  proved via a new trace inequality and a novel proof technique using auxiliary functions.  The method and analysis has been extended to variable coefficients and 3D problems.

Further directions of research  include the extension of the method and the analysis in this paper to non-homogeneous jump conditions and higher order IFE methods.
We plan to use the correction functions defined in \cite{guzman2016higher,He2012Immersed} to deal with non-homogeneous jump conditions. The correction functions are constructed to satisfy non-homogeneous jump conditions exactly on some points on the interface, which requires that the exact solution has a higher regularity to make the flux jump $\mathcal{J}_N(x):=[\beta\nabla u\cdot \mathbf{n}]_\Gamma(x)$ well-defined at these points on the interface. The method in \cite{He2012Immersed} has been analyzed in \cite{guzman2016accuracy} under the assumption that $u\in C^2(\overline{\Omega}^\pm)$.  Another way to construct   correction functions is based on the extension of $\mathcal{J}_N(x)$ (see \cite{jiESAJAM2018}). The analysis requires $u\in H^3(\Omega^+\cup\Omega^-)$. 
When  $u\in H^2(\Omega^+\cup\Omega^-)$, the flux jump $\mathcal{J}_N(x)$ belongs to $L^2(\Gamma)$ and $\mathcal{J}_N(x)$ is not well-defined at  points on the interface.  Once common approach is to use the average of $\mathcal{J}_N$ along $\Gamma\cap T$ to construct correction functions on an interface element $T$.   Since $|\Gamma\cap T|$ may  be close to zero, we should define the correction function on a larger fictitious interface element as  it has been done in \cite{adjerid2020enriched,201GUOSIAM}. 
For high order IFE methods, there are many exploratory works  \cite{adjerid2018higher,adjerid2017high}.
To our best knowledge, the proof of the optimal approximation capabilities of those higher order IFE spaces developed in [2,3] is an open problem. How to extend the analysis here to high order IFE methods is under investigation.

\textbf{Acknowledgment.}
The authors would like to thank the anonymous referees sincerely for their careful reading and helpful suggestions that improved the quality of this paper.

\bibliographystyle{plain}

 \begin{appendices}
\section{Technical results for the 2D cases}\label{appen2D}

\subsection{Proof of Lemma~\ref{lem_unique}}\label{pro_lem_unique}
\begin{proof} $\!\!:$
Note that when $T$ is an isosceles right triangle, the proof can be found in  the literature, see for example \cite{Li2003new,GuoIMA2019}.
 We now provide the proof for $\alpha_{max}\leq \pi/2$. Given a function $\phi\in S_h(T)$, if we known the jump $c_1:=(\nabla \phi^+-\nabla \phi^-)\cdot \mathbf{n}_h$ which is a constant, then the function $\phi$ can be written as 
\begin{equation}\label{app1_deco}
\phi=I_{h,T}\phi+c_1(w-I_{h,T}w),
\end{equation}
with
\begin{equation}\label{def_w2D}
w(x)=\left\{
\begin{aligned}
&w^+(x)=d_{\Gamma_{h,T}^{ext}}(x)~~ &\mbox{ in } T_h^+,\\
&w^-(x)=0  &\mbox{ in } T_h^-,
\end{aligned}\right.
~~~~
d_{\Gamma_{h,T}^{ext}}(x)=\left\{
\begin{aligned}
&{\rm dist}(x,\Gamma_{h,T}^{ext})~~ &\mbox{ if } x\in T_h^+,\\
&-{\rm dist}(x,\Gamma_{h,T}^{ext})  &\mbox{ if } x\in T_h^-,
\end{aligned}\right.
\end{equation}
where $I_{h,T}$ is the standard linear nodal interpolation operator on $T$, $\Gamma_{h,T}^{ext}$ is a straight line containing $\Gamma_h\cap T$, and ${\rm dist}(x, \Gamma_{h,T}^{ext})$ is the distance between $x$ and $\Gamma_{h,T}^{ext}$.
Substituting (\ref{app1_deco}) into the third identity in (\ref{ifem_modi}),  we obtain the following equation for $c_1$,
\begin{equation}\label{app1_equ}
(1+(\beta^-/\beta^+-1)\nabla I_{h,T}w\cdot\mathbf{n}_h)c_1=(\beta^-/\beta^+-1)\nabla I_{h,T}\phi\cdot{\mathbf{n}_h}.
\end{equation}
Clearly, if we can prove 
\begin{equation}\label{app1_01}
0\leq \nabla I_{h,T}w\cdot\mathbf{n}_h\leq 1,
\end{equation}
then
\begin{equation}\label{est_fenmu}
(1+(\beta^-/\beta^+-1)\nabla I_{h,T}w\cdot\mathbf{n}_h)\geq 
\left\{
\begin{aligned}
&1\qquad &&\mbox{ if } \beta^-/\beta^+\geq 1,\\
&\beta^-/\beta^+\qquad &&\mbox{ if } 0<\beta^-/\beta^+<1,
\end{aligned}\right.
\end{equation}
which implies that the equation (\ref{app1_equ}) has a unique solution. Substituting the solution of  (\ref{app1_equ}) into (\ref{app1_deco}) yields 
\begin{equation}\label{app1_basis}
\phi=I_{h,T}\phi+\frac{(\beta^-/\beta^+-1)\nabla I_{h,T}\phi\cdot{\mathbf{n}_h}}{1+(\beta^-/\beta^+-1)\nabla I_{h,T}w\cdot\mathbf{n}_h}(w-I_{h,T}w),
\end{equation}
which proves the lemma. 

Next, we prove (\ref{app1_01}).  There are two cases.  Case $\rm I$: $\triangle A_2DE=T_h^+$ (see Figure~\ref{triangle_analysis}) and Case $\rm II$: $\triangle A_2DE=T_h^-$. In  Case $\rm I$,  since  $w(A_1)=w(A_3)=0$,  it holds 
\begin{equation}\label{2Dcase1}
\nabla I_{h,T}w\cdot\mathbf{n}_h=\nabla\lambda_2\cdot\textbf{n}_hd_{\Gamma_{h,T}^{ext}}(A_2)=\nabla\lambda_2\cdot\textbf{n}_h|A_2A_{2,\perp}|=1-\lambda_2(A_{2,\perp}),
\end{equation}
where 
 $A_{2,\perp}$ is the orthogonal projection of the point $A_2$ onto the line $DE$,  and $\lambda_i(x)$ is the standard linear basis function defined by $\lambda_i(A_j)=\delta_{ij}$ (the Kronecker symbol). 
The polynomial extension of $\lambda_i(x)$ is also denoted by $\lambda_i(x)$ for simplicity of notations.
  In Case II  ($\triangle A_2DE=T_h^-$), for the sake of clarity, we replace the notations $w$, $\mathbf{n}_h$ and $d_{\Gamma_{h,T}^{ext}}$ by $\tilde{w}$,  $\tilde{\mathbf{n}}_h$ and $\tilde{d}_{\Gamma_{h,T}^{ext}}$, respectively. Obviously, $\tilde{\mathbf{n}}_h=-\mathbf{n}_h$ and $\tilde{d}_{\Gamma_{h,T}^{ext}}(A_2)=-|A_2A_{2,\perp}|$. Similar to (\ref{2Dcase1}), we have
  \begin{equation}\label{2Dcase2_ori}
 \nabla I_{h,T}\tilde{w}\cdot\tilde{\mathbf{n}}_h=1-\lambda_1(A_{1,\perp})+1-\lambda_3(A_{3,\perp}).
 \end{equation}
However, using the signed distance function $\tilde{d}_{\Gamma_{h,T}^{ext}}$ we also have
 \begin{equation}\label{2Dcase2}
 \begin{aligned}
 \nabla I_{h,T}\tilde{w}\cdot\tilde{\mathbf{n}}_h&=\nabla(\sum_{i=1,3}\tilde{d}_{\Gamma_{h,T}^{ext}}(A_i)\lambda_i)\cdot\tilde{\textbf{n}}_h=\nabla(\tilde{d}_{\Gamma_{h,T}^{ext}}-\tilde{d}_{\Gamma_{h,T}^{ext}}(A_2)\lambda_2)\cdot\tilde{\mathbf{n}}_h\\
 &=1-\nabla\lambda_2\cdot\textbf{n}_h|A_2A_{2,\perp}|=1-\nabla I_{h,T}w\cdot\mathbf{n}_h.
 \end{aligned}
 \end{equation}
Thus, it suffices  to consider Case $\rm I$: $\triangle A_2DE=T_h^+$ which is shown in Figure~\ref{triangle_analysis}.

If  $A_2$ and $A_{2,\perp}$ are on different sides of the line $A_1A_3$  (see Figure~\ref{triangle_analysis3} for an illustration), then we have $\angle A_1A_3A_2>\angle A_3QA_2>\frac{\pi}{2}$. This contradicts the condition $\alpha_{max}\leq \frac{\pi}{2}$. Thus, we conclude that
$A_2$ and $A_{2,\perp}$ are on the same side of the line $A_1A_3$, which together with the fact $\lambda_2(A_2)=1$, $\lambda_2(A_1)=\lambda_2(A_3)=0$, leads to
\begin{equation*}
\nabla I_{h,T}w\cdot\mathbf{n}_h=1-\lambda_2(A_{2,\perp})\leq 1,
\end{equation*}
On the other hand, using the condition $\angle A_2A_3A_1$ and $\angle A_2A_1A_3\leq \alpha_{max}\leq \frac{\pi}{2}$, we conclude $\overrightarrow{A_1A_3}\cdot\textbf{t}_h>0$, which implies 
\begin{equation*}
\begin{aligned}
\nabla I_{h,T}w\cdot\mathbf{n}_h&=\nabla\lambda_2\cdot\textbf{n}_h|A_2A_{2,\perp}|=|A_2A_{2,\perp}||\nabla\lambda_2|\textbf{n}_{A_1A_3}\cdot\textbf{n}_h\\
&=|A_2A_{2,\perp}||\nabla\lambda_2||A_1A_3|^{-1}\overrightarrow{A_1A_3}\cdot\textbf{t}_h\geq0,
\end{aligned}
\end{equation*}
where  $\textbf{n}_{A_1A_3}$ is the unit normal vector of the line $A_1A_3$ pointing toward $A_2$.  \qed
\end{proof}

\begin{figure} [htbp]
\centering
\subfigure[]{ \label{triangle_analysis} 
\includegraphics[width=0.23\textwidth]{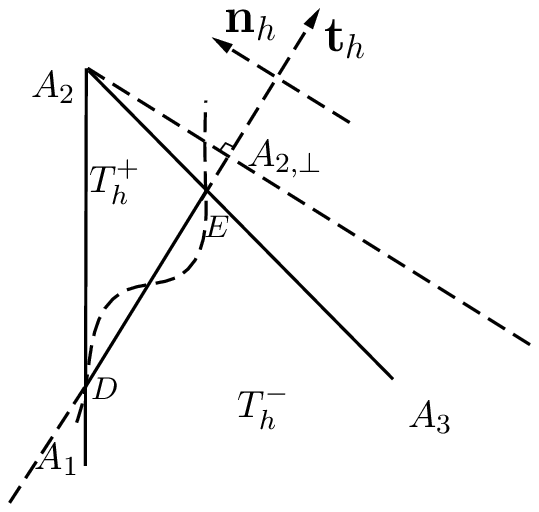}}
\subfigure[]{ \label{triangle_analysis3} 
\includegraphics[width=0.4\textwidth]{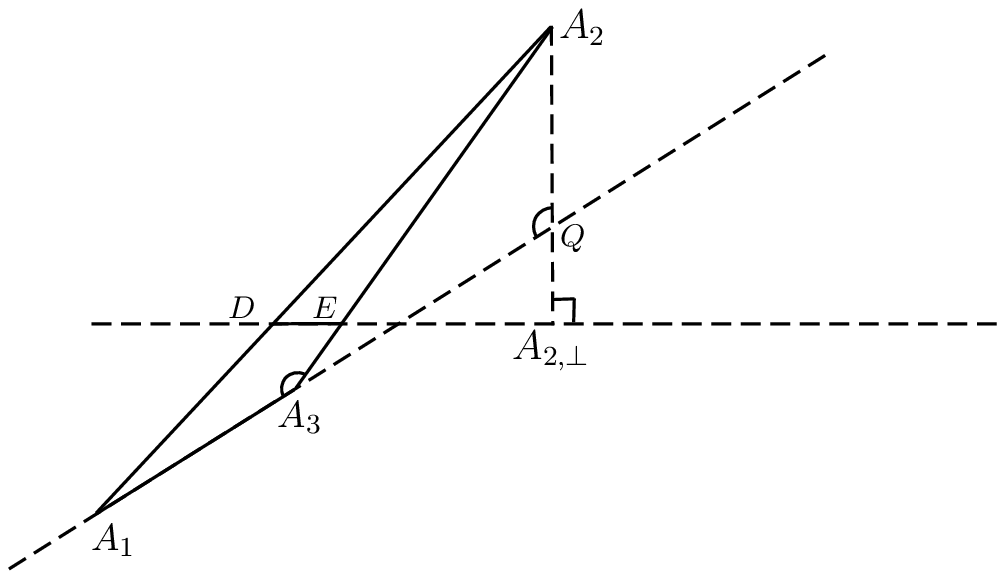}}
\subfigure[]{ \label{triangle_analysis2} 
\includegraphics[width=0.33\textwidth]{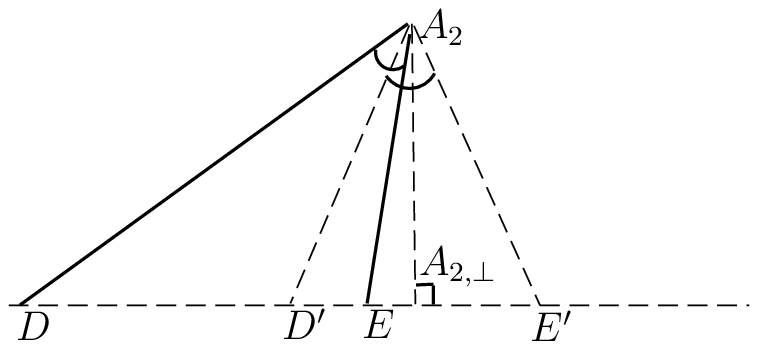}}
 \caption{A diagram of an interface element.} 
\end{figure}
\ \\
\textbf{A counter example for $\alpha_{max}>\frac{\pi}{2}$:}
\begin{equation*}
\begin{aligned}
&A_1=(0,0),~~ A_2=(-\sqrt{3},1),~~ A_3=(1,0),~~ D=(0,0),\\
&E=(-(2+\sqrt{3})^{-1}, \sqrt{3}(2+\sqrt{3})^{-1}),~~ \beta^-=3,~~ \beta^+=1,~~T_h^+=\triangle DA_2E.
\end{aligned}
\end{equation*}
By a direct calculation, we find that the shape function $\phi(x)$ cannot be determined by $\phi(A_i)$, $i=1,2,3$ in this case.

\subsection{Proof of Lemma~\ref{lem_jiest}}\label{pro_lem_jiest}
\begin{proof}:
First we present the following useful  inequality about basis functions of the linear IFE space $S_h(T)$.
Let $\phi_{A_i}\in S_h(T)$ be the basis function corresponding to a vertex $A_i$ of $T$ defined by $\phi_{A_i}(A_j)=\delta_{ij}$.  From (\ref{app1_basis}) and (\ref{app1_01}), it is easy to prove that \begin{equation}\label{est_phi}
|\phi_{A_i}|_{W_\infty^m(T)}\leq Ch^{-m},~ m=0,1,
\end{equation}
where the constant $C$ is independent of $h$ and the interface location relative to the mesh.

\textbf{Derive upper bounds of  $|\Upsilon(x)|_{H^m(T)},~m=0,1$}. We construct $\Upsilon$  as follows,
\begin{equation}\label{le_eta1}
\Upsilon=z-I^{{\rm IFE}}_{h,T} z, \quad z=\left\{
\begin{aligned}
&v\qquad \mbox{ in } T_h^+,\\
&0 \qquad \mbox{ in } T_h^-,
\end{aligned}\right.
\end{equation}
where  the function $v$ is linear and satisfies
\begin{equation}\label{le_eta12}
\beta^+\nabla v\cdot \textbf{n}_h=1,\quad v(D)=v(E)=0.
\end{equation}
Here $I^{{\rm IFE}}_{h,T} z$ interpolates nodal values of $v$ defined on $T$, i.e.,  $I^{{\rm IFE}}_{h,T} z=\sum_{i}v(A_i)\phi_{A_i}$.
It is easy to verify that the constructed function $\Upsilon$ satisfies the definitions (\ref{jum_ji1})-(\ref{jum_ji1_cons}).  Since $v(D)=v(E)=0$, we have $\nabla v\cdot\textbf{t}_h=0$. Thus,
\begin{equation*}
|\nabla v|^2= |\nabla v\cdot \textbf{n}_h|^2+ |\nabla v\cdot \textbf{t}_h|^2\leq C.
\end{equation*}
For any point $P\in T_h^+$, using the relation $v(P)=v(D)+\nabla v\cdot \overrightarrow{DP}$, we have
\begin{equation*}
|z|^2_{L^\infty(T)}=|v|^2_{L^\infty(T_h^+)}\leq |\nabla v|^2|DP|^2\leq Ch^2.
\end{equation*}
From (\ref{le_eta1}) and (\ref{est_phi}), we get the desired  estimates
\begin{equation*}
\begin{aligned}
&\|\Upsilon\|^2_{L^2(T)}\leq 2\|z\|^2_{L^\infty(T)}|T|+2\sum_iz^2(A_i)\|\phi_{A_i}\|^2_{L^\infty(T)}|T|\leq Ch^4,\\
&|\Upsilon|^2_{H^1(T)}\leq 2|\nabla v|^2|T^+_h|+2\sum_iz^2(A_i)|\phi_{A_i}|^2_{W^1_\infty(T)}|T|\leq Ch^2.
\end{aligned}
\end{equation*}

\textbf{Derive upper bounds of $|\Psi_i(x)|_{H^m(T^+_h\cup T_h^-)},~i=D,E,~m=0,1$}. Without loss of generality, we assume that the interface $\Gamma$ intersects with the line segments $\overline{A_1A_2}$ and $\overline{A_2A_3}$ at points $D, E$, respectively, see Figure~\ref{triangle_analysis} for an illustration. Since the triangulation is regular, we assume that there are two constants $\alpha_{min}$ and $\alpha_{max}$ such that $\alpha_{min}\leq\angle A_1A_2A_3\leq\alpha_{max}$.

Let $D^\prime$ and $E^\prime$ be two points on the line $DE$ such that $\angle DA_2E=\angle D^\prime A_1E^\prime$ and $|A_2D^\prime|=|A_2E^\prime|$, see Figure~\ref{triangle_analysis2} for an illustration.  Then, we have the key inequality
\begin{equation}\label{le2_gama}
|DE|\geq |D^\prime E^\prime|= 2|A_2A_{2,\perp}|\tan\frac{\angle A_1A_2A_3}{2}\geq 2|A_2A_{2,\perp}|\tan\frac{\alpha_{min}}{2}\geq C|A_2A_{2,\perp}|.
\end{equation}
Similar to (\ref{le_eta1}), we construct $\Psi_D(x)$  as follows,
\begin{equation}\label{le2_eta1}
\Psi_D=z-I_{h,T}^{{\rm IFE}} z, \quad z=\left\{
\begin{aligned}
&v\qquad \mbox{ in } T_h^+,\\
&0 \qquad \mbox{ in } T_h^-,
\end{aligned}\right.
\end{equation}
where the function $v$ is linear and satisfies
\begin{equation}\label{le2_v1}
\beta^+\nabla v\cdot\textbf{n}_h=0,\quad v(D)=1,\quad v(E)=0.
\end{equation}
From (\ref{le2_v1}), we have
\begin{equation*}
\begin{aligned}
\left|\nabla v\cdot \textbf{t}_h\right|=\frac{1}{|DE|},\qquad |v(A_2)|=|v(A_{2,\perp})|=|v(E)|+|A_{2,\perp}E|\left|\nabla v\cdot \textbf{t}_h\right|=\frac{|A_{2,\perp}E|}{|DE|}.
\end{aligned}
\end{equation*}
If $A_{2,\perp}\in \overline{T}$, then $|A_{2,\perp}E|\leq |DE|$ and $|v(A_2)|=|v(A_{2,\perp})|\leq 1$.
Otherwise, we have $\angle A_1A_2A_3<\pi/2$. Using (\ref{le2_gama}), we obtain
\begin{equation*}
|v(A_2)|\leq C\frac{|A_{2,\perp}E|}{|A_2A_{2,\perp}|}\leq C\tan(\frac{\pi}{2}-\angle A_1A_2A_3)\leq C\tan (\frac{\pi}{2}-\alpha_{min})\leq C,
\end{equation*}
where we have used the fact  that the line $DE$ cannot be parallel to the line $A_1A_2$.
Hence, we have
\begin{equation*}
\|z\|_{L^\infty(T)}\leq C~\mbox{ and }~\|z\|^2_{L^2(T)}\leq \|z\|^2_{L^\infty(T)}|T|\leq Ch^2.
\end{equation*}
 Using (\ref{le2_eta1}) and (\ref{est_phi}), we have     
 \begin{equation*}
 \|\Psi_D\|^2_{L^2(T)}\leq C(v(A_2)\|\phi_{A_2}\|^2_{L^\infty(T)}|T|+\|z\|^2_{L^2(T)})\leq Ch^2.
 \end{equation*}
 Since $v$ is linear, we know that
 \begin{equation*}
 |z|^2_{H^1(T_h^+\cup T_h^-)}=|\nabla v|^2|T_h^+|\leq \frac{|T_h^+|}{|DE|^2}\leq \frac{|DE||A_2A_{2,\perp}|}{2|DE|^2}\leq C,
 \end{equation*}
 where we have used the equality (\ref{le2_gama}). It follows from  (\ref{le2_eta1}) and (\ref{est_phi}) that
 \begin{equation*}
 |\Psi_D|^2_{H^1(T_h^+\cup T_h^-)}\leq C(v(A_2)|\phi_{A_2}|^2_{W^1_\infty(T)}|T|  +|z|^2_{H^1(T_h^+\cup T_h^-)})\leq C.
 \end{equation*}
The upper bound estimate for $\Psi_E$ is analogous. \qed
\end{proof}

\section{Technical results for the 3D cases}\label{appen3D}

\subsection{Proof of Lemma~\ref{lem_unique3D}}\label{pro_lem_unique3D}
\begin{proof} $\!\!:$
Similar to (\ref{app1_basis}) for the 2D cases,  we also have
\begin{equation}\label{basis_3D}
\phi=I_{h,T}\phi+\frac{(\beta^-/\beta^+-1)\nabla I_{h,T}\phi\cdot{\mathbf{n}_h}}{1+(\beta^-/\beta^+-1)\nabla I_{h,T}w\cdot\mathbf{n}_h}(w-I_{h,T}w),\qquad \forall \phi \in  S_h(T),
\end{equation}
where  $T_h^+$ and $T_h^-$ in the definition of $w$ in (\ref{def_w2D}) are replaced by $T^+$ and $T^-$,  that is,
\begin{equation}\label{def_w3D}
w(x)=\left\{
\begin{aligned}
&w^+(x)=d_{\Gamma_{h,T}^{ext}}(x)~~ &\mbox{ in } T^+,\\
&w^-(x)=0  &\mbox{ in } T^-.
\end{aligned}\right.
\end{equation}
It suffices to prove the following relation: 
\begin{equation}\label{app1_013D}
0\leq \nabla I_{h,T}w\cdot\mathbf{n}_h\leq 1.
\end{equation}

There are only two types of interface elements.  Type {I} interface element:  the plane $\Gamma_{h,T}^{ext}$ cuts three edges of the tetrahedron (see Figures~\ref{fig_3D1}); Type {II} interface element:  the plane $\Gamma_{h,T}^{ext}$ cuts four edges of the tetrahedron (see Figures~\ref{fig_3D2}). 

\begin{figure} [htbp]
\centering
\subfigure{ 
\includegraphics[width=0.33\textwidth]{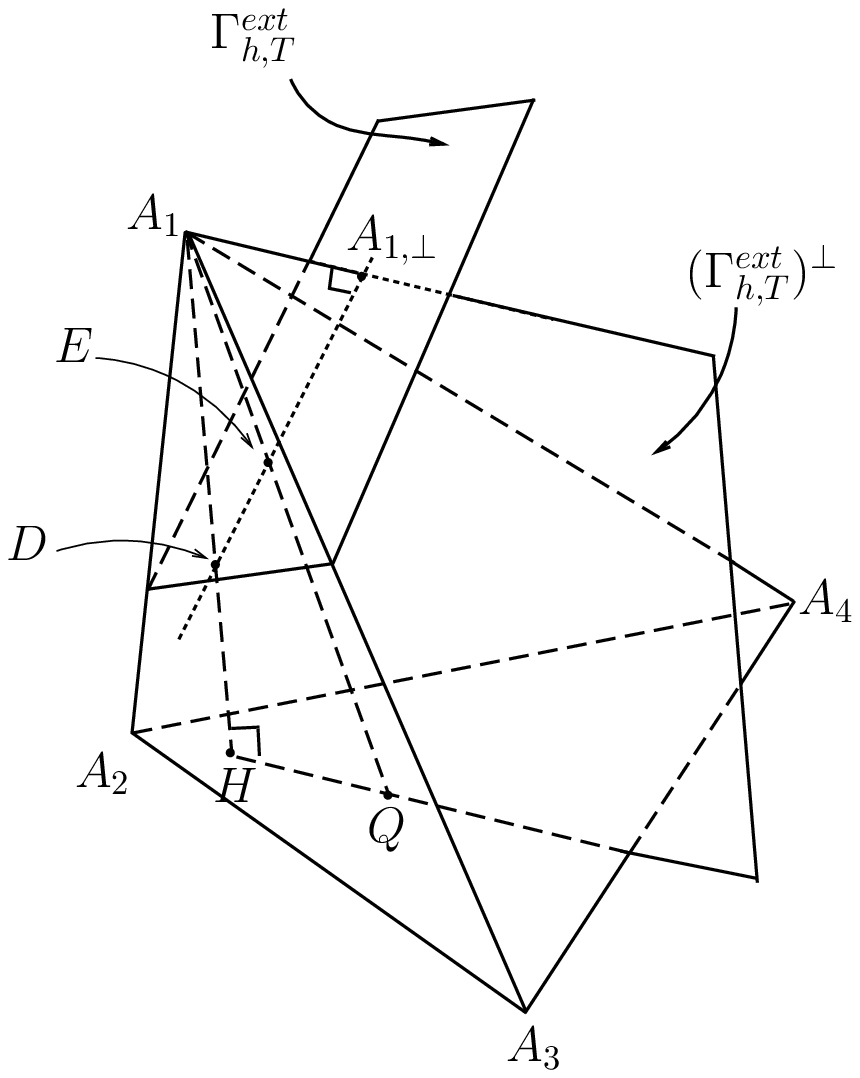}}
\qquad\qquad
\subfigure{
\includegraphics[width=0.15\textwidth]{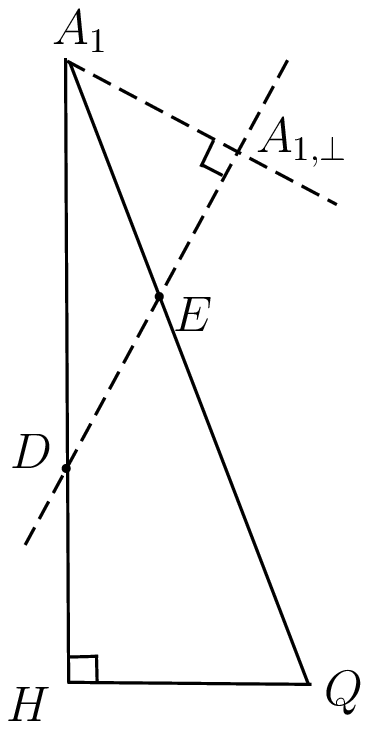}}
 \caption{Type {\rm I} interface element in 3D. The plane $\Gamma_{h,T}^{ext}$ cuts three edges of the element.\label{fig_3D1}} 
\end{figure}

For  Type I interface element, we take the tetrahedron in Figures~\ref{fig_3D1} as an illustration. Similar to the 2D cases, we only need to consider the case $A_1\in T^+$. Let $A_{i,\perp}$ be the orthogonal projection of the point $A_i$ onto the plane $\Gamma_{h,T}^{ext}$. Similar to (\ref{2Dcase1}), we have
\begin{equation}\label{appn_nad1}
\nabla I_{h,T}w\cdot\mathbf{n}_h =1-\lambda_1(A_{1,\perp}),
\end{equation}
where $\lambda_i$ is the standard 3D linear basis function associated with the vertex $A_i$.
Let $H$ be the orthogonal projection of the point $A_1$ onto the plane $A_2A_3A_4$.  The dihedral angle between $A_1A_2A_3$ and $A_4A_2A_3$ is denoted by $A_1\mbox{-}A_2A_3\mbox{-}A_4$. As we assume that the dihedral angles $A_1\mbox{-}A_2A_3\mbox{-}A_4$, $ A_1\mbox{-}A_3A_4\mbox{-}A_2$, $ A_1\mbox{-}A_2A_4\mbox{-}A_3$ are less than or equal to $\pi/2$, the point $H$ must be on the triangle $\triangle A_2A_2A_4$ or its boundary, and there exists a point of intersection  $D$ of the line segment $\overline{A_1H}$ and the plane $\Gamma_{h,T}^{ext}$.  
 Let $(\Gamma_{h,T}^{ext})^\perp$ be a plane that passes through the points $A_1$, $H$ and $A_{1,\perp}$.  Obviously, we can choose a point $Q$, different from $H$, on the line of intersection of the plane $(\Gamma_{h,T}^{ext})^\perp$ and the plane $A_2A_2A_4$ such that  $\Gamma_{h,T}^{ext}\cap \overline{A_1Q}\not=\emptyset$.  
The point of intersection of $\Gamma_{h,T}^{ext}$ and $\overline{A_1Q}$  is denoted by $E$. 
%

 Now we focus on the triangle $\triangle A_1HQ$ (see the right picture in Figure~\ref{fig_3D1}). Let $\tilde{\lambda}_1$ be the standard 2D linear basis function on the triangle $\triangle A_1HQ$ associated with the point $A_1$. Note that  $\lambda_1(H)=\lambda_1(Q)=0$ and $\lambda_1(A_1)=1$,  it holds $\tilde{\lambda}_1(x)=\lambda_1(x)$ on the plane $(\Gamma_{h,T}^{ext})^\perp$.
Since the maximum angle of the triangle $\triangle A_1HQ$ is equal to $\pi/2$, using the result of the 2D cases (see the proof of Lemma~\ref{lem_unique} in Appendix \ref{pro_lem_unique}), we obtain 
\begin{equation*}
\nabla I_{h,T}w\cdot\mathbf{n}_h =1-\tilde{\lambda}_1(A_{1,\perp})\in [0,1].
\end{equation*}

\begin{figure} [htbp]
\centering
\subfigure{ 
\includegraphics[width=0.55\textwidth]{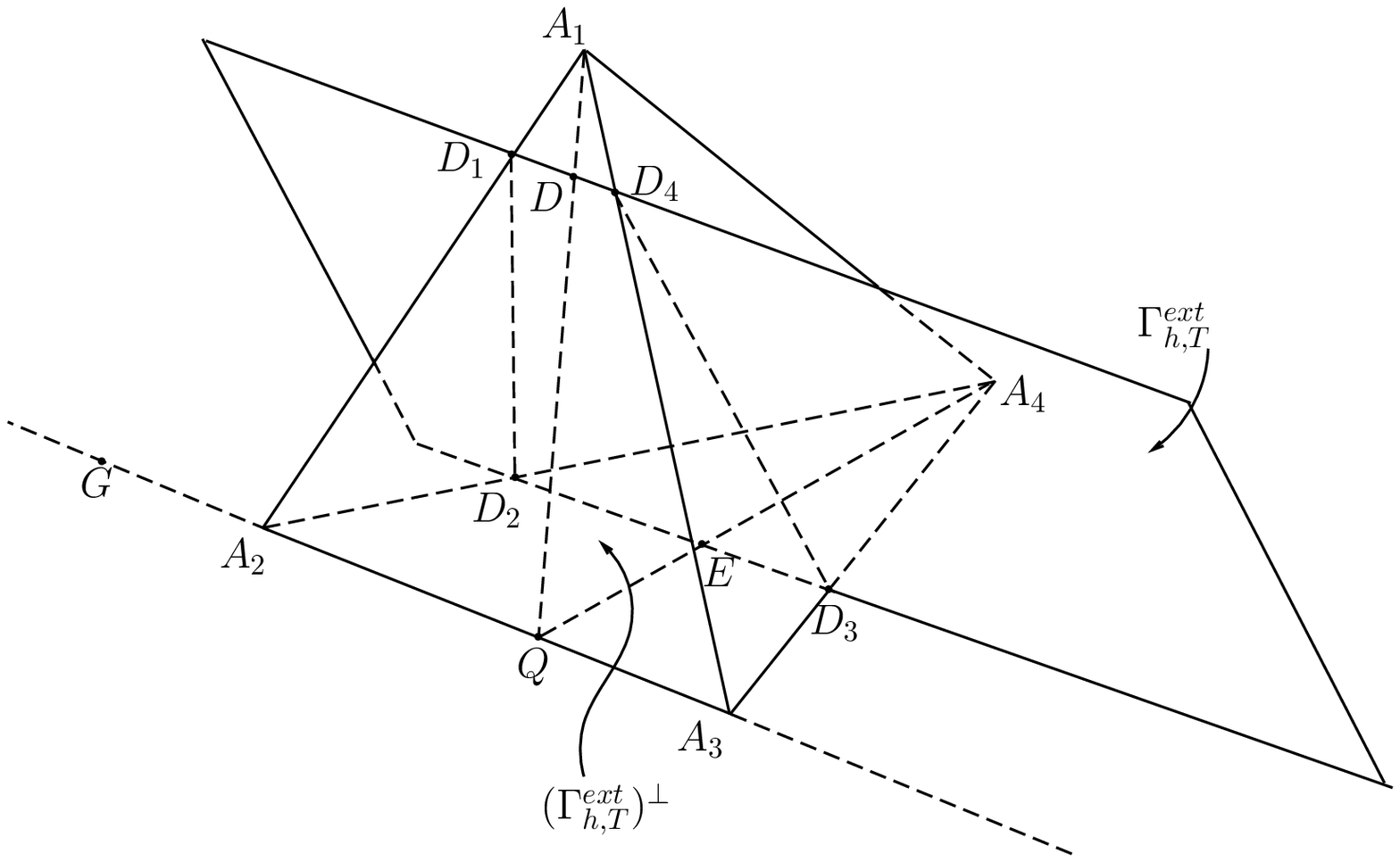}}
\qquad
\subfigure{
\includegraphics[width=0.23\textwidth]{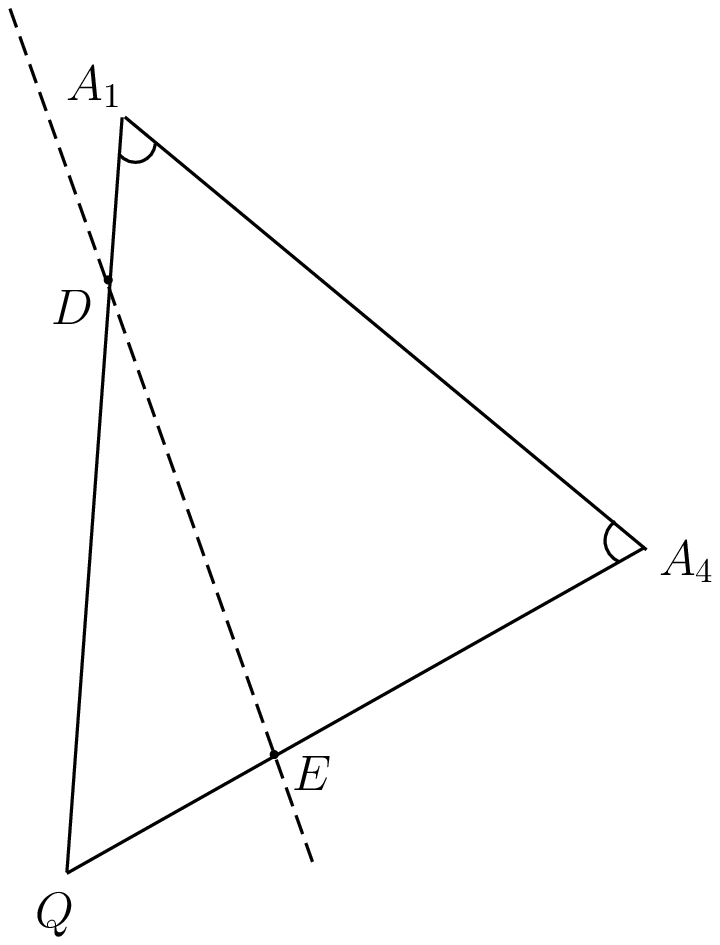}}
 \caption{Type {\rm II} interface element in 3D. The plane $\Gamma_{h,T}^{ext}$ cuts four edges of the element.\label{fig_3D2}} 
\end{figure}

For  Type II interface element, we take the tetrahedron in Figures~\ref{fig_3D2} as an illustration. The plane $\Gamma_{h,\Gamma}^{ext}$ intersects with the edges $\overline{A_1A_2}$, $\overline{A_2A_4}$, $\overline{A_3A_4}$ and $\overline{A_1A_3}$ at the points $D_1$, $D_2$, $D_3$ and $D_4$. In view of the limiting cases,
$$
\begin{cases}
      & \text{$D_4\rightarrow A_1$, $D_3 \rightarrow A_4$}~  (\text{i.e.,}\Gamma_{h,T}^{ext}\rightarrow \text{the plane }A_1A_2A_4), \\
      & \text{$D_2 \rightarrow A_2$, $D_3 \rightarrow A_3$}~  (\text{i.e.,}\Gamma_{h,T}^{ext}\rightarrow \text{the plane }A_1A_2A_3),\\
      &\text{$D_1\rightarrow A_2$, $D_4\rightarrow A_3$}~  (\text{i.e.,}\Gamma_{h,T}^{ext}\rightarrow \text{the plane }A_2A_3A_4),\\
      &\text{$D_1\rightarrow A_1$,  $D_2 \rightarrow A_4$}~  (\text{i.e.,}\Gamma_{h,T}^{ext}\rightarrow \text{the plane }A_1A_3A_4),\\
\end{cases}
$$
the following relation must be true,
\begin{equation*}
 0< D_4\mbox{-}D_1D_2\mbox{-}A_4< \max\{A_3\mbox{-}A_1A_2\mbox{-}A_4, A_3\mbox{-}A_2A_4\mbox{-}A_1, \pi-A_3\mbox{-}A_1A_4\mbox{-}A_2\}. 
\end{equation*}
Together with the condition $\gamma_{max}\leq \pi/2$, we conclude that,   
\begin{equation}\label{pro_ulti_geo3D0}
 0< D_4\mbox{-}D_1D_2\mbox{-}A_4<\pi-A_3\mbox{-}A_1A_4\mbox{-}A_2.
\end{equation}

Without loss of generality, we assume $A_1\in T^+$, so we have 
\begin{equation*}
\nabla I_{h,T}w\cdot\mathbf{n}_h =1-\lambda_1(A_{1,\perp})+(1-\lambda_4(A_{4,\perp})).
\end{equation*}
 Let $(\Gamma_{h,T}^{ext})^\perp$ be the plane that passes through the points $A_1$ and $A_4$ and is perpendicular to the plane $\Gamma_{h,T}^{ext}$.  Let $Q$ be the point of intersection of the plane $(\Gamma_{h,T}^{ext})^\perp$ and line $A_2A_3$. 

Now we focus on the triangle $\triangle A_1QA_4$ (see the right picture in Figure~\ref{fig_3D2}).  Let $\tilde{\lambda}_1$ and $\tilde{\lambda}_4$ be the standard 2D linear basis functions on the triangle $\triangle A_1QA_4$ associated with the points $A_1$ and $A_4$, respectively. Note that  $\lambda_1(A_4)=\lambda_1(Q)=0$ and $\lambda_4(A_1)=\lambda_4(Q)=0$,  we have
$\tilde{\lambda}_1(x)=\lambda_1(x)$ and $\tilde{\lambda}_4(x)=\lambda_4(x)$  on the plane $(\Gamma_{h,T}^{ext})^\perp$. Therefore, it holds
\begin{equation*}
\nabla I_{h,T}w\cdot\mathbf{n}_h =1-\tilde{\lambda}_1(A_{1,\perp})+(1-\tilde{\lambda}_4(A_{4,\perp})),
\end{equation*}
which is the same as the equation (\ref{2Dcase2_ori}) for Case II in the 2D cases if we consider the triangle  $\triangle A_1QA_4$. In order to use the result of the 2D cases,  we need to verify the angle condition of the triangle $\triangle A_1QA_4$.
In view of the relation (\ref{pro_ulti_geo3D0}), we consider two cases: $D_4\mbox{-}D_1D_2\mbox{-}A_4\in (0,\pi/2]$ and $D_4\mbox{-}D_1D_2\mbox{-}A_4\in [\pi/2,\pi-A_3\mbox{-}A_1A_4\mbox{-}A_2)$.
If the dihedral angle $D_4\mbox{-}D_1D_2\mbox{-}A_4\in (0,\pi/2]$, then the point $Q$ is on the ray $\overrightarrow{A_2A_3}$ and  the following relation holds:
\begin{equation}\label{3Dangle_pro1}
0\leq Q\mbox{-}A_1A_4\mbox{-}A_2<\pi/2.
\end{equation}
Note that the existence of the point $Q$ relies on $\alpha_{max}\leq \pi/2$, $\gamma_{max}\leq \pi/2$ and the relation (\ref{3Dangle_pro1}).
%
\begin{figure} [htbp]
\centering
\includegraphics[width=0.35\textwidth]{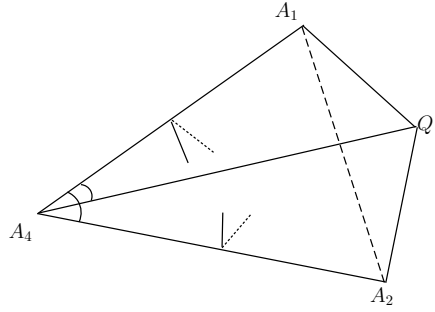}
 \caption{Estimate the angle $\angle QA_4A_1$\label{calri_3D_angle}} 
\end{figure}
By the conditions  $\angle A_2A_4A_1\leq \alpha_{\max}\leq \pi/2$, $Q\mbox{-}A_4A_2\mbox{-}A_1 \leq \gamma_{max}\leq \pi/2$, and the relation $Q\mbox{-}A_1A_4\mbox{-}A_2<\pi/2$ from (\ref{3Dangle_pro1}), it is easy to see that $\angle QA_4A_1\leq \pi/2$ (see Figure~\ref{calri_3D_angle} for clarity).  Analogously, we have $\angle QA_1A_4\leq \pi/2$ since $Q\mbox{-}A_1A_4\mbox{-}A_2<\pi/2$, $Q\mbox{-}A_1A_2\mbox{-}A_4\leq \gamma_{\max}\leq \pi/2$ and $\angle A_2A_1A_4 \leq \alpha_{\max}\leq \pi/2$.  
In view of  the proof of Lemma~\ref{lem_unique} in Appendix \ref{pro_lem_unique} for Case II, by the relations $\angle QA_1A_3\leq \pi/2$ and $\angle QA_4A_1\leq \pi/2$, we obtain the estimate (\ref{app1_013D}). We emphasize that the condition for the angle $\angle A_1QA_4$ is actually unnecessary when the points $D$ and $E$ are on the edges $\overline{QA_1}$ and $\overline{QA_4}$, respectively.

If the dihedral angle  $D_4\mbox{-}D_1D_2\mbox{-}A_4\in [\pi/2,\pi-A_3\mbox{-}A_1A_4\mbox{-}A_2)$, then the point $Q$ is on the ray $\overrightarrow{A_2G}$, and it holds $0\leq Q\mbox{-}A_1A_4\mbox{-}A_2<\pi/2-A_3\mbox{-}A_1A_4\mbox{-}A_2$, 
where the point $G$ is on the line $A_2A_3$  but not on the ray $\overrightarrow{A_2A_3}$ (see Figure~\ref{fig_3D2}). Obviously, we have $Q\mbox{-}A_1A_4\mbox{-}A_3<\pi/2$. Therefore, the proof of $\angle QA_4A_1\leq \pi/2$ and $\angle QA_1A_4\leq \pi/2$ is similar to that of the case $D_4\mbox{-}D_1D_2\mbox{-}A_4\in (0,\pi/2]$.  \qed 
\end{proof}

\subsection{Proof of Lemma~\ref{lem_jiest3D}}\label{pro_lem_jiest3D}
\begin{proof} $\!\!:$
Let $A_i$ be a vertice of the element $T$, and $\phi_{A_i}$ be the corresponding IFE basis function. Similar to the 2D cases, using (\ref{basis_3D})-(\ref{app1_013D}) we  obtain  
\begin{equation*}
|\phi_{A_i}|_{W_\infty^m(T^+\cup T^-)}\leq Ch^{-m},~ m=0,1. 
\end{equation*}
The function $\Psi(x)$ can be constructed explicitly as 
\begin{equation*}
\Psi=z-I_{h,T}^{\rm IFE}z, \quad\quad
z(x)=\left\{
\begin{aligned}
&1\quad &&\mbox{ in }  T^+,\\
&0&&\mbox{ in }  T^-.
\end{aligned} 
\right.
\end{equation*}
Then, we have
\begin{equation*}
\begin{aligned}
|\Psi|_{W^m_\infty(T^+\cup T^-)}&\leq |z|_{W^m_\infty(T^+\cup T^-)}+\|z\|_{L^\infty(T)}\sum_{i}|\phi_{A_i}|_{W^m_\infty(T^+\cup T^-)}\leq Ch^{-m},
\end{aligned}
\end{equation*}
which implies $|\Psi|^2_{H^m(T^+\cup T^-)}\leq Ch^{3-2m}$. The estimates for $\Upsilon$ and $\Theta_i$ can be obtained similarly by constructing these function as 
\begin{equation*}
\Upsilon=z-I_{h,T}^{\rm IFE}z,\quad\quad
z(x)=\left\{
\begin{aligned}
&\frac{1}{\beta^+}(x-x^*)\cdot\mathbf{n}_h\quad &&\mbox{ in }  T^+,\\
&0&&\mbox{ in }  T^-,
\end{aligned} 
\right.
\end{equation*}
and
\begin{equation*}
\Theta_{i}=z-I_{h,T}^{\rm IFE}z,\quad\quad
z(x)=\left\{
\begin{aligned}
&(x-x^*)\cdot\mathbf{t}_{i,h}\quad &&\mbox{ in }  T^+,\\
&0&&\mbox{ in }  T^-.
\end{aligned} 
\right.
\end{equation*} 
\ \qed 
\end{proof}

\subsection{Proof of Lemma~\ref{lem_tiaoyue} for the 3D cases}\label{pro_lem_tiaoyue}

\begin{proof} $\!\!:$
Since $I_h\phi$ is continuous across each face of the triangulation, it holds
\begin{equation}\label{pro_jum1}
\|[\phi]_e\|^2_{L^2(e)}=\|[\phi-I_h\phi]_e\|^2_{L^2(e)}\leq C\sum_{i=1,2}\left\|(\phi-I_h\phi)|_{T_i}\right\|^2_{L^2(e)}.
\end{equation}
It suffices to estimate the term on an element $T$ with  $e$ as its face. By (\ref{basis_3D})-(\ref{app1_013D}), we have
\begin{equation}\label{pro_jum2}
\left\|(\phi-I_{h}\phi)|_{T}\right\|^2_{L^2(e)}\leq Ch^2|e||\nabla I_{h,T}\phi\cdot{\mathbf{n}_h}|^2\leq Ch\|\nabla I_{h,T}\phi\cdot{\mathbf{n}_h}\|_{L^2(T)}^2.
\end{equation}
Using the identity (\ref{basis_3D}) we also have
\begin{equation*}
\nabla I_{h,T}\phi \cdot\mathbf{n}_h=\frac{\left(1+(\beta^-/\beta^+-1)\nabla I_{h,T}w\cdot\mathbf{n}_h\right)(\nabla \phi^\pm\cdot\mathbf{n}_h)}{1+(\beta^-/\beta^+-1)\nabla w^\pm\cdot\mathbf{n}_h}.
\end{equation*}
By the definition of $w$ in (\ref{def_w3D}) and the estimate (\ref{app1_013D}), we get
\begin{equation*}
|\nabla I_{h,T}\phi \cdot\mathbf{n}_h|\leq C|\nabla \phi^+\cdot\mathbf{n}_h| ~~\mbox{ and }~~|\nabla I_{h,T}\phi \cdot\mathbf{n}_h|\leq C|\nabla \phi^-\cdot\mathbf{n}_h|,
\end{equation*}
which leads to
\begin{equation}\label{pro_jum3}
\begin{aligned}
\|\nabla I_{h,T}\phi\cdot{\mathbf{n}_h}\|_{L^2(T)}^2&=|\nabla I_{h,T}\phi\cdot{\mathbf{n}_h}|^2|T^+|+|\nabla I_{h,T}\phi\cdot{\mathbf{n}_h}|^2||T^-|\\
&\leq C|\nabla \phi^+\cdot\mathbf{n}_h|^2|T^+|+C|\nabla \phi^-\cdot\mathbf{n}_h||T^-|\\
&\leq C\|\nabla \phi\|^2_{L^2(T)}.
\end{aligned}
\end{equation}
The desired result (\ref{tiaoyue}) now follows from (\ref{pro_jum1})-(\ref{pro_jum3}). \qed 
\end{proof}

\subsection{Proof of Lemma~\ref{lem_mish3D}}\label{pro_lem_mish3D}

\begin{proof} $\!\!:$
Note that $T^\triangle=(T_h^+\backslash T^+)\cup(T_h^-\backslash T^-)$ is the mis-matched region on $T$. 
For any $\phi\in S_h(T)$, it follows from (\ref{basis_3D})-(\ref{def_w3D}) that for $m=0,1$,
\begin{equation*}
\begin{aligned}
|\phi^+-\phi^-|_{W^m_\infty(T^\triangle)}&=\left|\frac{(\beta^-/\beta^+-1)\nabla I_{h,T}\phi\cdot{\mathbf{n}_h}}{1+(\beta^-/\beta^+-1)\nabla I_{h,T}w\cdot\mathbf{n}_h}\right||d_{\Gamma_{h,T}^{ext}}|_{W^m_\infty(T^\triangle)}\leq Ch^{2-2m}|\nabla I_{h,T} \phi\cdot\mathbf{n}_h|,
\end{aligned}
\end{equation*}
where in the last inequality we have utilized (\ref{app1_013D}) and  the first inequality in (\ref{rela_3d_nnh2}). The first inequality in  (\ref{rela_3d_nnh2}) also implies $|T^\triangle|/|T|\leq Ch$. By the definition of $\widehat{I_{h}^{\rm IFE}}$ in (\ref{def_hatIFE}) and the inequality  (\ref{pro_jum3}) we have
\begin{equation}\label{dis_1}
\begin{aligned}
|\widehat{I_{h}^{\rm IFE}}\phi-\phi|^2_{H^m(T)}&=|\phi^+-\phi^-|^2_{H^m(T^\triangle)}\leq Ch^{4-4m}|\nabla I_{h,T} \phi\cdot\mathbf{n}_h|^2|T|(|T^\triangle|/|T|)\\
&\leq Ch^{5-4m}\|\nabla I_{h,T} \phi\cdot\mathbf{n}_h\|^2_{L^2(T)}\leq Ch^{5-4m}\|\nabla\phi\|^2_{L^2(T)},
\end{aligned}
\end{equation}
and
\begin{equation}\label{dis_2}
\begin{aligned}
|\widehat{I_{h}^{\rm IFE}}\phi-\phi|^2_{L^2(e)}&=|\phi^+-\phi^-|^2_{L^2(e)}\leq Ch^{4}|\nabla I_{h,T} \phi\cdot\mathbf{n}_h|^2|T|(|e|/|T|)\\
&\leq Ch^{3}\|\nabla I_{h,T} \phi\cdot\mathbf{n}_h\|^2_{L^2(T)}\leq Ch^{3}\|\nabla\phi\|^2_{L^2(T)}.
\end{aligned}
\end{equation}
Choosing $\phi=I_{h}^{\rm IFE}v$, we get
\begin{equation*}
\widehat{I_{h}^{\rm IFE}}\phi-\phi=\widehat{I_{h}^{\rm IFE}}(I_{h}^{\rm IFE}v)-I_{h}^{\rm IFE}v=\widehat{I_{h}^{\rm IFE}}v-I_{h}^{\rm IFE}v,
\end{equation*}
which together with (\ref{dis_1}) and (\ref{dis_2}) yields the desired results (\ref{dis}).
\end{proof}

\subsection{Proof of Lemma~\ref{lem_consis_3D}}\label{pro_consis_3D}
\begin{proof} $\!\!:$  By the Cauchy-Schwarz inequality we have
\begin{equation}
\left|\int_{\Gamma}\beta^- \nabla u^-\cdot\mathbf{n} [v_h]_\Gamma ds\right|^2\leq C\|\nabla u^-\cdot\mathbf{n}\|^2_{L^2(\Gamma)}\sum_{T\in\mathcal{T}_h^\Gamma}\|[v_h]_{\Gamma}\|^2_{L^2(\Gamma\cap T)}.
\end{equation}
For any $\phi\in S_h(T)$, it follows from (\ref{basis_3D})-(\ref{def_w3D}) that
\begin{equation*}
\begin{aligned}
\|[\phi]_{\Gamma\cap T}\|_{L^\infty(\Gamma\cap T)}&=\left|\frac{(\beta^-/\beta^+-1)\nabla I_{h,T}\phi\cdot{\mathbf{n}_h}}{1+(\beta^-/\beta^+-1)\nabla I_{h,T}w\cdot\mathbf{n}_h}\right|\|d_{\Gamma_{h,T}^{ext}}\|_{L^\infty(\Gamma\cap T)}\leq Ch^2|\nabla I_{h,T} \phi\cdot\mathbf{n}_h|,
\end{aligned}
\end{equation*}
where in the last inequality we have used (\ref{app1_013D}) and  the first inequality in  (\ref{rela_3d_nnh2}). Using the fact $|\Gamma\cap T|\leq Ch^2$ which can be obtain by applying the interface trace inequality  (see Lemma 3.2  in \cite{2016High}) to a constant function, we further have
\begin{equation*}
\|[\phi]_{\Gamma\cap T}\|^2_{L^2(\Gamma\cap T)}\leq Ch^6|\nabla I_{h,T} \phi\cdot\mathbf{n}_h|^2\leq Ch^{3}\|\nabla I_{h,T} \phi\cdot\mathbf{n}_h\|^2_{L^2(T)}\leq Ch^{3}\| \phi\cdot\mathbf{n}_h\|^2_{L^2(T)},
\end{equation*}
where we have used (\ref{pro_jum3}) in the last inequality. The lemma follows from the above inequalities and the global trace inequality on $\Omega^-$.
\end{proof}

 \end{appendices}

\end{document}